\documentclass[bj,authoryear]{imsart}

\RequirePackage{amsmath,amsfonts,amssymb,amsthm}
\RequirePackage{natbib}
\RequirePackage{graphicx}

\RequirePackage{enumitem}

\usepackage{subcaption}
\usepackage{bm}
\usepackage{mwe}
\usepackage{chngcntr} 
\usepackage{soul} 
\usepackage{totcount} 

\usepackage{bookmark}

\bookmarksetup{
  open,
  numbered,
  depth=5, 
}
\usepackage{imsart} 

\startlocaldefs

\theoremstyle{plain}  
\newtheorem{theorem}{Theorem}[section]
\newtheorem{lemma}{Lemma}[section]
\newtheorem{corollary}{Corollary}[section]
\newtheorem{proposition}{Proposition}[section]
\regtotcounter{condition}

\theoremstyle{remark}
\newtheorem{definition}{Definition}[section]
\newtheorem{remark}{Remark}
\newtheorem{condition}{Condition}

\renewcommand{\le}{\leqslant}
\renewcommand{\ge}{\geqslant}
\renewcommand{\leq}{\leqslant}
\renewcommand{\geq}{\geqslant}

\newcommand{\bswu}{\boldsymbol{W}}
\newcommand{\bswl}{\boldsymbol{w}}

\newcommand{\such}{\, :\,  }
\newcommand{\fdd}{\mathrm{f.d.d.}}

\newcommand{\seminorm}{p}
\newcommand{\FDP}{\text{FDP}}
\newcommand{\FPR}{\text{FPR}}

\newcommand{\simiid}{\stackrel{\mathrm{iid}}\sim}
\newcommand{\dnorm}{\mathcal{N}}

\newcommand{\tauBHm}{\tau_{\mathrm{BH},m}}

\newcommand{\e}{\mathbb{E}}
\newcommand{\var}{\mathrm{var}}
\newcommand{\cov}{\mathrm{cov}}
\newcommand{\real}{\mathbb{R}}

\newcommand{\tran}{\mathsf{T}}

\newcommand{\cf}{\mathcal{F}}

\newcommand{\cy}{\mathcal{Y}}

\newcommand{\giv}{\!\mid\!} 
\newcommand{\err}{\varepsilon}

\newcommand{\dbern}{\mathrm{Bern}}
\newcommand{\dunif}{\mathrm{Unif}}

\newcommand{\paragraphLocal}[1]{%
  
  \par 
  \addvspace{\medskipamount}
  \noindent \textbf{#1\@addpunct{.}}\enspace\ignorespaces
}

\endlocaldefs

\begin{document}

\begin{frontmatter}
\title{A central limit theorem for the Benjamini-Hochberg false discovery proportion under a factor model \thanks{This article will appear in \textit{Bernoulli}. While there are formatting differences, the main text of this arXiv submission should be exactly the same as the version that will appear in \textit{Bernoulli}.}}

\runtitle{A CLT for BH FDP under a factor model}

\begin{aug}
\author[A]{\inits{D. M.}\fnms{Dan M.}~\snm{Kluger}\ead[label=e1]{kluger@stanford.edu}},%\orcid{0000-0002-7865-259X}},
\author[A]{\inits{A. B.}\fnms{Art B.}~\snm{Owen}\ead[label=e2]{owen@stanford.edu}}%\orcid{0000-0001-5860-3945}}
%%%%%%%%%%%%%%%%%%%%%%%%%%%%%%%%%%%%%%%%%%%%%%
%% Addresses                                %%
%%%%%%%%%%%%%%%%%%%%%%%%%%%%%%%%%%%%%%%%%%%%%%
\address[A]{Department of Statistics,
Stanford University, Stanford, CA, USA \printead[presep={.\ }]{e1,e2}}
%\printead{e1,e2}
\end{aug}

\begin{abstract}
The Benjamini-Hochberg (BH) procedure remains widely popular despite having limited theoretical guarantees in the commonly encountered scenario of correlated test statistics. Of particular concern is the possibility that the method could exhibit bursty behavior, meaning that it might typically yield no false discoveries while occasionally yielding both a large number of false discoveries and a false discovery proportion (FDP) that far exceeds its own well controlled mean. In this paper, we investigate which test statistic correlation structures lead to bursty behavior and which ones lead to well controlled FDPs. To this end, we develop a central limit theorem for the FDP in a multiple testing setup where the test statistic correlations can be either short-range or long-range as well as either weak or strong. The theorem and our simulations from a data-driven factor model suggest that the BH procedure exhibits severe burstiness when the test statistics have many strong, long-range correlations, but does not otherwise.

\end{abstract}



\end{frontmatter}

\section{Introduction}

The Benjamini-Hochberg (BH) procedure is a widely used method for balancing Type I and Type II errors when testing many hypotheses simultaneously. The procedure is designed to control the False Discovery Rate (FDR), which is the \textit{expected value} of the proportion of discoveries that are false (FDP), below a user specified threshold (\cite{BH95}). The procedure was originally shown to guarantee FDR control when the test statistics are assumed to be independent, an assumption unlikely to hold in most application settings. The BH procedure was later proven in \cite{BY01} to control the FDR when there are dependent test statistics satisfying the Positive Regression Dependency (PRDS) property. While PRDS is quite restrictive (for example, it does not hold for two-sided hypothesis tests when the test statistics are correlated or when there are negatively correlated test statistics (\cite{fithianLei2020})), under more general conditions simulation studies have found BH to conservatively control the FDR (\cite{FarcomeniFDRSim}, \cite{KimAndVandeWeilFDRSim}).

While FDR control is important, the motivation for this paper is our concern that FDR control alone can give investigators who use BH false confidence in a low prevalence of false discoveries among their rejected hypothesis. This can happen if the distribution of the FDP has both a wide right tail and a mean that is still below the user specified threshold. As an example, it would be worrisome in the plausible scenario that an investigator is led to believe that roughly 10 percent of their discoveries are false, when in fact, a majority of them are false. To address such concerns, a number of multiple testing procedures have been proposed to control the tail probability that the FDP exceeds a user specified threshold (\cite{KornEtAl}, \cite{Romano_FDP_control}, \cite{RomanoWolf_FDP_control}). \cite{Efron2007} also raised concerns about high variability of FDP due to correlations of the test statistics and proposed an empirical Bayes approach for estimating a dispersion parameter of the test statistics and controlling FDR conditionally on the dispersion parameter. Despite the promise of these methods, the BH procedure remains overwhelmingly popular and the default method of choice for investigators with multiple testing problems. It is therefore important to determine under which conditions are we assured that the distribution of the FDP will be well concentrated about its mean, the FDR, and under which conditions there is a risk that the distribution of the FDP has a wide right tail. Throughout this text we will refer to the former scenario with low variability of the FDP about its mean as the non-bursty regime, and the alarming, latter scenario where occasionally the FDP is much larger than expected as the bursty regime.

While previous simulations in the literature can be used to identify some settings where the BH procedure will exhibit burstiness (see for example, Figure 5 of \cite{frig:kloa:casu:2009} or Figure 1 of \cite{DelattreRoquainWideFig1}), the aim of this paper is to gain a theoretical understanding of when burstiness is a concern for BH.  We identify dependency structures among test statistics that in conjunction with certain proportions of nonnulls make BH prone to delivering bursts of false discoveries. We find other settings where such bursts must be rare. Our results are asymptotic and hold in a two-group mixture model previously studied by \cite{GenoveseWasserman}, \cite{DR16} and \cite{Izmirlian2020}. In that model, independent $\dbern(\pi_1)$ variables define which hypotheses are nonnull, the null $p$-values have the $\dunif(0,1)$ distribution and the nonnull $p$-values have some other distribution in common. 

The BH procedure and the asymptotic distribution of the FDP is well studied for the setting where the test statistics are independent.  \cite{FinnerAndRoters_withnonulls,FinnerAndRoters} study properties of the number of false discoveries under independence both when there are no nonnulls and when the nonnull $p$-values are always $0$. The limiting distribution of the FDP was studied by \cite{GenoveseWasserman} under independence of the test statistics; however, their asymptotic FDP results are derived for the ``plug-in" method (\cite{BH2000_plugin}) rather than the standard BH procedure. To our knowledge, a CLT for the FDP of the BH procedure itself was first explicitly stated in \cite{Neuvial2008}, which uses a functional delta method argument. \cite{Izmirlian2020}, using a CLT for a randomly stopped process, provides a simpler proof of a CLT for the FDP and corrects an error in Neuvial's asymptotic variance formula. These works show that in the two group mixture model with Bernoulli parameter $\pi_1$ and with $m$ hypotheses to test, when using the BH procedure at FDR control level $q$, $\sqrt{m}\big(\FDP - (1-\pi_1) q\big)$ converges as $m \rightarrow \infty$ to a centered Gaussian with variance that depends on $q$, $\pi_1$, and the common nonnull $p$-value distribution.

There are fewer results on the limiting distribution of the FDP for dependent test statistics.  % I think BH is understood here.
%Under dependence of the test statistics, there are few results for the limiting distribution of the FDP for the BH procedure. 
\cite{FarcomeniDependence} derives a CLT for the FDP of the plug-in procedure when the $p$-values are stationary and satisfy some mixing conditions but for brevity omits explicitly stated FDP CLTs for the standard BH procedure. Using the proof methodology of \cite{Neuvial2008}, \cite{DR11} derive a CLT for the FDP of the BH procedure for one sided testing, when the test statistics follow an equicorrelated Gaussian model, with correlation parameter $\rho\to0$ as the number of tests $m\to\infty$. %conducted goes to infinity. 
\cite{DR16} extend this result to settings where the Gaussian test statistics follow arbitrary dependence structures but the average pairwise correlation of the test statistics, and the average 2nd and 4th powers of the pairwise correlations of the test statistics satisfy some constraints.

In \cite{DR16}, CLTs for the FDP are derived under two distinct regimes. In their first regime, the average pairwise correlation among test statistics is strictly greater than $\mathcal{O}(1/m)$ for $m$ tests. Under this regime, the FDP is not $\sqrt{m}$-consistent for the product of the FDR control parameter with the limiting proportion of nulls, 
and the FDP only converges to a Gaussian with scale factors much smaller than $\sqrt{m}$. Their other regime considered has an average correlation among test statistics that is at most $\mathcal{O}(1/m)$. For this regime \cite{DR16} derive a CLT for the FDP with $\sqrt{m}$ scaling, but they require a restrictive assumption which they call ``vanishing-second order", precluding settings where there are short-range correlations of constant order. Examples of test statistic correlation matrices to which \cite{DR16} will not apply include tridiagonal Toeplitz correlation matrices as well as block correlation matrices of constant block size, both of which are simple models of interest for studying multiple testing under dependence.

With the aim of identifying dependency structures for which the investigator should be concerned about bursty behavior of the BH procedure, in this paper, we introduce a model that allows for a combination of long-range and potentially strong dependence among the test statistics via a factor model, along with additional strongly-mixing noise that has rapidly decaying long-range dependence. The model also allows for the proportion of nonnulls among the $m$ hypothesis tests to vary as a function of $m$. Under some regularity conditions on the factor model and on the noise with rapidly decaying long-range dependence, we prove a CLT for the FDP under more general conditions than prior CLTs. 
We also establish a CLT for the False Positive Ratio (FPR), which is the proportion of false discoveries among all tests conducted. The new CLTs hold conditionally on the realized latent variable of the factor model. Applying these new results, we make the following contributions to the literature of asymptotic results for the BH procedure:

\begin{enumerate}

\item 
CLTs of the FDP for simple models, with short-range and constant-order dependency structures, that were not covered by the results of \cite{DR16}. Examples include block correlation structures with fixed block size and banded correlation structures. These results allow for non-stationary test statistics, so they cannot be inferred from theorems in \cite{FarcomeniDependence} either.
    
\item 
Conditional CLTs in settings where the long-range dependency is modeled by a factor model which includes scenarios not covered in either \cite{FarcomeniDependence} or \cite{DR16}.
    
\item 
CLTs for the FPR rather than just the FDP because the FDP limiting behavior is unilluminating when it converges in probability to 1.
    
\item 
CLTs where the expected proportion of nonnulls varies as the number of test statistics grows, allowing for a sparse allocation of nonnulls.
    
\item 
A discussion of the dependency regimes under which the investigator should be concerned about BH having bursty behavior, such as the setting where the number of nonnulls is $o_p(\sqrt{m})$, and the dependency structure contains a factor model component.
\end{enumerate}

 To qualify point 2 above, we note that \cite{DR16} include some CLTs for the $\FDP$ under long-range dependency that our theorems do not.  Ours all have a $\sqrt{m}$ scaling. They include some with a slower than $\sqrt{m}$ scaling. For instance, they get such a CLT under an equicorrelated Gaussian model with correlation $\rho\to0$ but $\sqrt{m}\rho\to\infty$. To clarify point 5 above, we characterize what causes alarming burstiness of the BH FDP when the test statistics are dependent, which to our knowledge has not been analyzed theoretically before. 
 A separate issue is the pathologically low power of BH when the FDR control level $q$ is below a critical threshold \citep{ChiCriticallity}.
This issue was studied in greater generality by \cite{ZhangFanYu_LowPower} using the framework of \cite{StoreyTaylorSiegmund2004}.

The proof of our main theorem builds upon the proof structure seen in \cite{Neuvial2008} and \cite{DR16}. As is done in those works, we derive a functional CLT (FCLT) for the empirical cumulative distribution functions (ECDFs) of the null and nonnull $p$-values, compute the Hadamard derivative of the FDP written as a functional of the two ECDFs, and apply the functional delta method to obtain a CLT for the FDP. While the proofs of previous FDP CLTs in the literature require establishing an FCLT for the $p$-value ECDFs defined on $[0,1]$, in Section \ref{sec:Define_ab_subsection}, we define a focal interval $[a,b] \subset (0,1)$, and our proof demonstrates that merely an FCLT for the ECDFs restricted to $[a,b]$ is needed. Our use of a focal interval allows us to obtain a CLT for the FDP in new settings where the null and nonnull $p$-value ECDFs are poorly behaved asymptotically in either $[0,a)$ or $(b,1]$. To obtain an FCLT when restricting our attention to $[a,b]$, we use an FCLT from \cite{AndrewsAndPollard} for bounded function classes. The regularity conditions in \cite{AndrewsAndPollard} are conducive to obtaining an FCLT of the $p$-value ECDFs in settings where the test statistics follow block or banded correlation structures (see point 1 above). Therefore, in addition to the contributions enumerated above, our use of a focal interval and our use of an FCLT from \cite{AndrewsAndPollard} are contributions to proof methodology for BH asymptotics.

Our results suggest that approaches which estimate and remove the factor model components from the test statistics prior to applying BH can alleviate burstiness issues. A number of such approaches for estimating and removing factor model components in multiple testing settings have been proposed and have shown promise in simulations \citep{frig:kloa:casu:2009,sun2012multiple,FanHanGu,wang:zhao:hast:owen:2017,FanFarmtest}. Our CLT for the BH method itself can be a useful step towards deriving CLTs for methods that first estimate and remove factor model components and subsequently apply BH.

Figure~\ref{fig:Intro_burst4examples} shows simulations of the FDP under some models that we study in this paper.  In each case, there are $25{,}000$ Monte Carlo simulations. The BH procedure is used with $q=0.1$ on test statistics that are $\dnorm(0,1)$ for null hypotheses and $\dnorm(2,1)$ for alternative hypotheses. Each hypothesis is independently null with probability $0.9$ and nonnull otherwise. The models differ in the correlation among test statistics. For the first histogram, $m=22{,}283$ test statistics were sampled with correlations based on a $3$-factor model fit to some Duchenne Muscular Dystrophy data described in Section~\ref{sec:DataDrivenExample}. The second histogram is for the same correlation matrix after dividing the off-diagonal entries by $10$. Next are two block correlation models with blocks of size $100$ and within-block correlations of $0.5$ or $0.05$. To keep $m=22{,}283$ one of the blocks had only $83$ test statistics in it. In all four settings the FDR is seen to be controlled below $0.1$, as desired. The positive False Discovery Rate (pFDR), defined as the expected value of the FDP conditional on there being at least one rejection, does not exceed $0.1$ in these simulations either. Of the four histograms, 
the one with data-driven correlations
 shows a long tailed distribution for FDP that we consider extremely bursty, the two with synthetic block correlations show FDPs that are typically quite close to the target FDR of $q=0.1$, and the one with downscaled data-driven correlations is intermediate. 
\begin{figure}[t]
    \centering
            \begin{subfigure}[b]{190pt}
            \centering
            \includegraphics[width=190pt]{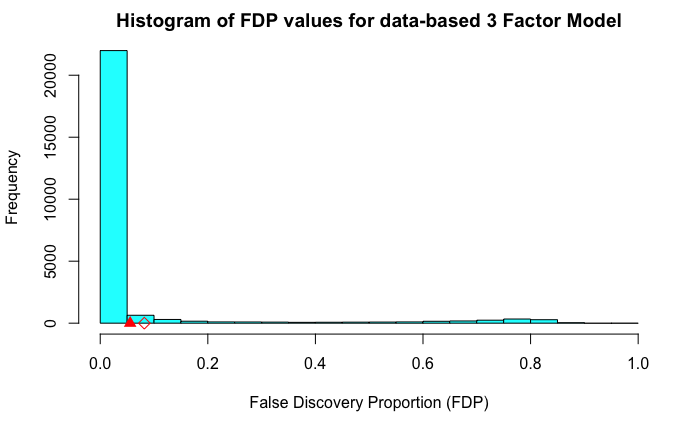}
        \end{subfigure}
        \hfill
        \begin{subfigure}[b]{190pt}
            \centering
            \includegraphics[width=190pt]{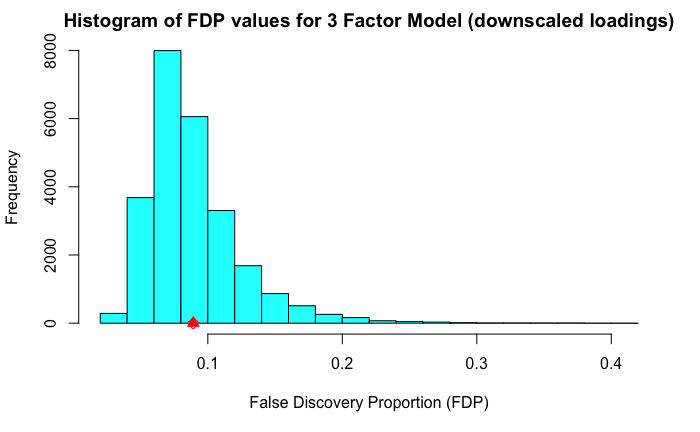}
        \end{subfigure}
         \vskip 1pt 
            \begin{subfigure}[b]{190pt}
            \centering
            \includegraphics[width=190pt]{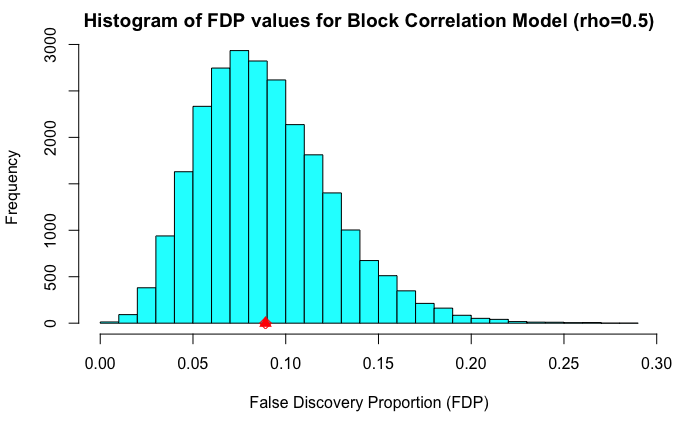}
          \end{subfigure}
        \hfill
        \begin{subfigure}[b]{190pt}
            \centering
            \includegraphics[width=190pt]{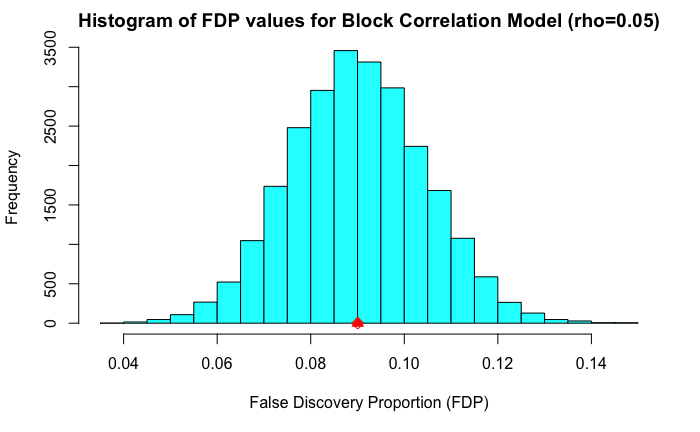}
        \end{subfigure}
          \caption[]
        {\label{fig:Intro_burst4examples}
        These are histograms of the false discovery proportion in $25{,}000$ simulations.  The data come from the two-group mixture model as described in the text.  Each histogram's mean is marked with a triangle and each histogram's mean amongst nonzero $\text{FDP}$ values is marked with a hollow diamond, which estimate the $\text{FDR}$ and $\text{pFDR}$ respectively. For three of the histograms, the triangle and diamond are close enough to overlap. The target FDR control is $q=0.1$.
        } 
\end{figure}

The organization of this paper is as follows. In Section \ref{sec:setup}, we describe our multiple testing setup and our model to account for both long-range correlations and short-range correlations among the test statistics. We also introduce our notation, definitions and the conditions
% I took out `helpful' ... don't want to boast %agreed -Dan
under which our results hold. In Section \ref{sec:StatmentOfTheorems}, we state our most general CLTs for the FDP and FPR of the BH procedure. These hold conditionally on the common latent factors in our model. The proofs of these theorems are provided in the \hyperref[sec:supp]{supplemental material}. In Section \ref{sec:no_factor_model}, we exploit these theorems to obtain FDP CLTs (that are not conditional on a latent factor) for settings where the long-range dependence is rapidly decaying and no factor model component is needed to account for long-range dependencies. In Section \ref{sec:Implications_under_factor_model}, we exploit these theorems to obtain FDP and FPR CLTs conditional on the latent factor that give insight into the burstiness of the BH procedure, and we show that the burstiness of the BH procedure is particularly alarming when the test statistics follow a factor model and the number of nonnulls is sparse (for example, if the number of nonnulls is $o_p(\sqrt{m})$ where $m$ is the number of hypotheses tested). In Section \ref{sec:DataDrivenExample}, we describe the Duchenne Muscular Dystrophy dataset and the 3-factor model that was fit to it, and we show simulations based on the fitted factor model. In Section \ref{sec:Discussion}, we discuss these results and their implications for multiple testing.

\section{Setup and definitions}\label{sec:setup}

In this section we introduce our notation for the two-group mixture model. Our version relies on a factor analysis model that we also introduce. We also review the BH procedure and state our regularity conditions in this section.

\subsection{Two-group mixture model with factors}\label{sec:MultipleTestingSetup}

Our setting has $m$ hypothesis tests indexed by $i=1,\dots,m$ and our asymptotics let $m\to\infty$. In a two-group mixture model, the marginal distribution of each of the $m$ $p$-values is a mixture of a common null distribution and a common nonnull distribution, both of which do not depend on $i$. For $1\le i\le m<\infty$, let $H_{mi}\in\{0,1\}$ be an indicator variable with $H_{mi}=1$ if and only if hypothesis $i$ of $m$ is nonnull. We take $H_{m1},\dots,H_{mm}\stackrel{\mathrm{iid}}{\sim} \dbern( \pi_1^{(m)})$ for $\pi_1^{(m)}\in(0,1)$.  

Our two-group mixture model is thus based on a sequence of nonnull probabilities, and letting $\pi_1^{(m)}\to0$ will let us model sparsity of nonnull hypotheses. For instance, with $\pi_1^{(m)}=\lambda/m$, the number of nonnulls has constant expectation $\lambda$ and has an asymptotic Poisson distribution.

To focus on dependency among tests it is convenient to assume Gaussian test statistics $X_{mi}$ for $1\le i\le m<\infty$. In our two-group mixture model, the test statistics where the null holds have mean zero and the ones where the alternative hypothesis holds have common mean $\mu_A>0$. 
We assume that $X_{mi}=\mu_A H_{mi} +Z_{mi}$ where $(Z_{m1},\dots,Z_{mm})$ is multivariate Gaussian, and we induce dependence among our $p$-values by introducing correlations among the $Z_{mi}$. 
We briefly remark that the common nonnull mean assumption is not necessary for our theoretical results to hold, but is made for cleaner exposition, as our primary interest is in investigating how the dependency between the test statistics can drive bursty behavior.

We study two kinds of dependence operating simultaneously.  One is an $\alpha$-mixing dependence that decays rapidly as the distance between hypothesis indices $i$ increases. This model captures some of the dependence one expects from hypotheses corresponding to a linearly ordered variable such as the position of a single nucleotide polymorphism (SNP) along the genome.

The other form of dependence we include is a factor model. Uses of factor models in multiple hypothesis testing include \cite{frig:kloa:casu:2009}, \cite{luca:kung:chi:2010},
 \cite{sun2012multiple} and
 \cite{gera:step:2020}. A factor model can capture the dependency structure commonly seen among the test statistics in multiple testing problems involving gene expression data because it can capture important aspects of the correlation matrix of gene expression measurements. For example, \cite{owen:2005} gives conditions where the correlation matrix for $m$ test statistics measuring association of a single phenotype with expression levels of $m$ genes is actually equal to the correlation matrix of the sampled gene measurements.

We construct the $k$-factor model for the array $\{ Z_{mi} \such 1 \leq i \leq m < \infty\}$ as follows. We assume that the number of factors $k$ remains fixed as $m \to \infty$, as is assumed in \cite{FanFarmtest}, among others. We let $\bswu \sim \dnorm(0,I_k)$ be the latent factor, which we suppose is only drawn once and does not change as $m \rightarrow \infty$. We let $\{\bm{L}_{mi} \such 1 \leq i \leq m < \infty \}$ be a triangular array of fixed `loading' vectors in $\real^k$. The factor model component of $Z_{mi}$ is $\bm{L}_{mi}^\tran \bswu$. 
For our $\alpha$-mixing model with possibly strong short-range correlations but rapidly diminishing long-range correlations, we let $\{\Sigma^{(m)} \}_{m=1}^{\infty}$ be a sequence of covariance matrices and for each $m$, we let $(\err_{m1},\dots,\err_{mm} ) \sim \dnorm(0,\Sigma^{(m)})$ independently of $\bswu$. To combine both dependency structures, we let $Z_{mi} = \bm{L}_{mi}^\tran \bswu +\err_{mi}$, giving our correlation structure for the array $\{ Z_{mi} \such 1 \leq i \leq m < \infty\}$.

We suppose that all of the test statistics have the same variance and without loss of generality, we take this common variance to be one. We do not assume the factor model to have a perfect fit, and assume instead that $\Vert \bm{L}_{mi} \Vert_2^2+ \Sigma_{ii}^{(m)}=1$ where $\Sigma_{ii}^{(m)} >0$ for all $i,m$. Because $Z_{mi} = \bm{L}_{mi}^\tran \bswu +\err_{mi}$, these assumptions give $Z_{m1},\dots,Z_{mm} \sim \dnorm(0,1)$ along with the two kinds of dependency discussed above.

We let $\varphi$ and $\Phi$ denote the probability density function (PDF) and the cumulative distribution function (CDF), respectively, of $\dnorm(0,1)$ and we let $\bar{\Phi}=1-\Phi$ be the complementary CDF. %We can therefore note that 
Then our $p$-values for one-sided hypothesis tests are \begin{align}\label{eq:defpmi}
P_{mi}=\bar{\Phi}(X_{mi})=\bar{\Phi}\bigl(\mu_A H_{mi}+\bm{L}_{mi}^\tran \bswu+\err_{mi}\bigr)
\end{align}
for $1\le i \leq m<\infty$, and so $P_{mi}\sim \dunif(0,1)$ for the true null hypotheses. Fixing $q \in (0,1)$, throughout the text we will let $\tauBHm$, $V_m$, $\FDP_m$, and $\FPR_m$ denote the rejection threshold, the number of false discoveries, the FDP, and the FPR respectively when applying the Benjamini-Hochberg procedure at level $q$ to the $p$-values $(P_{m1},\dots,P_{mm})$. The formulas for these quantities are given explicitly in the next subsection, where we review the BH procedure. In our main theorems, we state CLTs for the quantities $\FDP_m$ and $\FPR_m$ conditionally on the value of the latent factor $\bswu=\bswl \in \real^k$.

\subsection{The BH procedure}

In this subsection, we describe how the BH procedure is conducted at level $q$ on $m$ tests with $p$-values $(P_{m1}, \dots, P_{mm})$. First take the sorted $p$-values $P_{m(1)} \leq \dots \leq P_{m(m)}$ and set $P_{m(0)}=0$. The number of rejected hypothesis will be given by \begin{equation}
    R_m \equiv \max \bigl\{j \such P_{m(j)} \leq \frac{jq}{m},\ j \in \{0,1, \dots, m\} \bigr\}.
\end{equation} 

The BH procedure rejects the hypotheses that correspond to the $R_m$ smallest $p$-values: that is it will reject all hypothesis $i$ for which $P_{mi} \leq P_{m(R_m)}\equiv \tauBHm$. As noted in \cite{Neuvial2008}, $\tauBHm$ can equivalently be defined as the largest $t \in [0,1]$ at which the empirical CDF (ECDF) of the $p$-values is at least as large as $t/q$. We leverage this equivalence in our theorem proofs.

Letting $H_{m1}, \dots, H_{mm}$ be as defined in Section \ref{sec:MultipleTestingSetup}, the number of false discoveries 
%and the total number of discoveries are 
is 
\begin{equation}
    V_m \equiv \sum_{i=1}^m I \{ P_{mi} \leq \tauBHm, H_{mi}=0 \}.
\end{equation} Then $\FPR_m \equiv V_m/m$ 
and $\FDP_m \equiv V_m/\max\{R_m, 1\}$.
\subsection{Definitions and conditions}\label{sec:def_and_conditions}

Here we present some definitions as well as regularity conditions sufficient for our conditional CLTs to hold. All of the definitions, conditions and formulas in this subsection are conditional on a fixed value of the latent factor $\bswu\in\real^k$. 

\subsubsection{Variance and mixing conditions on \texorpdfstring{$\varepsilon$}{epsilon}}\label{sec:mixing_conditions}

Recall from our setup that $\var(\err_{mi})>0$ for all $m$, $i$, allowing us to define $\tilde{\err}_{mi} \equiv \err_{mi}/\sqrt{\var(\err_{mi})}$ for convenience. It is also helpful to introduce the following condition, which forces the variance of all $\err_{mi}$ terms to be bounded away from zero.

\begin{condition}\label{cond:boundedL}
    $S_L \equiv \sup_{1 \leq i \leq m < \infty} \Vert \bm{L}_{mi} \Vert_2^2 < 1$.
\end{condition}

\paragraphLocal{Note about Condition~\ref{cond:boundedL}:} recalling that  $\Vert \bm{L}_{mi} \Vert_2^2+\var(\err_{mi})=1$ in our model, this condition provides a uniform bound 
$\var(\err_{mi})\ge1-S_L>0$ for all $1\le i\le m<\infty$.  

To describe the mixing condition on $(\err_{mi})_{1 \leq i \leq m < \infty}$, for $1\le i\le m<\infty$, define $\xi_{mi} \equiv (\err_{mi},H_{mi})$. Now let $\mathcal{A}_1^n(m)$  be the $\sigma$-field generated by the variables $\xi_{mi}$ for $1 \leq i \leq n$ and $\mathcal{A}_{n+d}^{\infty}(m)$ be the $\sigma$-field generated by the variables $\xi_{mi}$ for $n+d\le i \le m$. For integers $d\ge1$ our $\alpha$-mixing parameters $\alpha(d)\in[0,1]$ are defined by \begin{equation}\label{eq:def_alpha_mixing_coefficient}
    \alpha(d) \equiv \sup_{n,m \in \mathbb{N}} \sup_{\substack{A_0 \in \mathcal{A}_1^n(m) \\ A_1 \in \mathcal{A}_{n+d}^{\infty}(m)}} \big| P(A_0 \cap A_1) - P(A_0) P(A_1) \big|.
\end{equation}

\begin{condition}\label{cond:niceQandgamma} 
There exists an even integer $Q > 2$ and $\gamma>0$ such that both $$\mathrm{(i)}\quad \frac{\gamma}{2+\gamma}+\frac{2}{Q} < 1\quad\text{and}\quad\mathrm{(ii)}\quad \sum_{d=1}^{\infty} d^{Q-2} \alpha(d)^{\frac{\gamma}{Q+\gamma}} < \infty$$ 
\end{condition}  

\paragraphLocal{Notes about Condition~\ref{cond:niceQandgamma}:} Throughout the text we will let $Q$, $\gamma$ be such numbers. Note that it is possible that this condition can be loosened to allow $Q$ to be rational, but then we need to trust a claim in Andrews and Pollard (1994) that their Theorem 2.2 would still hold for $Q$ not an even integer.  

Condition~\ref{cond:niceQandgamma} will hold when the correlation between $\err_{mi}$ and $\err_{mj}$ is a rapidly decaying function of $\vert i -j \vert$.  If this correlation is always zero for each $|i-j|>M$ (making the error sequences $(\err_{mi})_{1 \leq i \leq m}$ $M$-dependent for each $m$), Condition~\ref{cond:niceQandgamma} will hold. In the following remark, we argue that Condition~\ref{cond:niceQandgamma} will typically hold when $(\tilde{\err}_{mi})_{1 \leq i \leq m}$ is modelled by a stationary ARMA process or by a stationary GARCH process (a definition of these processes can be found in \cite{ARMA_GARCH_texbook}, for example). \newline

\begin{remark}\label{remark:stationaryARMA_fast_enough}
    Suppose that the standardized errors $\tilde{\err}_{mi}$ just depend on $i$ and not on $m$. If $(\tilde{\err}_{i})_{i \in \mathbb{Z}}$ can be modelled by a stationary ARMA model with absolutely continuous errors with respect to Lebesgue measure on $\real$, then Condition~\ref{cond:niceQandgamma} will hold. To see this, note that by Theorem 1 in \cite{MokkademARMA}, such a stationary ARMA process $(\tilde{\err}_{i})_{i \in \mathbb{Z}}$ will be geometrically completely regular and hence the $\alpha$-mixing coefficients of $(\tilde{\err}_i)_{i \in \mathbb{Z}}$ will be $\mathcal{O}(\theta^d)$ for some $\theta \in (0,1)$. By independence of the $H_{mi}$ and since $\var(\err_{i})>0$, this implies that $\alpha(d)=\mathcal{O}(\theta^d)$ will hold for that same $\theta \in (0,1)$, further implying that Condition \ref{cond:niceQandgamma} will hold. By similar reasoning, if $(\tilde{\err}_{i})_{i \in \mathbb{Z}}$ is modeled by a stationary GARCH process, Theorem 8 in \cite{Lindner2009} implies that under certain conditions on the GARCH process errors, Condition~\ref{cond:niceQandgamma} will hold.
\end{remark}

\subsubsection{Definitions of some subdistributions of \texorpdfstring{$p$}{p}-values and their condition}

For any positive integer $m$, define $\pi_0^{(m)} \equiv 1-\pi_1^{(m)}$, and then write 
$$H_{mi0} \equiv 1-H_{mi} \quad \text{and} \quad H_{mi1} \equiv H_{mi}.$$
Our subsequent definitions use $r=0$ for quantities based on the null hypotheses and $r=1$ for quantities from the nonnull hypotheses.
For $t \in [0,1]$ and $r \in \{0,1 \}$ let $$\hat{F}_{m,r}(t) \equiv \frac{1}{m} \sum_{i=1}^m H_{mir} I \{ P_{mi} \leq t \}= \frac{1}{m} \sum_{i=1}^m H_{mir} I \{ \bar{\Phi}(\mu_A r + \err_{mi}+\bm{L}_{mi}^\tran  \bswu) \leq t \}.$$ 
We call $\hat{F}_{m,0}$ and $\hat{F}_{m,1}$ the empirical subdistribution functions of the null and nonnull $p$-values respectively. These empirical subdistribution functions sum to the ECDF of the $p$-values. Let $\gamma_{mir}: [0,1] \rightarrow [0,1]$ be the monotone increasing bijection given by
$$\gamma_{mir}(t) \equiv \Pr(P_{mi} \leq t \giv H_{mi}=r,\bswu=\bswl) = \bar{\Phi} \Big( \frac{\bar{\Phi}^{-1}(t)- \mu_A r -\bm{L}_{mi}^\tran \bswl  }{\sqrt{1- \Vert \bm{L}_{mi} \Vert_2^2 }} \Big).$$ We aggregate $\gamma_{mir}$ in the following subdistribution functions
$$F_{m,r}(t) \equiv \e\big( \hat{F}_{m,r}(t) \giv \bswu =\bswl \big)  = \frac{\pi_r^{(m)}}{m} \sum_{i=1}^m \gamma_{mir}(t )$$ and then let
\begin{align}\label{eq:defFr}
F_r(t)  \equiv  \lim_{m \rightarrow \infty} F_{m,r}(t) =  \lim_{m \rightarrow \infty} \frac{\pi_r^{(m)}}{m} \sum_{i=1}^m \gamma_{mir}(t).
\end{align}
Condition~\ref{cond:F0F1_defined} ensures that these quantities are well defined.  

\begin{condition}\label{cond:F0F1_defined}
    For all $t \in [0,1]$ and $r \in \{0,1 \}$, $F_r(t) \equiv \lim_{m \rightarrow \infty} F_{m,r}(t)$ exists. 
\end{condition} 

\subsubsection{Defining the asymptotic ECDF and the Simes point}

Now define $\hat{G}_m, G: [0,1] \rightarrow [0,1]$ via
\begin{align}\label{eq:defghatandg}
\hat{G}_m(t) \equiv \hat{F}_{m,0}(t)+\hat{F}_{m,1}(t) \quad\text{and}\quad G(t) \equiv F_0(t)+F_1(t )
\end{align}
for $t \in [0,1]$. Note that $\hat{G}_m$ is the ECDF of the $p$-values and $G$ is the limiting expected ECDF of the $p$-values. Throughout the text we will refer to $G$ as the \textit{asymptotic ECDF} because under most dependency structures, we expect $G$ to be the point-wise limit in probability of $\hat{G}_m$.

The rejection threshold for the BH procedure at level $q$ is the largest point $t$ such that the ECDF of the $p$-values evaluated at $t$ lies above the line through the origin of slope $1/q$, called the Simes line. 
It is reasonable to expect the limiting $p$-value rejection threshold for the BH procedure at level $q$ to be the largest $t$ at which $(t,G(t))$ intersects the Simes line. 
We use the term \textit{Simes point} to describe the largest point where the asymptotic ECDF intersects the Simes line. More precisely, the Simes point is
\begin{equation}\label{eq:Simes_point_def}
    \tau_* \equiv \sup \bigl\{ t \in (0,1) \such G(t) \geq t/q \bigr\},
\end{equation}
interpreting the supremum of the empty set to be zero.
The Simes point satisfies $0\leq\tau_*\leq q$. The upper limit follows from $G(t)\leq1$. Both $G$ and $\tau_*$ depend on the specific realization of latent factor $\bswu \in \real^k$ on which we condition.

\begin{figure}[t]
\centering
\includegraphics[width=0.95 \hsize]{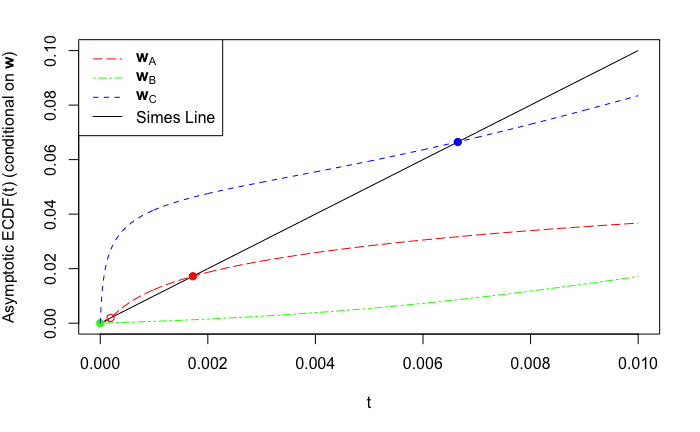}
\caption{\label{fig:SimesLinePoint_example} 
The curves are conditional asymptotic ECDFs of $p$-values in a 3-factor model based on some Duchenne Muscular Dystrophy data described in Section \ref{sec:DataDrivenExample}.
The three draws satisfy $\Phi(\bswl_A)=(0.8,0.4,0.9)$,
$\Phi(\bswl_B)=(0.45,0.56,0.62)$
and $\Phi(\bswl_C)=(0.02,0.85,0.78)$.
Filled circles show the Simes points.
An open circle for $\bswl_A$ shows a crossing of the Simes line that is not the Simes point because it is not the final crossing.  }
\end{figure}

Figure~\ref{fig:SimesLinePoint_example} illustrates the Simes points.  The setting has $\mu_A=2$, $\pi_0=0.9$ and $q=0.1$.  The horizontal axis has putative $p$-values over the range $t\in[0,0.01]$. The Simes line is $t/q$. There are $m=22{,}283$ hypotheses corresponding to genes in the GDS 3027 Duchenne Muscular Dystrophy data described in Section~\ref{sec:DataDrivenExample}. For three draws $\bswu\sim\dnorm(0,I_3)$ we show the asymptotic ECDF curves. One of them crosses the Simes line twice and the Simes point is the last crossing. One crosses it only once and one has Simes point $\tau_*=0$ because the Simes line is never crossed.

We will need continuity of 
$G(\cdot)$ on $(0,1)$ under Conditions~\ref{cond:boundedL} and \ref{cond:F0F1_defined}.
We do not know whether $G$ must be continuous at $0$ or $1$, but our results do not depend on that.

\begin{proposition}\label{prop:G_is_continuous}
    Under Conditions~\ref{cond:boundedL} and \ref{cond:F0F1_defined}, $G$ is continuous on $(0,1)$.
\end{proposition}

\begin{proof}
    It is sufficient to show that $G$ is Lipschitz continuous on $(\epsilon,1-\epsilon)$ whenever $0<\epsilon<1/2$. For any such $\epsilon$, observe that for $r \in \{0,1 \}$ and integers $1\leq i\leq m<\infty$
    \begin{equation*}
    \gamma_{mir}'(t)=\frac{1}{\varphi(\Phi^{-1}(t)) \sqrt{1- \Vert \bm{L}_{mi} \Vert_2^2}}\,  \varphi \bigg( \frac{\bar{\Phi}^{-1}(t) -\bm{L}_{mi}^\tran \bswl -\mu_Ar  }{\sqrt{1- \Vert \bm{L}_{mi} \Vert_2^2 }} \bigg).
    \end{equation*} 
    Now $\varphi(\cdot)\le1/\sqrt{2\pi}$ and then using
    Condition~\ref{cond:boundedL} 
    it follows that for any $t \in (\epsilon,1-\epsilon)$, $$\vert \gamma_{mir}'(t) \vert \leq \frac{1/\sqrt{2\pi}}{\varphi(\Phi^{-1}(t)) \sqrt{1- S_L}} \leq \frac{1/\sqrt{2\pi}}{\varphi(\Phi^{-1}(\epsilon))\sqrt{1-S_L}} \equiv C_{\epsilon}.$$ 
    Since $\sup_{t \in (\epsilon,1-\epsilon)} \vert \gamma_{mir}'(t) \vert \leq C_{\epsilon} < \infty$, $\vert \gamma_{mir}(t)-\gamma_{mir}(s) \vert \leq C_{\epsilon} \vert t-s\vert$ for any $t,s \in (\epsilon,1-\epsilon)$. 
    This argument holds for any $1 \leq i \leq m < \infty$ and $r \in \{0,1\}$, and so for any $t,s \in (\epsilon,1-\epsilon)$ and integer $m$ and $r \in \{0,1 \}$, \begin{equation*}
        \vert F_{m,r}(t)- F_{m,r}(s) \vert \leq \frac{\pi_r^{(m)}}{m} \sum_{i=1}^m \vert \gamma_{mir}(t)-\gamma_{mir}(s) \vert  \leq C_{\epsilon} \vert t-s\vert.
    \end{equation*} Taking the limit as $m \rightarrow \infty$ of the left side of the above inequality, which exists by Condition~\ref{cond:F0F1_defined}, we get $\vert F_{r}(t)- F_{r}(s) \vert \leq C_{\epsilon} \vert t-s\vert$ for all $t,s \in (\epsilon,1-\epsilon)$ and for both $r \in \{0,1 \}$. Thus, $F_r$ is Lipschitz continuous on $(\epsilon,1-\epsilon)$ for $r \in \{0,1\}$ which implies $G=F_0+F_1$ is Lipschitz continuous on $(\epsilon,1-\epsilon)$.
\end{proof}
 
\subsubsection{Defining a focal interval \texorpdfstring{$[a,b]\subset[0,1]$}{[a,b]} for our processes}\label{sec:Define_ab_subsection}

We are going to work with an interval $[a,b]$ of positive length for which the Simes point $\tau_*$ is the unique element $t\in(a,b)$ with $G(t)=t/q$. First we need a technical condition to rule out some pathological behavior. Under this condition there will exist such an interval $[a,b]$.
 
\begin{condition}\label{cond:noaccum}
The Simes point is positive, is the largest point where $G$ actually crosses the Simes line, and is not an accumulation point for points of intersection of $G$ and the Simes line. 
That is,
    \begin{enumerate} 
        \item[(i)] $\tau_*>0$,
        \item[(ii)] $\tau_*=\sup \bigl\{ t \in (0,1) \such G(t) > t/q \bigr\}$, and
        \item[(iii)]  $\tau_*$ is not an accumulation point of $\bigl\{ t \in (0,1) \such G(t ) = t/q \bigr\}$. 
    \end{enumerate}
\end{condition}

\paragraphLocal{Note about Condition~\ref{cond:noaccum}:}         
For many factor model choices, Condition~\ref{cond:noaccum} will hold with some probability in $(0,1)$ depending on the specific realization of $\bswu \sim \dnorm(0,I_k)$. This is due to the dependence of $\tau_*$ and $G$ on the specific realization of the latent factor $\bswu \in \real^k$ on which we condition. For example, see Figure \ref{fig:SimesLinePoint_example}.

\begin{proposition}\label{prop:existance_ab}

If Conditions \ref{cond:boundedL}, \ref{cond:F0F1_defined}, and \ref{cond:noaccum} hold, then for any $b\in(q,1)$ there exists a point $a \in (0,q)$ such that

\begin{enumerate}
\item[(i)] $G(a) > a/q$, and
\item[(ii)]  the Simes point $\tau_*$ is the unique $t \in (a,b)$ solving $G(t) = t/q$.
\end{enumerate}
\end{proposition}
\begin{proof}
Pick any $b \in (q,1)$ and suppose that Conditions  \ref{cond:boundedL}, \ref{cond:F0F1_defined}, and \ref{cond:noaccum} hold. By Condition \ref{cond:noaccum},
$$\tau_* \equiv \sup \bigl\{ t \in (0,1) \such G(t) \geq t/q \bigr\}= \sup \bigl\{ t \in (0,1) \such G(t) > t/q \bigr\} >0.$$
Now $\tau_* \leq q$ because $G(t) \leq 1$ for all $t \in (0,1)$. Also, $G(\tau_*)=\tau_*/q$ by continuity of $G$ (see Proposition \ref{prop:G_is_continuous}). Hence, because $\tau_* = \sup \{ t \in (0,1) \such G(t) > t/q \}$ but $G(\tau_*)=\tau_*/q$, there exists a sequence $a_n \uparrow \tau_*$ such that for all $n$, $G(a_n) > a_n/q $ and $a_n< \tau_*$. Since $\{ t \in (0,1) \such G(t) = t/q \}$ does not have an accumulation point at $\tau_*$ and since $G(t)-t/q$ is continuous, there is a sufficiently large $N_*$ with $G(t)>t/q$ for all $t \in [a_{N_*},\tau_*)$. We choose $a=a_{N_*}\in(0,\tau_*)$ and then property (i) holds by our definition of $a_n$. Also, $a\in(0,q)$ because $a<\tau_*\le q$. Turning to property (ii), $G(t)>t/q$ for all $t\in[a,\tau_*)$ by the the choice of $N_*$ and $a$, while for all $t \in (\tau_*,b)$, $G(t) < t/q $ by the definition of $\tau_*$.
\end{proof}

Throughout the text, when conditioning on $\bswu=\bswl \in \real^k$, if Conditions \ref{cond:boundedL}, \ref{cond:F0F1_defined}, and \ref{cond:noaccum} hold, we will let $[a,b]$ be an interval satisfying the properties (i) and (ii) with $a \in (0,q)$ and $b \in (q,1)$ that are guaranteed by Proposition \ref{prop:existance_ab}. 

\subsubsection{Defining our stochastic process and its Gaussian process limit}{\label{sec:Defining_GP_of_interest}}

Our stochastic processes of interest are two jointly distributed random càdlàg functions on $[a,b]$.
We will show convergence to a pair of Gaussian processes with continuous sample paths on $[a,b]$. The expressions $C([a,b]\times\{0,1\})$ and $(C[a,b])^2$ are both awkward, while $C[a,b]^2$ denotes functions on a square region. Therefore, we use the symbol $[a,b]_2$ to denote $[a,b]\times \{0,1\}$ and study random elements in $C[a,b]_2$ and $D[a,b]_2$. Explicitly, $C[a,b]_2$ is the collection of all pairs of real valued continuous functions on $[a,b]$ while $D[a,b]_2$ is the collection of all pairs of real valued càdlàg functions 
on $[a,b]$.
We study
the following processes in $D[a,b]_2$:
\begin{align}\label{eq:defws}
\begin{split}
W_{m,r}(t) &\equiv \sqrt{m} \big( \hat{F}_{m,r}(t) -F_{m,r}(t) \big) \quad\text{and}\quad \\
\hat{W}_{m,r}(t) &\equiv \sqrt{m} \big( \hat{F}_{m,r}(t) -F_{r}(t) \big).
\end{split}
\end{align}

We are ultimately interested in a functional central limit theorem (FCLT) for the joint process $\big(\hat{W}_{m,0}(\cdot),\hat{W}_{m,1}(\cdot)\big)$, so we must find the limiting joint covariance kernel of this pair of processes. To describe this limiting covariance kernel, we introduce some convenient definitions and notation.

For convenience, throughout the text we will define $\{ \Gamma^{(m)} \}_{m=1}^{\infty}$ to be the sequence of correlation matrices corresponding to $\{\Sigma^{(m)} \}_{m=1}^{\infty}$ and, as before, for each $m,i$ define $\tilde{\err}_{mi} \equiv \err_{mi}/\sqrt{\var(\epsilon_{mi})}=\err_{mi}/\sqrt{1-\Vert \bm{L}_{mi} \Vert_2^2}$. Note that $(\tilde{\err}_{m1},\dots,\tilde{\err}_{mm} ) \sim \dnorm(0,\Gamma^{(m)})$ and that each $\tilde{\err}_{mi}$ has unit variance. For any $t,s \in [0,1]$ and $|\rho|\le1$, define
\begin{align}\label{eq:rhotilde} 
\tilde{\rho}(t,s, \rho) \equiv \Pr\biggl( \tilde{\err}_1 \geq \bar{\Phi}^{-1}(t), \tilde{\err}_2 \geq \bar{\Phi}^{-1}(s)  \,\Bigm|\, 
\Bigl(\begin{matrix} \tilde{\err}_1 \\ \tilde{\err}_2  \end{matrix} \Bigr)
\sim \dnorm\Big( \bm{0},  \Bigl(\,\begin{matrix} 1 & \rho \\ \rho & 1  \end{matrix}\Bigr)  \Big) \biggr)-ts.
\end{align}
Given a bivariate Gaussian with unit variance and correlation $\rho$, the above quantity is the covariance between the indicator that the first coordinate of this bivariate Gaussian exceeds its $1-t$ quantile and the indicator that the second coordinate of this bivariate Gaussian exceeds its $1-s$ quantile.

Now for any $s,t\in(a,b)$ and $r_0,r_1\in\{0,1\}$ and $m \in \mathbb{N}_+$ define 
$$c_m^{(r_0,r_1)} (t,s) \equiv  \cov \big(W_{m,r_0}(t) , W_{m,r_1}(s) \big).$$
It is convenient to break up the expression of $c_m^{(r_0,r_1)} (t,s)$ into two terms. Define $$c_{m,\text{diag}}^{(r_0,r_1)}(t,s) \equiv \frac{1}{m} \sum_{i=1}^m \Big( \pi_{r_0}^{(m)} \gamma_{mir_0}(t \wedge s ) I \{r_0=r_1 \} - \pi_{r_0}^{(m)} \pi_{r_1}^{(m)} \gamma_{mir_0}(t)  \gamma_{mir_1}(s) \Big),$$ and define $$c_{m,\text{cross}}^{(r_0,r_1)}(t,s) \equiv \frac{\pi_{r_0}^{(m)} \pi_{r_1}^{(m)} }{m} \sum_{i \neq j} \tilde{\rho} \Big( \gamma_{mir_0}(t) ,\gamma_{mjr_1}(s) , \Gamma^{(m)}_{ij} \Big).$$ 
In the following proposition we show that $c_m=c_{m,\text{diag}}+c_{m,\text{cross}}$. 

\begin{proposition}\label{prop:splitcov}
    For any $s,t\in[a,b]$ and $r_0,r_1\in\{0,1\}$ and $m\ge2$
    %$(t,r_0) \in [a,b]_2$ and $(s,r_1) \in [a,b]_2$, 
    $$c_{m}^{(r_0,r_1)}(t,s)=c_{m,\mathrm{diag}}^{(r_0,r_1)}(t,s)+c_{m,\mathrm{cross}}^{(r_0,r_1)}(t,s).$$
\end{proposition}

\begin{proof}
    For $i,j\in[m]$ define 
    \begin{align*} C_{i,j,m}^{(r_0,r_1)}(t,s) 
    &\equiv  \cov \big(H_{mir_0} I \{P_{mi} \leq t \}, H_{mjr_1} I \{P_{mj} \leq s \} \big)  \\  
    & =\cov\big(H_{mir_0} I \{ \bar{\Phi}(\tilde{\err}_{mi}) \leq \gamma_{mir_0}(t)  \},H_{mjr_1}  I \{ \bar{\Phi}( \tilde{\err}_{mj}) \leq \gamma_{mjr_1}(s)  \} \big).  \end{align*}
    For $i \neq j$, $H_{mir_0}$, $H_{mjr_1}$ and $\tilde{\err}$ are all independent, so 
    \begin{align*} 
    C_{i,j,m}^{(r_0,r_1)}(t,s) 
    &=  \pi_{r_0}^{(m)} \pi_{r_1}^{(m)}  \cov \big( I \{ \bar{\Phi}(\tilde{\err}_{mi}) \leq  \gamma_{mir_0}(t)   \}, I \{  \bar{\Phi} (\tilde{\err}_{mj}) \leq  \gamma_{mjr_1}(s)   \} \big) \\ 
    &= \pi_{r_0}^{(m)} \pi_{r_1}^{(m)} \tilde{\rho} \big( \gamma_{mir_0}(t) ,\gamma_{mjr_1}(s) , \Gamma^{(m)}_{ij} \big).   \end{align*}
    When $i=j$, $H_{mir_0} H_{mjr_1}=H_{mir_0} I\{r_0=r_1 \}$ and $\tilde{\epsilon}_{mi}=\tilde{\epsilon}_{mj}$, so 
    that
    $$\begin{aligned} C_{i,j,m}^{(r_0,r_1)}(t,s) = & \pi_{r_0}^{(m)} I \{r_0=r_1 \} \gamma_{mir_0}(t \wedge s) - \pi_{r_0}^{(m)}\pi_{r_1}^{(m)} \gamma_{mir_0}(t) \gamma_{mir_1}(s).  \end{aligned}$$ 
    Since the above expressions hold for any $i,j \in [m]$,
        \begin{align*} 
        c_m^{(r_0,r_1)} (t,s)  &= \cov \big(W_{m,r_0}(t) , W_{m,r_1}(s) \big) \\ 
     &= m \cov \big(\hat{F}_{m,r_0}(t) , \hat{F}_{m,r_1}(s) \big) \\  
     &= \frac{1}{m} \sum_{i=1}^m \sum_{j=1}^m \cov \big(H_{mir_0} I \{P_{mi} \leq t \}, H_{mjr_1} I \{P_{mj} \leq s \} \big)  \\ 
      &= \frac{1}{m} \sum_{i=1}^m \sum_{j=1}^m C_{i,j,m}^{(r_0,r_1)}(t,s) \\  
      &= \frac{1}{m} \sum_{i=1}^m  C_{i,i,m}^{(r_0,r_1)}(t,s) +\frac{1}{m} \sum_{i \neq j} C_{i,j,m}^{(r_0,r_1)}(t,s) \\ 
       &= c_{m,\mathrm{diag}}^{(r_0,r_1)}(t,s)+c_{m,\mathrm{cross}}^{(r_0,r_1)}(t,s). \qedhere
       \end{align*}
    
\end{proof}

Now 
define  
\begin{align}\label{eq:defc}
c^{(r_0,r_1)} (t,s) \equiv \lim_{m \rightarrow \infty} c_m^{(r_0,r_1)}(t,s).
\end{align}
By the simplified formula for $c_m^{(r_0,r_1)}(t,s)$, the above limit exists by Condition~\ref{cond:F0F1_defined} if we further impose Condition~\ref{cond:cov_kernel_converges} below. 

\begin{condition}\label{cond:cov_kernel_converges}
For any $(s,t)\in[a,b]$ and $r_0,r_1\in\{0,1\}$,
 \begin{align*}
   &\lim_{m \rightarrow \infty} \frac{\pi_{r_0}^{(m)} \pi_{r_1}^{(m)}}{m} \sum_{i=1}^m \gamma_{mir_0}(t) \gamma_{mir_1}(s)\quad\text{and}\\ 
   &\lim_{m \rightarrow \infty} \frac{\pi_{r_0}^{(m)} \pi_{r_1}^{(m)}}{m} \sum_{i \neq j} \tilde{\rho}( \gamma_{mir_0}(t),\gamma_{mir_1}(s), \Gamma^{(m)}_{ij} )  
   \end{align*}
   both exist.
\end{condition}

It is easy to see that the function $c(\cdot,\cdot)$ 
defined above gives a joint covariance kernel that is symmetric and positive semidefinite because it is the limit of symmetric and positive semidefinite joint covariance kernels $c_m$. \subsubsection{Regularity conditions on \texorpdfstring{$F_0$, $F_1$, $F_{m,0}$ and $F_{m,1}$}{ the limit of the subdistributions}} 

Before introducing the main theorems, we introduce another two conditions that will be used in their proof. 

\begin{condition}\label{cond:F0F1_differfentiable}
    Both $F_0$ and $F_1$ are differentiable at $\tau_*$.
\end{condition}  

The final condition is needed to derive a $\big(\hat{W}_{m,0}(\cdot), \hat{W}_{m,1}(\cdot) \big)$ FCLT from a $\big(W_{m,0}(\cdot), W_{m,1}(\cdot) \big)$ FCLT. We would like to hold the subdistribution functions $F_{m,0}$ and $F_{m,1}$ constant as $m$ changes but this does not hold in all cases of interest.  Instead we assume that they approach limits $F_0$ and $F_1$ at a fast rate.
\begin{condition}\label{cond:F0F1_fast_unif_conv}
    For $r \in \{0,1\}$, $\lim_{m \rightarrow \infty} \sup_{t \in [a,b]} \big|  \sqrt{m} \big(F_{m,r}(t) - F_r(t) \big)  \big| =0$.
\end{condition}

\paragraphLocal{Notes about Condition~\ref{cond:F0F1_fast_unif_conv}:} As with all of the other conditions in Section \ref{sec:def_and_conditions}, Condition~\ref{cond:F0F1_fast_unif_conv} merely needs to hold for the fixed value of the latent factor $\bswu\in\real^k$ on which we condition. In addition, a version of Theorem~\ref{theorem:conditionalFDPCLT} below will still hold if we loosen Condition~\ref{cond:F0F1_fast_unif_conv} to say that there exist Gaussian processes $Z_0$ and $Z_1$ on $[a,b]$ that are independent from the noise $\varepsilon$ such that for both $r \in \{0,1 \}$, 
$$ \sqrt{m} \big(F_{m,r}(\cdot) - F_r(\cdot) \big) \xrightarrow{\mathcal{D}} Z_r (\cdot).$$ This looser condition can be useful to study asymptotic behavior of the BH procedure in settings where the nonnull effect sizes $\mu_A$ are not constant and instead are assumed to come from some prior distribution. However, if we use this looser condition, the resulting theorem statement will be messier.

\section{Statement of the theorems}\label{sec:StatmentOfTheorems}

\begin{theorem}{\label{theorem:conditionalFDPCLT}} 
%Suppose we are in the multiple testing setup of 
For the model of Section \ref{sec:MultipleTestingSetup}, suppose that  conditionally on a specific value of the latent factor $\bswu=\bswl \in \real^k$ that Conditions 1--\total{condition} all hold.
Then
\begin{equation}\label{eq:Theorem_ConditionalCLT}
    \sqrt{m} \Big( \emph{\FDP}_m - \frac{qF_0(\tau_*)}{\tau_*} \mid\bswu=\bswl\Big) \xrightarrow{d} \dnorm(0, \sigma_L^2) 
\end{equation} 
as $m\to\infty$ where
\begin{equation}\label{eq:Variance_ConditionalCLT}
    \sigma_L^2 \equiv  \frac{q^2}{\tau_*^2 } 
    \big( (1+\alpha)^2 c^{(0,0)}(\tau_*,\tau_*) + \alpha^2  c^{(1,1)}(\tau_*,\tau_*) +  2 \alpha(1+\alpha) c^{(1,0)}(\tau_*,\tau_*) \big)
\end{equation} 
for the function $F_0$ given at~\eqref{eq:defFr},
the asymptotic ECDF $G$ given at~\eqref{eq:defghatandg},
the Simes point $\tau_*$ from~\eqref{eq:Simes_point_def},
the covariances $c^{(r_0,r_1)}(\cdot,\cdot)$ given by~\eqref{eq:defc}
and
\begin{equation}\label{eq:Def_of_alpha}
    \alpha \equiv \frac{ F_0'(\tau_*)-F_0(\tau_*)/\tau_*}{1/q -G'(\tau_*)}.
\end{equation} 
\end{theorem}

\begin{proof}
See the \hyperref[sec:supp]{supplemental material} for a proof of this theorem. 
\end{proof}

The proof is quite long but to summarize we first derive an FCLT for the joint process $(\hat{F}_{m,0}, \hat{F}_{m,1})$ by proving finite dimensional distribution convergence using a CLT from \cite{Neumann13} and then extend to an FCLT by using a result from \cite{AndrewsAndPollard}. We then define $\Psi^{(\FDP)}: D[a,b]_2 \rightarrow \real$ to be a particular function satisfying $\FDP_m=\Psi^{(\FDP)}(\hat{F}_{m,0},\hat{F}_{m,1})$ with probability converging to $1$ as $m \rightarrow \infty$. Then we argue that $\Psi^{(\FDP)}$ is Hadamard differentiable at $(F_0,F_1)$ tangentially to $C[a,b]_2$ and compute the Hadamard derivative by mimicking the approach in \cite{Neuvial2008}. To complete the proof of the CLT given in \eqref{eq:Theorem_ConditionalCLT}, we tie these results together with the functional delta method in Chapter 20.2 of \cite{VDV}.
Using the same proof technique we obtain the following conditional CLT for the ratio $V_m/m$.

\begin{theorem}{\label{theorem:conditionalFPR_CLT}} 
Under the conditions of Theorem~\ref{theorem:conditionalFDPCLT}, \begin{equation}\label{eq:Theorem_ConditionalFPR_CLT}
    \sqrt{m} \Big( \frac{V_m}{m} - F_0(\tau_*)\bigm|\bswu=\bswl \Big) \xrightarrow{d} \dnorm(0, \sigma_R^2) \quad \text{as } m \rightarrow \infty
\end{equation} 
where
\begin{equation}\label{eq:VarianceFPR_ConditionalCLT}
    \sigma_R^2 \equiv   (1+\beta)^2 c^{(0,0)}(\tau_*,\tau_*) + \beta^2  c^{(1,1)}(\tau_*,\tau_*) +  2 \beta(1+\beta) c^{(1,0)}(\tau_*,\tau_*)  
\end{equation} 
for 
\begin{equation}\label{eq:Def_of_beta}
    \beta \equiv \frac{ F_0'(\tau_*)}{1/q -G'(\tau_*)}.
\end{equation}

\end{theorem}

\begin{proof}
See the \hyperref[sec:supp]{supplemental material} for a proof of this theorem. 
\end{proof}

It is often the case that for a particular factor and noise model, Conditions 1--\total{condition} will not hold for all values of the latent factor $\bswu =\bswl \in \real^k$. Condition \ref{cond:noaccum} will often be violated when drawing $\bswu \sim \dnorm(0,I_k)$. We do not expect a positive probability that $\tau_*$ will be an accumulation point of $\{t \such G(t)=t/q\}$, nor do we expect a positive probability that $\sup \{ t \in (0,1) \such G(t) \geq t/q \} \neq \sup \{ t \in (0,1) \such G(t) > t/q \}$, but we do expect a positive probability that $\tau_*=0$. Therefore, in the next theorem we describe the asymptotic behavior of the BH procedure, conditional on $\bswu=\bswl \in \real^k$ when $\tau_*=0$.

\begin{theorem}{\label{theorem:tau_star_0_res}}
    
    For the model of Section \ref{sec:MultipleTestingSetup}, when conditioning on the latent factor $\bswu=\bswl \in \real^k$, suppose that Conditions \ref{cond:boundedL}, \ref{cond:niceQandgamma}, and \ref{cond:F0F1_defined} hold, that $\tau_*=0$, and that Conditions \ref{cond:cov_kernel_converges}, \ref{cond:F0F1_fast_unif_conv} hold when setting $[a,b]=[0,1]$. Then 
    \begin{equation}
        \tauBHm \giv \bswu=\bswl \xrightarrow{p} 0 \quad\text{ and } \quad\frac{V_m}{m} \giv \bswu=\bswl \xrightarrow{p} 0.
    \end{equation}
\end{theorem}

\begin{proof}
    See the \hyperref[sec:supp]{supplemental material} for a proof of this theorem.
\end{proof}

\begin{remark}
    Theorem \ref{theorem:tau_star_0_res} will still hold if we loosen the mixing Condition \ref{cond:niceQandgamma} and simply require that the error array $\{ \err_{mi} \such 1 \leq i \leq m < \infty \}$ has summable $\alpha$-mixing coefficients. This is because the proof of Proposition S.2.1 in the supplemental material \citep{KlugerAndOwen_Supplement} 
    does not require Condition \ref{cond:niceQandgamma} to hold and merely requires that the error array has summable $\alpha$-mixing coefficients.
\end{remark}

\begin{remark}\label{remark:FDP_not_to_0}
    When $\tau_*=0$ and $\tauBHm \xrightarrow{p} 0$, it is not guaranteed that  $\FDP_m \xrightarrow{p} 0$. For example, in the scenario where $\bm{L}_{mi}=0$ for all $m,i$, the errors $( \err_{m1},\dots , \err_{mm} )$ are independent and $\pi_1^{(m)}=1/m$, one can show that $\tau_*=0$ and $\tauBHm \xrightarrow{p} 0$, yet $\liminf_{m \rightarrow \infty} \Pr ( \FDP_m \geq 0.5) >0$. As another example, \cite{GontscharukFinner13} provide a scenario where $\tauBHm \xrightarrow{p} 0$, but asymptotically the false discovery rate exceeds the FDR control parameter $q$, implying that $\FDP_m$ cannot possibly converge to zero in probability in their scenario.

\end{remark}

In the above theorems, the limiting of behavior of BH depends on a latent factor which in practice is unobserved. Estimation of the unobserved latent factor is out of scope for this paper but for estimators of the latent factor and properties of these estimators we point the reader to \cite{FanHanGu}, \cite{AzrielAndSchwartzman}, \cite{sun2012multiple}, \cite{wang:zhao:hast:owen:2017} and \cite{FanFarmtest}. 
In the next section we focus on results for the case where the correlations are short-range and there is no factor model or latent factor to consider.

\section{Corollaries when there is no factor model component}\label{sec:no_factor_model}

The simplest applications of Theorems \ref{theorem:conditionalFDPCLT} and \ref{theorem:conditionalFPR_CLT} are to settings where there is no factor model component. That is $k=0$, or equivalently $\bm{L}_{mi}=0$ for $1\le i\le m$. Then the test statistics are  $X_{mi}=\mu_A H_{mi}+\err_{mi}$ where $(\err_{m1},\dots,\err_{mm}) \sim \dnorm(0, \Gamma^{(m)})$ for a correlation matrix $\Gamma^{(m)}\in\real^{m\times m}$. 
Next suppose, as is usual in the two-group mixture model that $\pi_1^{(m)}=\pi_1\in(0,1)$ for all $m\ge1$.
Finally, we will assume that the errors $(\err_{m1},\dots,\err_{mm})$ satisfy mixing Condition~\ref{cond:niceQandgamma}, which can hold, for example, if the errors are $M$-dependent (see Remark~\ref{remark:stationaryARMA_fast_enough} for other examples where Condition~\ref{cond:niceQandgamma} is met).

\sloppy Many of the \total{condition} conditions in our theorems hold trivially in this setting. Most trivially, $S_L=0$ making Condition \ref{cond:boundedL} hold. The mixing Condition \ref{cond:niceQandgamma} holds by assumption. Also in this setting, because $F_0(t)=F_{m,0}(t)=(1-\pi_1)t$ and $F_1(t)=F_{m,1}(t)=\pi_1 \bar{\Phi} ( \bar{\Phi}^{-1}(t) - \mu_A)$ for each $m$, Conditions \ref{cond:F0F1_defined} and  \ref{cond:F0F1_fast_unif_conv} on the subdistributions can be seen to hold.

To check that Condition \ref{cond:noaccum} ruling out pathologies about $\tau_*$ holds 
note that $G(t)=(1-\pi_1)t+\pi_1 \bar{\Phi} ( \bar{\Phi}^{-1}(t) - \mu_A)$. A simple calculation shows that $G'(t)=(1-\pi_1)+ \exp(\mu_A \bar{\Phi}^{-1}(t) -\mu_A^2/2)$ and 
$$G''(t)= -\pi_1 \exp(\mu_A \bar{\Phi}^{-1}(t) -\mu_A^2/2)/\varphi( \bar{\Phi}^{-1}(t)),$$ 
implying that $G$ is strictly concave on $(0,1)$ and that $G'(t) \rightarrow \infty$ as $t \downarrow 0$. By strict concavity of $G$, and since both $G$ and the Simes line intersect the origin, $\{ t > 0 \such G(t)=t/ q \}$ contains at most one point. It remains to show existence of a point $t>0$ with $G(t)=t/q$.
%in $\{ t > 0 \such G(t)=t/ q \}$. 
Since $G'(t) \rightarrow \infty$ as $t \downarrow 0$ and $G(0)=0$, there must be an $\epsilon>0$ such that $G(\epsilon)>\epsilon/q$. Also $G(1)=1 < 1/q$, so the continuous function $t \mapsto G(t)-t/q$ must cross $0$ at some unique $t_* \in (0,1)$. By continuity of $G$ and uniqueness this unique $t_*$ is the Simes point $\tau_*$ defined in \eqref{eq:Simes_point_def} and further Condition \ref{cond:noaccum} will be satisfied. Because $\tau_* \in (0,1)$ and because $F_0$ and $F_1$ are differentiable on $(0,1)$, Condition \ref{cond:F0F1_differfentiable} also holds.

The only remaining condition to check is Condition~\ref{cond:cov_kernel_converges} on convergence of the covariance kernels. Let $(a,b) \subset (0,1)$ be any 
open interval containing $\tau_*$ and $q$. Since for $r \in \{0,1\}$, $\gamma_{mir}$ does not vary with $i$ or $m$, the first limit in Condition~\ref{cond:cov_kernel_converges} always holds. Since for each $m,i$, $\gamma_{mi0}(t)=t$ and $\gamma_{mi1}(t)=\bar{\Phi} ( \bar{\Phi}^{-1}(t) - \mu_A)$, Condition~\ref{cond:cov_kernel_converges} holds whenever
\begin{align}\label{eq:defvarrho}\varrho(t,s)\equiv \lim_{m \rightarrow \infty} \frac1m\sum_{i \neq j} \tilde{\rho} ( t ,s, \Gamma_{ij}^{(m)} )
\end{align}
exists for all $s,t\in \big[a, \bar{\Phi} ( \bar{\Phi}^{-1}(b) - \mu_A) \big]$
with $\tilde\rho$ defined at~\eqref{eq:rhotilde}.
We summarize this along with an application of Theorem \ref{theorem:conditionalFDPCLT} in Corollary~\ref{cor:CLTnoFactor}.

\begin{corollary}\label{cor:CLTnoFactor}
    In the setting of Section \ref{sec:MultipleTestingSetup}, suppose further that: 
    \begin{enumerate}
    \item [\quad i)] the factor loadings $\bm{L}_{mi}$ are all zero and
    \item [\quad ii)] the probability $\pi_1\in(0,1)$ of nonnull hypotheses does not depend on $m$.
    \end{enumerate}
    Let $\tau_*$ be the unique $t \in (0,1)$ satisfying $t/q=\pi_0t+\pi_1 \bar{\Phi}( \bar{\Phi}^{-1}(t)- \mu_A)$ and let $(a,b) \subset (0,1)$ be any open interval containing both $\tau_*$ and $q$. If the $\err_{mi}$ are such that mixing Condition~\ref{cond:niceQandgamma} holds and the correlations among the $\err_{mi}$ are such that $\varrho(t,s)$ defined at~\eqref{eq:defvarrho} exists
    for all $t,s \in \big[a, \bar{\Phi} ( \bar{\Phi}^{-1}(b) - \mu_A) \big]$, %$\varrho_{r_0r_1}(t,s)=\lim_{m \rightarrow \infty}  (1/m) \sum_{i \neq j} \tilde{\rho} ( t ,s, \Gamma_{ij}^{(m)} )$ 
    then 
    \begin{equation}
        \sqrt{m} \big( \emph{\FDP}_m- \pi_0 q \big) \xrightarrow{d} \dnorm \Big(0, \frac{\pi_0^2 q^2}{\tau_*^2} \Big(\frac{\tau_*}{\pi_0} -  \tau_*^2+ 
      \varrho(\tau_*,\tau_*)  \Big) \Big)
    \end{equation}
        where $\pi_0 \equiv 1- \pi_1$.

\end{corollary}

\begin{proof}
As discussed before the statement of the corollary, Conditions 1--\total{condition} hold in this setting. Noting that in this setting $\alpha=0$ and there is no dependency on the latent factor $\bswu$, the result holds by Theorem \ref{theorem:conditionalFDPCLT}.
\end{proof}

We note that this corollary will hold even if $q \notin (a,b)$ and the requirement that $q \in (a,b)$ was included for a cleaner proof of Theorem \ref{theorem:conditionalFDPCLT}. 

We also note that if the correlations between the test statistics are known, or even if only the first few moments of the test statistic pairwise correlations are known, the quantity $\varrho(\tau_*,\tau_*)$ can be computed efficiently using the first few terms in a Hermite polynomial expansion, as seen in Theorem 2 of \cite{SchwartzmanAndLin}. Below, we specialize Corollary~\ref{cor:CLTnoFactor} to settings with block diagonal correlations and with Toeplitz correlations. In these settings, Condition \ref{cond:niceQandgamma} will hold and $\varrho(\tau_*,\tau_*)$ will have an easily-expressed formula. 

\begin{corollary}[Block diagonal correlations]\label{cor:CLTBlock}
In the setting of Corollary~\ref{cor:CLTnoFactor}, suppose that 
    %Suppose our test statistics follow a two-group mixture model with mixture parameter $\pi_1$ fixed in $(0,1)$ for all $m$, and suppose the test statistics have no factor model component with errors coming from a Gaussian block model, i.e., 
    $(\err_{m1},\dots,\err_{mm})$
    %\sim \dnorm(0,\Sigma^{(m)})$ where $\Sigma^{(m)}$ 
    %is a correlation matrix with blocks of size $s_B$ and off diagonal correlations on the block $\rho_B$. 
    has a block diagonal correlation matrix with blocks of fixed size $s_B$ in which the off diagonal correlations are $\rho_B$.
    %We assume that the blocks are of equal size, except the last one if $m$ does not divide $s_B$. 
    Let $\tau_*$ be the unique $t \in (0,1)$ satisfying $t/q=\pi_0t+\pi_1 \bar{\Phi}( \bar{\Phi}^{-1}(t)- \mu_A)$. Then
    \begin{equation}
        \sqrt{m} \big( \emph{\FDP}_m- \pi_0 q \big) \xrightarrow{d} \dnorm \bigg(0, \frac{\pi_0^2 q^2}{\tau_*^2} \Big( \frac{\tau_*}{\pi_0} -  \tau_*^2+ 
      (s_B-1)  \tilde{\rho} ( \tau_* ,\tau_*, \rho_B ) \Big) \bigg)
     \end{equation} where $\pi_0 \equiv 1- \pi_1$ and $\tilde{\rho}$ is defined at~\eqref{eq:rhotilde}.
\end{corollary}
\begin{proof}
This follows from a direct application of Corollary \ref{cor:CLTnoFactor}.
\end{proof}
The corollary as written requires $m$ to be a multiple of $s_B$, but it extends easily to $m\to\infty$ through an arbitrary sequence of $m$. One can let the ``last" block be smaller than the others if necessary.

Another simple correlation structure we can consider has banded Toeplitz correlation matrices for $\err_{m1},\dots,\err_{mm}$. 

\begin{corollary}[Toeplitz correlation]
\label{cor:SimpleToeplitz}
In the setting of Corollary~\ref{cor:CLTnoFactor}, suppose that 
$(\err_{m1},\dots,\err_{mm})$ has a Toeplitz correlation matrix $$\Gamma_{ij}^{(m)}=I\{i=j \}+ \sum_{\ell=1}^M \rho_l I \{ \vert i-j \vert=l \}$$%+\rho_2 I \{\vert i-j \vert=2 \}$$ 
where $\rho_1, \dots, \rho_M \in (-1,1)$ are such that $\Gamma^{(m)}$ is positive semi-definite for all $m>M$. Let $\tau_*$ be the unique $t \in (0,1)$ satisfying $t/q=\pi_0t+\pi_1 \bar{\Phi}( \bar{\Phi}^{-1}(t)- \mu_A)$. Then,
\begin{equation}
        \sqrt{m} \big( \emph{\FDP}_m- \pi_0 q \big) \xrightarrow{d} 
        \dnorm\bigg(0, \frac{\pi_0^2 q^2}{\tau_*^2} \Big(\frac{\tau_*}{\pi_0} -  \tau_*^2+ 
     2\sum_{\ell=1}^M\tilde\rho(\tau_*,\tau_*,\rho_\ell)  \Big) \bigg)
     \end{equation} where $\pi_0 \equiv 1- \pi_1$ and  $\tilde{\rho}$ is defined at~\eqref{eq:rhotilde}.
\end{corollary}
\begin{proof}
This follows from a direct application of Corollary \ref{cor:CLTnoFactor}.
\end{proof}

We check the CLTs provided by Corollaries \ref{cor:CLTBlock} and \ref{cor:SimpleToeplitz} via simulation in Figure \ref{fig:BlockAndToeplitzCLTConvergence}. We compare the normal approximation given by these CLTs to the normal approximation given by Corollary 4.2 in \cite{DR16}. We simulate block and banded correlation structures that do not satisfy the sufficient conditions in their Corollary 4.2. 

For each correlation structure and each $m \in \{10^2,10^3,10^4,10^5 \}$, we ran $25{,}000$ Monte Carlo simulations with $\mu_A=2$, $\pi_0=0.9$, and $q=0.1$.
The block correlation matrices we considered had 
block size $s_B=20$, and within block correlations $\rho_B =0.5$. 
The banded Toeplitz correlation matrix had $M=2$ with $\rho=\rho_1=0.65$ above and below the diagonal and $\rho=\rho_2=0.3$ two rows above and below the diagonal.
Our normal approximation fits very well for the larger values of $m$. For large $m$ the normal approximation from \cite{DR16} appears to accurately estimate the mean but not the variance of the FDP. From this we believe that something in their sufficient conditions must also have been necessary.
\begin{figure}[t]
\centering
\includegraphics[width=0.95\hsize]{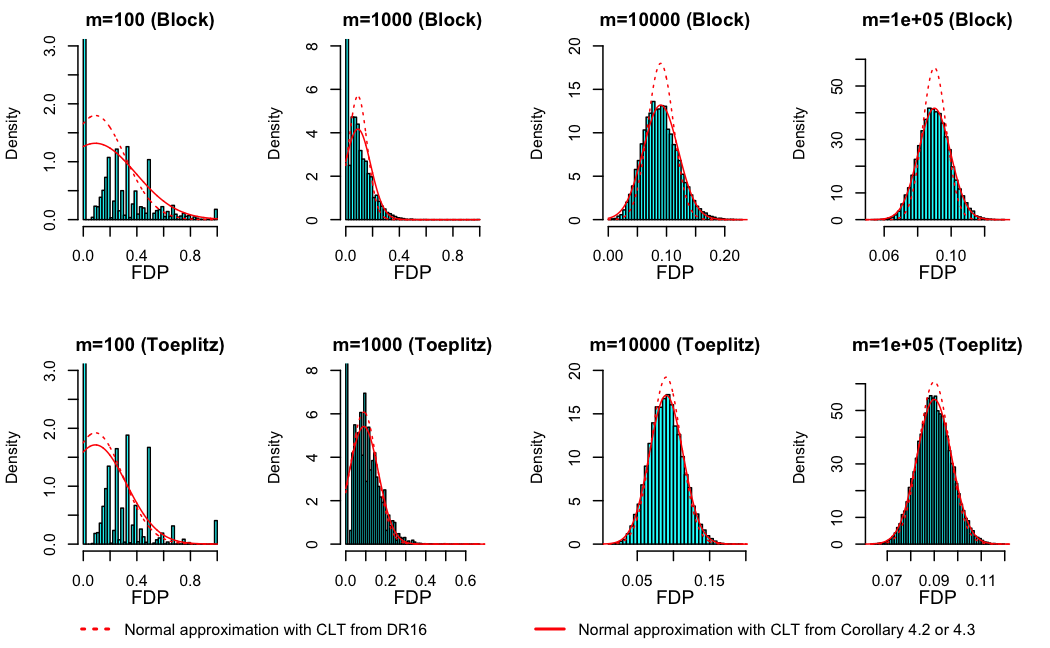}
\caption{\label{fig:BlockAndToeplitzCLTConvergence}  
This figure compares the histograms of FDP to our normal approximation and an earlier one by \cite{DR16} that does not necessarily cover these cases.  Each panel is based on $25{,}000$ Monte Carlo simulations as described in the text.
}
\end{figure}

\section{Burstiness in a factor model}\label{sec:Implications_under_factor_model}

The results in the previous section are CLTs that do not require conditioning on the latent factor as they assumed no long-range correlations modeled via a factor model. The CLTs are messier when factor model components are introduced, so we present two examples for factor model settings where the formulas for the asymptotic distribution of the $\FDP$ have some simplifications. 

\subsection{\texorpdfstring{$1$}{1}-factor model for long-range equicorrelated Gaussian noise}\label{sec:one_factor_model}

Suppose that for each $m$, $H_{m1},\dots,H_{mm} \simiid\dbern(\pi_1)$ for a fixed $\pi_1 \in (0,1)$, but we now have a one dimensional latent factor; that is $\bswu \sim \dnorm(0,1)$. For simplicity, we consider the simplest factor model structure: an equicorrelated Gaussian model. In particular, we let $\bm{L}_{mi} =\sqrt{\rho_1}$ where $\rho_1 \in [0,1)$ for all $m,i$. We will also allow for errors with shorter range correlations to be added to the model by supposing that $(\tilde{\err}_{m1},\dots,\tilde{\err}_{mm}) \sim \dnorm(0,\Gamma^{(m)})$ where $\Gamma^{(m)}$ is a correlation matrix with blocks of size $s_B$ and off diagonal within-block correlations of $\rho_2$. We assume that the blocks are of equal size, except the last one if $m$ does not divide $s_B$. In this model the test statistics are $X_{mi}=\mu_A H_{mi} +\sqrt{\rho_1} \bswu+ \sqrt{1-\rho_1} \tilde{\err}_{mi}$ and the correlation structure of the errors (not related to the indicators $H_{mi}$ of whether the hypotheses are true) follows a matrix $\Sigma_{B_2}$ where $$(\Sigma_{B_2})_{ij}=\begin{cases} 1, & \text{ if } i=j, \\ 
\rho_1, & \text{ if } i \text{ and } j \text{ are in different blocks}, \\ 
\rho_1 +(1-\rho_1) \rho_2, & \text{ if } i \neq j \text{ and if } i \text{ and } j \text{ are in the same block}. \end{cases}$$

In such a setting it is easy to show that all of our conditions, except possibly Condition~\ref{cond:noaccum} will hold.
Condition \ref{cond:noaccum} ruling out pathologies involving $\tau_*$ will hold depending on the value of $\bswu$ drawn from $\dnorm(0,1)$.
Some $\bswl$ may give $\tau_*=0$ though we do not expect a positive probability that $\tau_*$ will be an accumulation point of $\{t\such G(t)=t/q\}$.

\begin{corollary}\label{cor:eqGaussianWithBlock}

    In the multiple hypothesis testing setting of Section \ref{sec:one_factor_model}, condition on $\bswu=\bswl \in \real$ and for $t \in (0,1)$ let $G(t)=(1-\pi_1) \gamma_0(t)+\pi_1 \gamma_1(t)$ with $\gamma_r(t)= \Bar{\Phi} \big( (\Bar{\Phi}^{-1}(t)-\mu_A r-\sqrt{\rho_1} \bswl ) /\sqrt{1-\rho_1} \big)$ for $r \in \{0,1 \}$. If $\bswl$ is such that $G$ satisfies Condition~\ref{cond:noaccum} that the Simes point is positive with no pathologies, then 
    
    $$\sqrt{m} \Big( \emph{\FDP}_m - \frac{q \pi_0 \gamma_0(\tau_*)}{\tau_*} \giv \bswu= \bswl \Big) \xrightarrow{d} \dnorm(0, \sigma_{L,2}^2) \quad \text{as } m \rightarrow \infty$$ where $\pi_0 \equiv 1-\pi_1$ and

    $$\sigma_{L,2}^2 =  \frac{q^2}{\tau_*^2 } \Big( (1+\alpha)^2 c^{(0,0)}(\tau_*,\tau_*) + \alpha^2  c^{(1,1)}(\tau_*,\tau_*) +  2 \alpha(1+\alpha) c^{(1,0)}(\tau_*,\tau_*) \Big) $$
    where 
    $$\alpha = \frac{ \pi_0 \gamma_0'(\tau_*)-\pi_0 \gamma_0(\tau_*)/\tau_*}{1/q -G'(\tau_*)}$$
        and for $r \in \{0,1\}$, $$c^{(r,r)}(\tau_*,\tau_*)= \pi_{r} \gamma_{r}(\tau_*)-\pi_{r}^2 \gamma_{r}(\tau_*)\gamma_{r}(\tau_*)+ \pi_{r}^2 (s_B-1) \tilde{\rho}( \gamma_{r}(\tau_*),\gamma_{r}(\tau_*),\rho_2)$$ and $$c^{(1,0)}(\tau_*,\tau_*)= -\pi_{0} \pi_{1}\gamma_{0}(\tau_*)\gamma_{1}(\tau_*)+ \pi_{0} \pi_{1} (s_B-1) \tilde{\rho}( \gamma_{0}(\tau_*),\gamma_{1}(\tau_*),\rho_2)$$ where $\tilde{\rho}$ is defined in Equation~\eqref{eq:rhotilde}.

\end{corollary}

\begin{proof}
As mentioned earlier in this section, Conditions  \ref{cond:boundedL}, \ref{cond:niceQandgamma}, \ref{cond:F0F1_defined}, \ref{cond:cov_kernel_converges}, \ref{cond:F0F1_differfentiable}, and \ref{cond:F0F1_fast_unif_conv} will hold for any value of $\bswu$ drawn. Therefore the above result holds from applying Theorem \ref{theorem:conditionalFDPCLT} in the setting where Condition \ref{cond:noaccum} also holds.
\end{proof}

\begin{remark}
    In the setting of Corollary \ref{cor:eqGaussianWithBlock}, if $m/s_B$ is small, then perhaps the test statistic correlations can be modeled with a factor model with a bit more than $m/s_B$ factors but asymptotically such an approach would require adding infinitely many factors in the model. 
    In our setup, the equicorrelations $\rho_1$ are long-range and persist as $m \rightarrow \infty$ and hence they are modeled with a factor model whereas the additional noise with correlation blocks of size $s_B$ involves short-range correlations and are therefore not modeled as a factor.
\end{remark}

In this subsection, we have demonstrated that Theorems \ref{theorem:conditionalFDPCLT} and \ref{theorem:conditionalFPR_CLT} can be used to provide further insight into asymptotics of the BH procedure in the setting of Gaussian test statistics with constant pairwise correlation. In such a setting, \cite{FinnerDickhausRoters2007} find the limiting expected values of both the FDP and the FPR as functions of a one-dimensional latent factor. We have extended their results by deriving the limiting distribution of the FDP as a function of a one-dimensional latent factor (the limiting distribution of the FPR can similarly be derived from Theorem \ref{theorem:conditionalFPR_CLT}). We also considered a more general setting than the equicorrelated Gaussian model in order to exhibit that Theorems \ref{theorem:conditionalFDPCLT} and \ref{theorem:conditionalFPR_CLT} can handle settings with both short and long-range correlations simultaneously.

\subsection{Setting where number of nonnulls is \texorpdfstring{$o_p(\sqrt{m})$}{oprootm}}\label{sec:nonnullsOpRootm}

Here we consider sparse nonnulls
with $\pi_{1}^{(m)}=o(1/\sqrt{m})$. 
It can be shown with a Chernoff bound for the binomial that in this case the number of nonnulls is $o_p(\sqrt{m})$. Also suppose that under the model for test statistics of Section \ref{sec:MultipleTestingSetup}, Conditions 1-\total{condition} hold. Then it will follow that $F_1=0$ and moreover $W_{m,r} =\sqrt{m} ( \hat{F}_{m,1} -F_{m,1}) \xrightarrow{p} 0$. 
If this is the case, then we will have $c^{(1,1)}=0$, $c^{(1,0)}=0$, $F_0(\tau_*)=\tau_*/q$ and $\alpha =-1$.

\begin{corollary}\label{cor:sparse_nonnulls}
     Suppose that we are in the multiple testing setting of Section \ref{sec:MultipleTestingSetup} and that $\pi_{1}^{(m)}=o(1/\sqrt{m})$. If, conditionally on a specific value of the latent factor $\bswu=\bswl \in \real^k$, Conditions 1-\total{condition} hold, then $$\sqrt{m}\Big( \emph{\FDP}_m - 1 \giv \bswu=\bswl \Big) \xrightarrow{p} 0 $$ and $$\sqrt{m} \Big( \frac{V_m}{m} - \frac{\tau_*}{q} \giv \bswu=\bswl \Big) \xrightarrow{d} \dnorm(0, \sigma_{R,3}^2) \quad \text{as } m \rightarrow \infty$$ where $$\sigma_{R,3}^2 \equiv (1-q G'(\tau_*))^{-2} c^{(0,0)}(\tau_*,\tau_*).$$
\end{corollary}

\begin{proof}
    Apply Theorems \ref{theorem:conditionalFDPCLT} and \ref{theorem:conditionalFPR_CLT} noting that $F_0(\tau_*)=\tau_*/q$, $\alpha=-1$, $c^{(1,1)}=0$, $c^{(1,0)}=0$.
\end{proof}

The corollary indicates that severe bursts can occur; $V_m/m$ can converge to a positive number even while the proportion of hypotheses that are nonnull converges to 0. 

We check the result of Corollary \ref{cor:sparse_nonnulls} via simulation in Figure \ref{fig:SparseNonnullLimitBehavior}. We simulate from the 1-factor model described in Section \ref{sec:one_factor_model}, except now the proportion of nonnulls $\pi_{1}^{(m)}$ is not fixed in $m$. Instead we set $\pi_{1}^{(m)}=5 m^{-2/3}$. 
For each $m \in \{10^2,10^3,10^4,10^5, 10^6 \}$, we conditioned on $\bswl=2.5$ and ran $25{,}000$ Monte Carlo simulations with $\mu_A=2$, $\pi_0^{(m)}=1-5 m^{-2/3}$, $q=0.1$, $\rho_1=0.3$, and $\rho_2=0.6$. For these choices of parameters, the conditions of Corollary \ref{cor:sparse_nonnulls} are met.

\begin{figure}
\centering

\includegraphics[width=0.95\hsize]{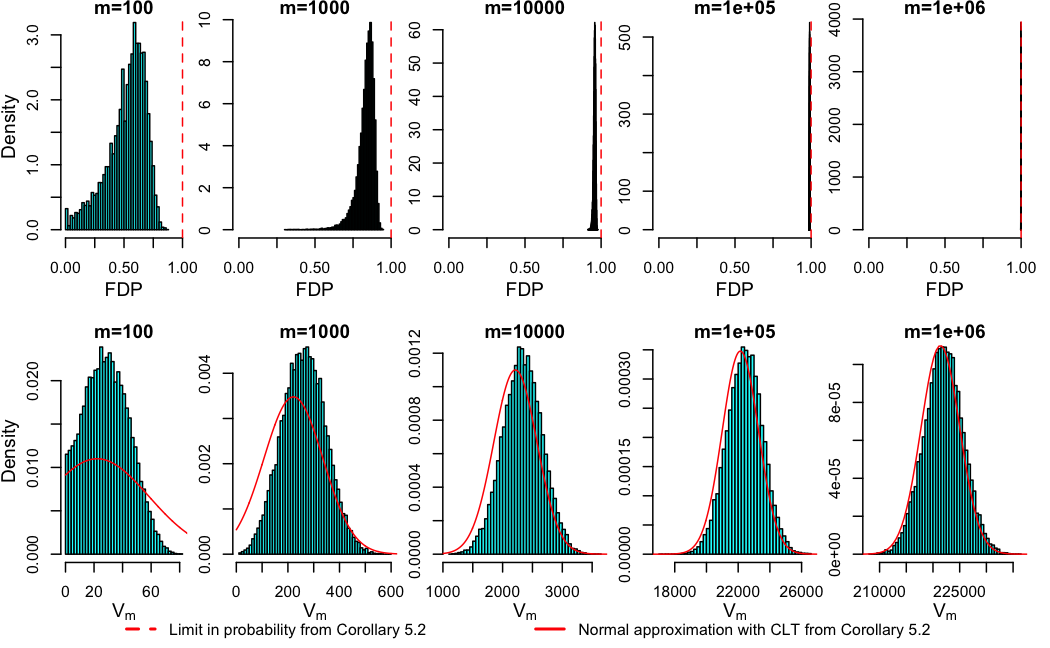}
\caption{ \label{fig:SparseNonnullLimitBehavior} 
This figure compares the histograms of the FDP and the number of discoveries $V_m$ to the asymptotic estimates of their distributions given by Corollary \ref{cor:sparse_nonnulls}.  Each panel is based on $25{,}000$ Monte Carlo simulations as described in the text.
}
\end{figure}

\section{Data driven factor model example}\label{sec:DataDrivenExample}
We fit a 3-factor model to the GDS 3027 Duchenne Muscular Dystrophy (DMD) data, which can be found on the Gene Expression Omnibus. This data set was analyzed in \cite{OtherDMD_dataset_use} and \cite{Jingshu2020DMD} 
and had $m=22{,}283$ genes and $n=37$ subjects. Of these subjects 23 had DMD and 14 did not.  We centered the data for each gene and stored it in a matrix $Y \in \real^{37 \times 22{,}283}$. 
To fit a homoskedastic factor model, we looked at the plot of the singular values of $Y$ and 
chose to work with the largest three of them for illustrative purposes. We then computed $Y_3$, the singular value decomposition-based rank $3$ approximation to $Y$, and estimated the homoskedastic noise, $\sigma_E$ as the standard deviation of the entries in $Y-Y_3$. 
We let $\tilde{L} \in \real^{m\times 3}$ be the matrix whose columns consist of the first 3 right singular vectors of $Y$ scaled by their corresponding singular values. 
We subsequently treat $\tilde{L}$ and $\sigma_E$ as fixed quantities and then assume the following factor model under the global null: $Y^\tran=\tilde{L}F+\sigma_E E$, where the entries of $F \in \real^{3 \times 37}$ and $E \in \real^{m\times 37}$ are IID standard Gaussians. Under the alternative we suppose that for each nonnull gene, the values of $Y$ for that gene are shifted by a fixed constant for DMD subjects and a different fixed constant, maintaining centering of the columns of $Y$, for the control subjects. 

Since the dataset is from a case-control study, to compute the test statistics we condition on DMD status and assume that the stochasticity in our observations $Y$ comes from the random matrices $E$ and $F$. The unstandardized test statistics $X_{\text{ustd}} \in \real^{m}$ are simply the difference-in-means between the DMD group and the control group for each gene and this unstandardized test statistics vector has covariance matrix proportional to $\sigma_E^2I_{m}+\tilde{L} \tilde{L}^\tran$. The standardized test statistics $X \in \real^{m}$ are given by dividing each entry of $X_{\text{ustd}}$ by the squareroot of the corresponding diagonal entry of $\sigma_E^2I_{m}+\tilde{L} \tilde{L}^\tran$. The vector of test statistics then satisfy $X= \vec{\mu} +\bm{L}\bswu+\varepsilon$ where $\vec{\mu}$ is a vector of constant means (which are zero for the null genes), $\bm{L}$ is a matrix of factor loadings similar to $\tilde{L}$ but with appropriately rescaled rows, $\bswu \sim \dnorm(0,I_3)$ and $\varepsilon$ is heteroskedastic, independent, and centered Gaussian noise. Assuming that each standardized nonnull test statistic has the same mean $\mu_A>0$, that we conduct one-sided testing, and that the nonnulls are determined by IID $\dbern(\pi_1)$ draws, using the test statistics $X$ we are in the multiple hypothesis testing setting of Section \ref{sec:MultipleTestingSetup}.

\begin{figure}[t]
\centering
\includegraphics[width=0.95 \hsize]{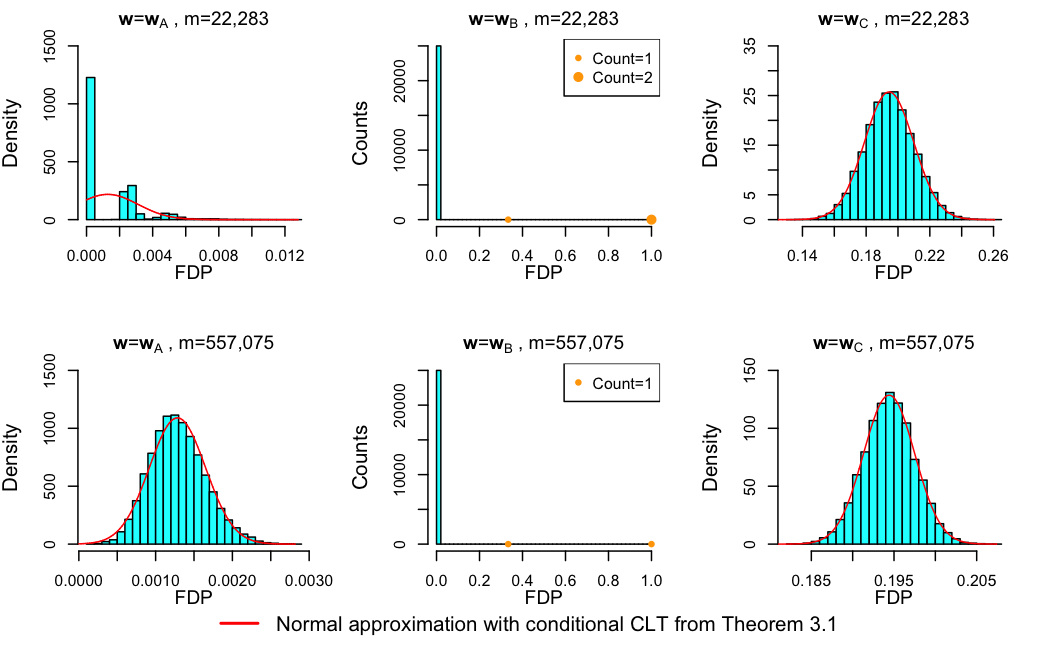}
\caption{\label{fig:Theorem1_DMDk3_convergence} 
The histograms show $25{,}000$ samples of the number of false discoveries in the two group model with a data-driven three factor model for dependence.
The 3-factor model has draws A, B and C as described in the text. The top row has $m=22{,}283$ hypotheses and the bottom has $25$ times as many hypotheses. The nominal FDR control threshold is $q=0.1$.  Draw C yields a higher FDP while draw A yields a lower one. In draw B, nearly all of the $25{,}000$ Monte Carlo simulations yielded an $\FDP=0$, and the small number of simulations for which $\FDP \neq 0$ are denoted by orange circles for visibility. The red curves for A and C are asymptotic Gaussians from Theorem~\ref{theorem:conditionalFDPCLT}. Case B does not satisfy the sufficient conditions for the conditional CLT, but satisfies the conditions for Theorem~\ref{theorem:tau_star_0_res}. 
}
\end{figure}

Figure~\ref{fig:SimesLinePoint_example}
in Section~\ref{sec:setup}
shows the asymptotic ECDF of the $p$-values for three specific realizations of the latent factor $\bswl$ in the data-driven 3-factor model and multiple testing setting described above, with $\mu_A=2$, $\pi_1=0.1$, and $q=0.1$. Figure~\ref{fig:Theorem1_DMDk3_convergence} shows histograms of the $\FDP$ based on $25{,}000$ Monte Carlo simulations for the same data-driven 3-factor model, multiple testing setup, factor outcomes, and parameters as Figure~\ref{fig:SimesLinePoint_example}.

In the top panel of Figure~\ref{fig:Theorem1_DMDk3_convergence}, $m=22{,}283$ as is the case in the original 3-factor model fit to the GDS 3027 dataset. In the bottom panel, to increase the number of tests and check asymptotic behavior, we copy each row of factor loadings in the original factor model $25$ times to get a distribution of the $\FDP$ when $m=22{,}283 \times 25$ tests are conducted. That is much larger than we would need for gene expression and approaches the range we would encounter for SNPs. In case A, the CLT is reasonable for the larger but not the smaller
sample size.  The CLT fits well for both sample sizes for case C.  In case B, the sufficient conditions for the conditional CLT do not hold and Theorem \ref{theorem:tau_star_0_res} holds instead.

This simulation shows some bursty behavior for BH as follows.  Cases A and C are both covered by the conditional CLT and there we see that even in cases covered by the conditional CLT, the FDP can vary greatly, being nearly Gaussian with means varying by nearly 100-fold. When cases like case B arise there is no conditional CLT, and by Theorem \ref{theorem:tau_star_0_res}, the BH rejection threshold converges to $0$ in probability. In case B, we observe a very heavy tail to the FDP distribution, although fewer than 1 out of every $5{,}000$ Monte Carlo simulations yields a nonzero FDP, and no simulation yields more than 1 false discovery.  In conclusion, the simulations are consistent with the results of Theorems \ref{theorem:conditionalFDPCLT} and \ref{theorem:tau_star_0_res}.

\begin{figure}[t]
\centering
\includegraphics[width=0.945 \hsize]{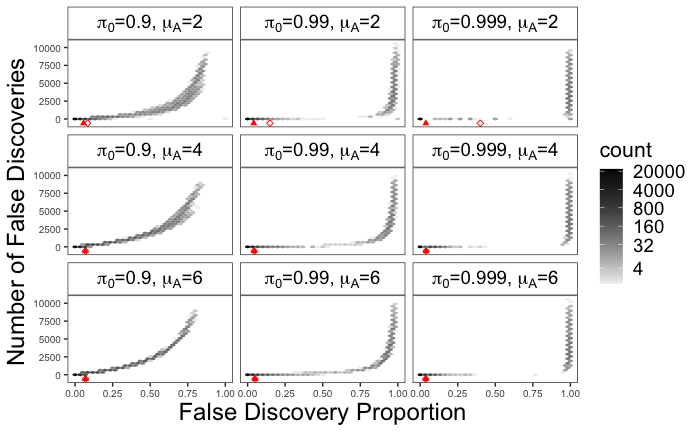}
\caption{\label{fig:DMDk3_Burstiness_Sims} 
The hexagonal-binned heatmaps each show the results of $25{,}000$ Monte Carlo simulations for the data-driven 3-factor model described in the text, using various nonnull effect sizes $\mu_A$ and proportions of nulls $\pi_0$. In each simulation, the BH method is applied at level $q=0.1$ to $m=22{,}283$ test statistics. For each plot, the mean of the FDP is marked with a solid triangle and the mean amongst nonzero $\text{FDP}$ values is marked with a hollow diamond, which estimate the $\text{FDR}$ and $\text{pFDR}$ respectively. For the plots with $\mu_A \in \{4,6\}$, the triangle and diamond are close enough to overlap.
}
\end{figure}

A large $\FDP$ tail is not necessarily indicative of alarming bursty behavior for BH, as the $\FDP$ can be equal to $1$ in scenarios where there is only one false discovery. Looking at the $\FDP$ simultaneously with the number of false discoveries $V_m$ gives a clearer sense of whether the bursty behavior is alarming. In Figure~\ref{fig:DMDk3_Burstiness_Sims}, we run ${25,000}$ Monte Carlo simulations using the previously described data-driven 3-factor model. In contrast to Figure~\ref{fig:Theorem1_DMDk3_convergence}, we do not condition on specific realizations of the latent factor $\bswl$ and we also plot the joint distribution of the $\FDP$ and the number of false discoveries rather than the marginal distribution of the $\FDP$. In the simulations, we set the FDR control parameter $q=0.1$ and repeat the simulations for the nonnull effect size $\mu_A \in \{2,4,6\}$ and for Bernoulli mixture null parameter $\pi_0 \in \{0.9,0.99,0.999\}$.

\begin{remark}
    Controlling pFDR using Storey's $q$-value is another popular multiple testing approach that is heralded for avoiding floods of false positives (\cite{StoreyTibsh}). Our simulations show that when the nonnull effect sizes are small and the number of nonnulls is sparse the pFDR will be high, implying that controlling the pFDR would mitigate the issue of burstiness in such cases. When there are many nonnulls or when the nonnull effect sizes are large, the pFDR is nearly equal to the FDR (due to few simulations with no discoveries), implying that controlling the pFDR would not mitigate burstiness in such cases. 
\end{remark}

\begin{remark}
For the data-driven 3-factor model with $m=22{,}283$, the distribution of $\FDP_m$ is largely driven by the realization of the latent factor $\bswl$. This can be seen in the top panel of Figure \ref{fig:Theorem1_DMDk3_convergence}, and we further quantified this by running $1{,}000$ Monte Carlo simulations from the model for each of $1{,}000$ different randomly generated latent factor vectors. The variance of the FDP between groups with different latent factors was approximately $325$ times larger than the average variance within each latent factor vector group.
\end{remark}

\section{Discussion}\label{sec:Discussion}

Here we discuss the conclusions that can be drawn from our theorems and simulations about when BH exhibits alarming burstiness and when BH is safe from burstiness concerns. We end with a discussion of the relevance and feasibility of factor model-based corrections for addressing burstiness concerns.

Burstiness occurs when there are many strong, long-range correlations between the test statistics. When we model the long-range correlations via a factor model, this phenomenon can be explained by Theorem \ref{theorem:conditionalFDPCLT}. By Theorem \ref{theorem:conditionalFDPCLT}, the asymptotic limit of $\FDP_m \giv \bswu$ is $qF_0(\tau_*)/\tau_*$, a quantity that can vary drastically for different realizations of $\bswu \sim \dnorm(0,I_k)$. The variation in $qF_0(\tau_*)/\tau_*$ is greater when the long-range correlations are stronger (or equivalently, when the factor model loading vectors $\bm{L}_{mi}$ have larger magnitude). Therefore, the FDP has high variability when there are many strong, long-range correlations between the test statistics. Meanwhile, by Theorem \ref{theorem:conditionalFPR_CLT}, there could be a flood of false discoveries, making the bursts severe. Our simulations from the 3-factor model fit to the DMD dataset indicate a wide right tail of the FDP distribution as well as severe bursts (see the top-left panel of Figure \ref{fig:Intro_burst4examples} and Figure \ref{fig:DMDk3_Burstiness_Sims}). Notably, we find that sparsity of the number of nonnulls exacerbates burstiness issues. This can be explained by Corollary \ref{cor:sparse_nonnulls} and is observed in Figures \ref{fig:SparseNonnullLimitBehavior} and \ref{fig:DMDk3_Burstiness_Sims}.

Conversely, our theorems and simulations indicate that there are many settings where the test statistics are correlated, but the BH procedure is free of burstiness concerns. When there are no long-range correlations, no factor model is needed to model the correlations, so the variance of the FDP will decrease rapidly as the number of tests increases, even when the short-range correlations are strong. For example, in the setting of Corollary \ref{cor:CLTnoFactor}, $\FDP_m$ converges to a quantity less than the desired FDR control $q$ and has variance of order $1/m$, even with strong short-range correlations. The simulations in the bottom of Figure \ref{fig:Intro_burst4examples} and all the panels of Figure \ref{fig:BlockAndToeplitzCLTConvergence} involve strong short-range correlations and still demonstrate this desirable behavior (the desirable behavior is not seen in the panels of Figure \ref{fig:BlockAndToeplitzCLTConvergence} where $m$ is small because, in that case, the ``short-range" correlations are actually long-range relative to the number of tests). Even when there are long-range correlations but the long-range correlations are weak, modeled by a factor model with loading vectors $\bm{L}_{mi}$ of small magnitude, the BH procedure will not exhibit worrisome bursts. With loading vectors $\bm{L}_{mi}$ with small magnitude, the $\gamma_{mir}$ terms will not be sensitive to the realized value of $\bswu \sim \dnorm(0,I_k)$, and in turn $F_0,F_1, \tau_*,$ and $qF_0(\tau_*)/\tau_*$ will also not be so sensitive to the realized value of $\bswu$. Therefore, when the loading vectors $\bm{L}_{mi}$ have small magnitude, the asymptotic limit of $\FDP_m \giv \bswu$ given in Theorem \ref{theorem:conditionalFDPCLT} will not oscillate much as $\bswu \sim \dnorm(0,I_k)$ varies. Indeed, in Figure \ref{fig:Intro_burst4examples}, when the long-range correlations are all reduced by a factor of 10 (as we move from the top-left panel to the top-right panel), alarming burstiness is no longer observed.

These results suggest that estimating the correlation structure with a factor model can be useful for identifying whether or not burstiness is a concern in particular applications. The results further suggest that methods which estimate and remove the factor model components from the test statistics prior to applying BH (e.g. methods that estimate $\bswu$ and $(\bm{L}_{mi} )_{i=1}^m$, subtract $\bm{L}_{mi}^\tran \bswu$ from each test statistic, and subsequently apply BH) can alleviate burstiness issues. A number of such approaches for estimating and removing factor model components in multiple testing settings have been proposed and have shown promise in simulations \citep{frig:kloa:casu:2009,sun2012multiple,FanHanGu,wang:zhao:hast:owen:2017,FanFarmtest}. 

While a full discussion of these recent methods is out of scope for this paper, we briefly note that there are two major challenges with estimation and removal of factor model components in a multiple testing setting. First, it is possible that removing the factor model components from the test statistics might remove some of the important signal that one is trying to detect with hypothesis testing. For example, this can happen if a large collection of genes is associated with one of the leading factors, yet at the same time, that collection of genes is also associated with the outcome variable. To avoid such issues, methods which remove the factor model components often rely upon an assumption that the number of nonnulls is sparse. Second, estimating the underlying factor model for a dataset is statistically challenging. It is difficult to estimate the latent factors $\bswu$ and the factor loadings $(\bm{L}_{mi} )_{i=1}^m$ well without a large number of samples, and even choosing the number of latent factors $k$ is a difficult task. For a comparison of methods for estimating the number of latent factors, see \cite{OwenWangBiCrossVal}.

\begin{acks}[Acknowledgments]
 The authors wish to thank Will Fithian, Kevin Guo, Grant Izmirlian, Lihua Lei, Kenneth Tay, Marius Tirlea, Jingshu Wang, and anonymous reviewers for helpful comments and discussions. The authors also thank Kevin Guo for a proof sketch on how to remove a condition from an earlier version of this paper and Jingshu Wang for sharing the DMD data. Finally, the authors would like to thank two anonymous referees, an Associate Editor and the Editor for their constructive comments that helped improve the quality of this paper.
\end{acks}

\begin{funding}
DMK was supported by a Stanford Graduate Fellowship and a Stanford Graduate Interdisciplinary Fellowship. ABO was supported by the National Science Foundation under grants IIS-1837931 and DMS-2152780.
\end{funding}

\bibliographystyle{imsart-nameyear.bst}
\bibliography{FDP_CLT}

\pagebreak

\begin{supplement}\label{sec:supp}
%\stitle{Supplement to ``A central limit theorem for the
%Benjamini-Hochberg false discovery proportion
%under a factor model"}
\stitle{Proofs of Theorems \ref{theorem:conditionalFDPCLT}, 
\ref{theorem:conditionalFPR_CLT}, and \ref{theorem:tau_star_0_res}}
\sdescription{A supplement with the proofs of Theorems \ref{theorem:conditionalFDPCLT}, \ref{theorem:conditionalFPR_CLT}, and \ref{theorem:tau_star_0_res} can be found below. The first section of the supplement restates definitions and theorems from the probability literature that we use in our proofs. The second section of the supplement provides proofs of Theorems \ref{theorem:conditionalFDPCLT}, \ref{theorem:conditionalFPR_CLT}, and \ref{theorem:tau_star_0_res}. The third section of the supplement exhibits Hadamard derivative calculations that are used in our proofs of Theorems \ref{theorem:conditionalFDPCLT} and  \ref{theorem:conditionalFPR_CLT}. }
\end{supplement}

%\appendix

\counterwithin{theorem}{section}
\counterwithin{lemma}{section}
\counterwithin{definition}{section}
\counterwithin{proposition}{section}
\counterwithin{remark}{section}
\counterwithin{corollary}{section}
%\counterwithin{equation}{section} %Uncomment if you want equations in the appendix to be numbered differently
%\let\oldTheorem
%\renewcommand{\theorem}{S\theorem}
\setcounter{section}{0}
\setcounter{equation}{0}

\renewcommand{\thesection}{S.\arabic{section}}   
\renewcommand{\theequation}{S\arabic{equation}}   

\section{
%Results and definitions from the literature that will be used in proving 
Background needed to prove Theorems \ref{theorem:conditionalFDPCLT}, \ref{theorem:conditionalFPR_CLT} and \ref{theorem:tau_star_0_res}}

This section contains results and definitions that we draw upon in order to prove our theorems. They come from \cite{Neumann13}, \cite{HallAndHyde}, \cite{AndrewsAndPollard} and \cite{VDV}.

\subsection{Theorems for proving  \texorpdfstring{$\big(W_{m,0}(\cdot),W_{m,1}(\cdot) \big)$}{theEP} converges in f.d.d.}

\cite{Neumann13} gives a CLT for triangular arrays of dependent random variables. Neumann's CLT in conjunction with the Cramér Wold device, will later be used to establish f.d.d. convergence of our empirical process $\big(W_{m,0}(\cdot),W_{m,1}(\cdot) \big)$ from~\eqref{eq:defws} to a joint Gaussian process. Below, we restate results from Theorem 2.1 in \cite{Neumann13} with notation convenient for us.

\begin{theorem}\label{theorem:NeumannCLT} (Neumann's CLT) Let $\{Y_{mi} \such 1 \leq i \leq m < \infty \}$ be a triangular array of centered random variables with $\sup_{m \in \mathbb{N}} \sum_{i=1}^m \e (Y_{mi}^2) < \infty$. Further suppose that both
\begin{equation}\label{eq:NeumannCond1}
    \sigma_m^2 \equiv \var \Big( \sum_{i=1}^m Y_{mi} \Big) \to \sigma^2 \in [0, \infty) \quad \text{as} \quad m \to \infty, \text{ and }
\end{equation} \begin{equation}\label{eq:NeumannCond2}
    \lim\limits_{m \to \infty} \sum_{i=1}^m \e \Big( Y_{mi}^2 I \{ \vert Y_{mi} \vert > \epsilon \} \Big)=0 \quad \text{for any } \epsilon>0.
\end{equation} Furthermore, suppose there exists a summable sequence of real numbers $\{ \theta_d \}_{d \in \mathbb{N}}$ such that for any $u,d,m \in \mathbb{N}$ and any $j_1,\dots, j_u, b_1,b_2 \in \mathbb{N}$ satisfying $1 \leq j_1 < j_2 < \dots < j_u < j_u+d =b_1 \leq b_2 \leq m$, the following two covariance upper bounds both hold for any measurable function $g : \real^u \to \real$ with $\sup_{x \in \real^u} \vert g(x) \vert \leq 1$:

\begin{enumerate}
    \item[(i)] $\big| \cov \big( g(Y_{mj_1}, \dots, Y_{mj_u} ) Y_{mj_u} , Y_{mb_1} \big) \big| \leq \big( \e(Y_{mj_u}^2) + \e(Y_{mb_1}^2)+1/m \big) \theta_d $, and
    \item[(ii)] $\big| \cov \big( g(Y_{mj_1}, \dots, Y_{mj_u} )  , Y_{mb_1}Y_{mb_2} \big) \big| \leq \big( \e(Y_{mb_1}^2) + \e(Y_{mb_2}^2)+1/m \big) \theta_d $.
\end{enumerate} Then $\sum_{i=1}^m Y_{mi} \xrightarrow{d} \dnorm(0,\sigma^2)$.

\end{theorem}

To check the covariance upper bound conditions in Neumann's CLT, it will be helpful to use Theorem A.5 in \cite{HallAndHyde}. Below, we restate Theorem A.5 in \cite{HallAndHyde} with notation convenient for us.

\begin{theorem}\label{theorem:HallandHydeA5}
    Suppose that $U$ and $V$ are random variables which are $\mathcal{G}$-measurable and $\mathcal{H}$-measurable, respectively, and that almost surely both $\vert U \vert \leq C_0$ and $\vert V \vert \leq C_1$. Then \begin{align}\label{eq:HallHydeIneq}
        \big| \cov(U,V) \big| \leq 4 C_0 C_1 \sup_{\substack{A_0 \in \mathcal{G} \\ A_1 \in \mathcal{H}}} \big| \Pr(A_0 \cap A_1) - \Pr (A_0) \Pr(A_1) \big|.
    \end{align}
\end{theorem}

\subsection{Results and definitions from \texorpdfstring{\cite{AndrewsAndPollard}}{Andrews and Pollard (1994)}}

In order to obtain an FCLT from f.d.d. convergence we will need to introduce some definitions and results from \cite{AndrewsAndPollard}. \cite{AndrewsAndPollard} use a chaining argument to provide an FCLT for bounded function classes. Their FCLT does not require independence.

We first introduce some definitions from their paper. Let $\{ \Xi_{mi} : 1 \leq i \leq m < \infty \}$ be a triangular array of $\cy$-valued random elements of a measurable space. Define $\{ \tilde{\alpha}(d) \}_{d=1}^{\infty}$ to be the sequence of $\alpha$-mixing coefficients for this array. We use $\mathcal{F}$ to denote a collection of bounded functions from $\cy$ to $\mathbb{R}$.

\begin{definition}\label{def:rho_semi_norm}
   The seminorm function for the array $\{ \Xi_{mi} : 1 \leq i \leq m < \infty \}$  of $\cy$-valued random variables is the function $\seminorm$ that maps any real valued function $f$ on $\cy$, to a real number via the following equation
\begin{equation}
    \seminorm(f)= \sup_{m,i} \sqrt{ \e \big( f( \Xi_{mi})^2 \big)}.
\end{equation} This $\seminorm(\cdot)$ can be used to define a seminorm on a collection of functions on $\cy$.
\end{definition}

Before stating their FCLT we also restate Definition 2.1 of \cite{AndrewsAndPollard} which gives the bracketing number.

\begin{definition}\label{def:Pollard_bracketing_numbers}
    Let $\mathcal{F}$ be a function class and $\seminorm$ be the seminorm function for an array of random variables $\{ \Xi_{mi} : 1 \leq i \leq m < \infty \}$, as in Definition \ref{def:rho_semi_norm}. For $\delta>0$ the bracketing number $N(\delta)=N(\delta, \mathcal{F},\seminorm)$ is the smallest natural number $N$ for which there exists functions $f_1,\dots,f_N \in \mathcal{F}$ and functions $b_1,\dots,b_N : \cy \rightarrow \mathbb{R}$ with $\seminorm(b_i) \leq \delta$ for all $i \in \{1,\dots,N\}$ such that for each $f \in \mathcal{F}$, there exists a $j\in\{1,\dots,N\}$ such that $\vert f -f_j \vert \leq b_j$. Here, $\vert f -f_j \vert \leq b_j$ means that $\vert f( \Xi)- f_j(\Xi) \vert \leq b_j(\Xi)$ for all $\Xi \in \cy$, and we will use this shorthand notation in the proof of Theorem \ref{theorem:FCLT_Pollard}. 

\end{definition}

To relate the function class $\mathcal{F}$ to an empirical process we introduce the following operator $\nu_m$.

\begin{definition}\label{def:indexing_operator}
Let $\mathcal{F}$ be a class of real valued functions whose domain contains the support of each individual entry in the array $\{ \Xi_{mi} \such 1 \leq i \leq m < \infty \}$. For any $m \in \mathbb{Z}_{+}$ and any $f \in \mathcal{F}$, define \begin{equation}
    \nu_m f \equiv \frac{1}{\sqrt{m}} \sum_{i=1}^m \big( f(\Xi_{mi})- \e \big( f(\Xi_{mi} ) \big) \big).
\end{equation} 
\end{definition}

This operator $\nu_m$ maps individual functions in $\mathcal{F}$ to centered and renormalized random variables. It defines an empirical process indexed by the elements of $\mathcal{F}$.
We are now ready to restate the FCLT given in Corollary 2.3 of \cite{AndrewsAndPollard}.

\begin{theorem}\label{theorem:FCLT_AndrewsPollard}
    Let $\{ \Xi_{mi} : 1 \leq i \leq m < \infty \}$ be a strongly mixing triangular array whose $\alpha$-mixing coefficients $\{ \tilde{\alpha}(d) \}_{d=1}^{\infty}$ satisfy 
    \begin{equation}
        \sum_{d=1}^{\infty} d^{Q-2} \tilde{\alpha}(d)^{\gamma/(Q+\gamma)} < \infty
    \end{equation}
    for some even integer $Q \geq 2$ and some $\gamma>0$. Let $\seminorm$ be the seminorm function for the array $\{ \Xi_{mi} : 1 \leq i \leq m < \infty \}$ as given in Definition \ref{def:rho_semi_norm}, and let $\mathcal{F}$ be a uniformly bounded class of real-valued functions whose bracketing numbers (Definition \ref{def:Pollard_bracketing_numbers}) satisfy \begin{equation}
        \int_0^1 x^{-\gamma/(2+\gamma)} N(x,\mathcal{F},\seminorm)^{1/Q} dx < \infty
    \end{equation} 
    for the same $Q$ and $\gamma$. Finally, let $\nu_m$ be the operator from Definition \ref{def:indexing_operator}. If for all $k \in \mathbb{Z}_+$ and $f_1,\dots,f_k \in \mathcal{F}$, $(\nu_m f_1,\dots,\nu_m f_k)$ converges to a multivariate Gaussian distribution as $m \to \infty$, then $\{\nu_m f \such f \in \mathcal{F} \}$ converges in distribution to a Gaussian process indexed by $\mathcal{F}$ with $\seminorm$-continuous sample paths.
\end{theorem}

We will now  clarify what $\seminorm$-continuous sample paths and ``convergence in distribution" in the above theorem mean. We start by clarifying what $\seminorm$-continuity of sample paths for a Gaussian process means.

\begin{definition}\label{def:AP_rho_continuous}
    Let $\mathcal{F}$ be a function class, let $\seminorm$ be a seminorm, and let $W_{\mathcal{F}}$ be a Gaussian process indexed by $\mathcal{F}$, defined on the probability space $\Omega$. Note that for each $f \in \mathcal{F}$, $W_{\mathcal{F}}(f)$ is a real valued random variable with a univariate normal distribution, and for a fixed $\omega \in \Omega$, $W_{\mathcal{F}}(f)(\omega) \in \real$ is a specific realization of that random variable. Fixing $\omega \in \Omega$, we say that $W_{\mathcal{F}}(\cdot)(\omega)$ has a $\seminorm$-continuous sample path at $f$ % I think we need this f
    if for all $\epsilon>0$, there exists an $\eta>0$ such that $\vert W_{\mathcal{F}}(f)(\omega) - W_{\mathcal{F}}(g)(\omega)\vert < \epsilon$ holds for any $g\in\cf$ with $\seminorm(f-g) < \eta$. We say $W_{\mathcal{F}}$ has $\seminorm$-continuous sample paths, if for every $\omega \in \Omega$ and $f \in \mathcal{F}$, $W_{\mathcal{F}}(\cdot)(\omega)$ has a $\seminorm$-continuous sample path at $f$.
\end{definition}

For a fixed $m$, the process $\{\nu_m f \such f \in \mathcal{F} \}$ is not guaranteed to be Borel measurable. Chapter 9 of \cite{PollardTextbook} raises caution about measurability issues of such empirical process and page 9 of \cite{dudley_1999} provides an example showing that the empirical process of IID $\dunif[0,1]$ random variables is not guaranteed to be measurable with respect to the $\sup$-norm topology. To skirt around these measurability issues, we will use a modified notion of convergence provided in Definition 9.1 of \cite{PollardTextbook}. Introducing the modified notion of convergence allows us to use the FCLT from \cite{AndrewsAndPollard} (reproduced in Theorem \ref{theorem:FCLT_AndrewsPollard}) and ultimately allows us to extend upon the results of \cite{DR16}, which use a different FCLT. 

We reproduce the modified definition of convergence in distribution below. Some readers might prefer to skip this definition and to instead treat everything as measurable so that the outer expectation and the expectation are the same and so that the standard notion of convergence in distribution is interchangeable with the modified definition.  

\begin{definition}\label{def:Pollard_weak_convergence}
    Let $\{X_n \}_{n=1}^{\infty}$ be a sequence of functions from a probability space $\Omega$ into a metric space $\mathcal{X}$ (not necessarily Borel measurable), and let $X$ be a Borel measurable map from $\Omega$ into $\mathcal{X}$. For any bounded real valued function $f$, $\e(f(X_n))$ is not necessarily well defined, so we define $\e^*$ to be an operator that gives the outer expectation of a real valued function on $\Omega$. In particular, if $H$ is a real valued function on $\Omega$, then
    $$\e^*(H) \equiv \inf \{ \e(U) \such H \leq U, \text{ and } U \text{ is measurable and integrable}\}.$$ 
    We say that $X_n$ converges in distribution to $X$, if for every bounded, uniformly continuous, real valued function $f$ on $\mathcal{X}$, $\lim_{n \rightarrow \infty} \e^*(f(X_n))= \e(f(X))$.
\end{definition}

Throughout the text we will use the operator $\xrightarrow{\mathcal{D}}$ in the statements such as $X_n \xrightarrow{\mathcal{D}} X$ to denote convergence in distribution in the sense of Definition \ref{def:Pollard_weak_convergence}.

\subsection{Results and definitions from \texorpdfstring{\cite{VDV}}{van der Vaart (1998)}}

In the proof of Theorems \ref{theorem:conditionalFDPCLT} and \ref{theorem:conditionalFPR_CLT}, we will rely on some definitions and theorems from \cite{VDV}. Most notably, we need to introduce the notion of Hadamard differentiability and the functional delta method from \cite{VDV}. The functional delta method is defined for an alternate definition of convergence in distribution, provided in Chapter 18.2 of \cite{VDV}, which turns out to be equivalent to 
\nocite{PollardTextbook} Pollard's (1990)
% there is also `possessive cite' but it doesn't work in every latex context
notion of convergence in distribution in Definition \ref{def:Pollard_weak_convergence}. We reproduce the alternate definition of convergence in distribution provided in Chapter 18.2 of \cite{VDV} below.

\begin{definition}\label{def:VDV_weak_convergence}
    Let $\{X_n \}_{n=1}^{\infty}$ be a sequence of random elements and $X$ be a Borel-measurable random variable. $X_n$ converges in distribution to $X$ in the sense of Chapter 18.2 in \cite{VDV}, if $ \lim_{n \rightarrow \infty} \e^*(f(X_n))= \e(f(X))$ for all bounded, continuous functions $f$, where $\e^*$ denotes outer expectation.

\end{definition}

\begin{remark}
    We can equivalently define $\xrightarrow{\mathcal{D}}$ to denote convergence in distribution in the sense of Definition \ref{def:VDV_weak_convergence} because this notion of convergence in distribution is equivalent to the one provided in Definition \ref{def:Pollard_weak_convergence}.
\end{remark}

\begin{proof}
     Definition \ref{def:Pollard_weak_convergence} states that $X_n$ converges in distribution to $X$ if and only if for all bounded, uniformly continuous functions $f$, $\lim_{n \rightarrow \infty} \e^*(f(X_n))= \e(f(X))$. Definition \ref{def:VDV_weak_convergence} states that $X_n$ converges in distribution to $X$ if and only if $ \lim_{n \rightarrow \infty} \e^*(f(X_n))= \e(f(X))$ for all bounded, continuous functions $f$ (which need not be uniformly continuous). Clearly, if convergence in distribution holds in the sense of Definition \ref{def:VDV_weak_convergence}, it will hold in the sense of Definition \ref{def:Pollard_weak_convergence}. To show the converse note that if $ \lim_{n \rightarrow \infty} \e^*(f(X_n))= \e(f(X))$ for all bounded, uniformly continuous functions $f$, then $ \lim_{n \rightarrow \infty} \e^*(f(X_n))= \e(f(X))$ for all bounded, Lipschitz functions $f$, but by the Portmanteau Lemma (Lemma 18.9 in \cite{VDV}) this is equivalent to  $\lim_{n \rightarrow \infty} \e^*(f(X_n))= \e(f(X))$ for all bounded, continuous functions $f$. The notions of convergence of distribution in Definitions \ref{def:Pollard_weak_convergence} and \ref{def:VDV_weak_convergence} are therefore equivalent.
\end{proof}

In our proofs of Theorems \ref{theorem:conditionalFDPCLT} and \ref{theorem:conditionalFPR_CLT}, it will also be helpful to introduce the alternate definition for convergence in probability of non-measurable real valued functions from a probability space, seen in Chapter 18.2 of \cite{VDV}, and to derive Slutsky's lemma for our alternate notions of convergence in distribution and convergence in probability.

Below, we reproduce the alternate notion of convergence in probability seen in Chapter 18.2 of \cite{VDV}.

\begin{definition}\label{def:VDV_convergence_in_prob}
    If $\{X_n \}_{n=1}^{\infty}$ is a sequence of functions from a probability space $\Omega$ into a metric space $(\mathcal{X},d)$, we say that $X_n$ converges in probability to $X$ if for all $\epsilon>0$, $ \lim_{n \rightarrow \infty} \mathbb{P}^*(d(X_n,X) > \epsilon)=0$. Here, $\mathbb{P}^*$ denotes outer probability measure since neither $X_n$ nor $d(X_n,X)$ needs to be Borel-measurable. Throughout the text, we will use the operator $\xrightarrow{\mathcal{P}}$ to denote convergence in probability in the sense of this definition.
\end{definition}

We will need to use Slutsky's lemma for our alternate definitions of convergence in probability and convergence in distribution, which we state and prove below by directly applying Theorems 18.10(v) and 18.11(i) from Chapter 18.2 of \cite{VDV}.

\begin{lemma}[Slutsky's lemma]\label{lemma:Slutsky_new}
     If $X_n \xrightarrow{\mathcal{D}} X$ and $Y_n \xrightarrow{\mathcal{P}} c$ for some constant $c$, then $X_n +Y_n \xrightarrow{\mathcal{D}} X+c$.
\end{lemma}

\begin{proof}
Suppose $X_n \xrightarrow{\mathcal{D}} X$ and $Y_n \xrightarrow{\mathcal{P}} c$. Then, by Theorem 18.10(v) in \cite{VDV}, $(X_n,Y_n) \xrightarrow{\mathcal{D}} (X,c)$. By the continuous mapping theorem (Theorem 18.11(i) in \cite{VDV}) since $(Z_1,Z_2) \mapsto Z_1+Z_2$ is continuous, it follows that $X_n +Y_n \xrightarrow{\mathcal{D}} X+c$.
\end{proof}

Since our proofs of Theorems \ref{theorem:conditionalFDPCLT} and \ref{theorem:conditionalFPR_CLT} rely heavily on the concept of Hadamard differentiability and on the functional delta method, we reproduce the definition of Hadamard differentiability as well as the functional delta method from Chapter 20.2 of \cite{VDV} below.

\begin{definition}[Hadamard differentiability]\label{def:Hadamard_Differentiability}
 Let $\mathbb{D}$, $\mathbb{F}$ be normed spaces, let $\phi: \mathbb{D}_{\phi} \rightarrow \mathbb{F}$ be a map defined on a subset $\mathbb{D}_{\phi} \subseteq \mathbb{D}$, and consider an element $\theta \in \mathbb{D}_{\phi}$. Further, let $\mathbb{D}_0$ be a different subset of $\mathbb{D}$. We say that $\phi$ is Hadamard differentiable at $\theta$ tangentially to the subset $\mathbb{D}_0$ if there exists a continuous, linear map $\phi_{\theta}': \mathbb{D}_0 \rightarrow \mathbb{F}$ such that for any $h \in \mathbb{D}_0$ and collection of elements of $\mathbb{D}$ $(h_t)_{t >0}$ satisfying both $\lim_{t \downarrow 0} \Vert h_t-h \Vert_{\mathbb{D}} =0$ and $\{ \theta+th_t \such t>0 \} \subseteq \mathbb{D}_{\phi}$, 
$$ \Bigl\Vert \frac{\phi(\theta+th_t) -\phi(\theta)}{t} - \phi_{\theta}'(h) \Bigr\Vert_{\mathbb{F}} \rightarrow 0 \quad \text{as} \ t \downarrow 0. $$ 
This map $\phi_{\theta}': \mathbb{D}_0 \rightarrow \mathbb{F}$, if it exists, is said to be the Hadamard derivative of the function $\phi$ at $\theta$ tangentially to the subset $\mathbb{D}_0$.

\end{definition}

\begin{theorem}[Functional Delta Method]\label{theorem:functional_delta_method}
     Let $\mathbb{D}$ and $\mathbb{F}$ be normed linear spaces and let $\phi: \mathbb{D}_{\phi} \subseteq \mathbb{D} \rightarrow \mathbb{F}$ be Hadamard differentiable at $\theta$ tangentially to $\mathbb{D}_0 \subseteq \mathbb{D}$ with Hadamard derivative at $\theta$ tangentially to $\mathbb{D}_0$ given by the linear map $\phi_{\theta}': \mathbb{D}_0 \rightarrow \mathbb{F}$. Let $\{T_n \}_{n=1}^{\infty}$ be a sequence of maps from our probability space to $\mathbb{D}_{\phi}$ such that $r_n(T_n-\theta) \xrightarrow{\mathcal{D}} T$ for some sequence of numbers $r_n \rightarrow \infty$ and a $\mathbb{D}_0$ valued random variable $T$. Then $r_n \big( \phi(T_n)-\phi(\theta)\big) \xrightarrow{\mathcal{D}} \phi_{\theta}'(T)$.
\end{theorem}

\begin{proof}
See Theorem 20.8 in \cite{VDV}.
\end{proof}

\section{Proof of Theorems \ref{theorem:conditionalFDPCLT}, \ref{theorem:conditionalFPR_CLT}, and \ref{theorem:tau_star_0_res}}

In this section, we will assume the setting of Theorem \ref{theorem:conditionalFDPCLT}. Further, to simplify notation and analysis, we will fix the realization of the latent factor $\bswl$, so that all probabilities, expectations and outer expectations assume a fixed $\bswl$ and are taken with respect to the probability measure and $\sigma$-algebra generated by $\{ (\varepsilon_{mi},H_{mi}), 1 \leq i \leq m < \infty \}$.

\subsection{Helpful results for proving Theorems \ref{theorem:conditionalFDPCLT} and \ref{theorem:conditionalFPR_CLT}}
In this subsection, we develop results to prove Theorems \ref{theorem:conditionalFDPCLT}, \ref{theorem:conditionalFPR_CLT} and \ref{theorem:tau_star_0_res}. We start with a helpful lemma.

\begin{lemma}\label{lemma:Lipschitz}
    Under the conditions of Theorem \ref{theorem:conditionalFDPCLT}, there is % ``exists -> is'' avoided a line split inside C<\infty
    a constant $C<\infty$ such that for $r \in \{0,1\}$ and $1 \leq i \leq m < \infty$
    \begin{equation}
           \sup_{(t,s) \in [a,b]^2} \Big| \gamma_{mir}(t)-\gamma_{mir}(s) \Big| \leq C \vert t-s \vert.
    \end{equation}
\end{lemma}

\begin{proof}
     For any $r \in \{0,1 \}$ and
     $1 \leq i \leq m < \infty$, note that the derivative of $\gamma_{mir}(t)$ with respect to $t$ 
%is given by 
% \begin{equation*}
%     \gamma_{mir}'(t)=\frac{1}{\varphi(\Phi^{-1}(t)) \sqrt{1- \Vert \bm{L}_{mi} \Vert_2^2}}  \varphi \Bigg( \frac{\bar{\Phi}^{-1}(t) -\bm{L}_{mi}^\tran \bswl -\mu_Ar  }{\sqrt{1- \Vert \bm{L}_{mi} \Vert_2^2 }} \Bigg)
% \end{equation*}
satisfies
$$\vert \gamma_{mir}'(t) \vert \leq \frac{\varphi(0)}{\varphi(\Phi^{-1}(t)) \sqrt{1- S_L}} \leq \frac{\varphi(0)}{\sqrt{1-S_L}} \max \Bigl\{ \frac{1}{\varphi(\Phi^{-1}(a))},\frac{1}{\varphi(\Phi^{-1}(b))} \Bigr\} \equiv C$$
by the argument used in the proof of Proposition~\ref{prop:G_is_continuous}.
%Noting that $\varphi$ is the standard Gaussian p.d.f so it takes a maximum value of $\frac{1}{\sqrt{2 \pi}} \leq 1$ and recalling by Condition~\ref{cond:boundedL} that $S_L = \sup_{1 \leq i \leq m < \infty} \Vert \bm{L}_{mi} \Vert_2^2 < 1$. It thus follows that for any $t \in [a,b]$ $
%Since the derivative of $\gamma_{mir}$ has absolute first derivative on $[a,b]$ bounded by $C$ it follows that $\gamma_{mir}$ is $C$-Lipschitz continuous on $[a,b]$ (as a consequence of the mean value theorem). 
Since $C$ does not depend on $r$, $i$ or $m$, this completes the proof.
\end{proof}

We now define the joint Gaussian process on $[a,b]_2$ to which $\big( W_{m,0}(\cdot), W_{m,1}(\cdot) \big)$ converges in distribution and start by proving f.d.d. convergence.

\begin{definition}
We define $(\tilde{W}_0,\tilde{W}_1)$ be a joint Gaussian process on $[a,b]_2$ with mean zero and joint covariance kernel $c$.
\end{definition}

We note that $(\tilde{W}_0,\tilde{W}_1)$ is well defined because $c$ is a symmetric and positive semidefinite joint covariance kernel.

\begin{proposition}\label{prop:fdd_convergence}
Under the conditions of Theorem \ref{theorem:conditionalFDPCLT}, $( W_{m,0}, W_{m,1} ) \xrightarrow{\fdd} (\tilde{W}_0,\tilde{W}_1)$.
\end{proposition}

\begin{proof}
    Fix any $k \in \mathbb{N}$ and fix any $t_1,\dots, t_k \in [a,b]$ and $r_1,\dots ,r_k \in \{0,1 \}$. Define $\Sigma \in \real^{k \times k}$ to be the covariance matrix given by $\Sigma_{ll'}=c^{(r_l,r_{l'})}(t_l ,t_{l'})$ for $l,l' \in \{1,\dots,k \}$. To show f.d.d. convergence we must show that $\big(  W_{m,r_1}(t_1),\dots, W_{m,r_k}(t_k) \big) \xrightarrow{d} \dnorm (0, \Sigma)$. By the Cramér-Wold device, it suffices to show that  $\sum_{l=1}^k a_l  W_{m,r_l}(t_l) \xrightarrow{d} \dnorm (0, \bm{a}^\tran \Sigma \bm{a} )$ for any $\bm{a} \in \real^k$.

    Now fix any $\bm{a} \in \real^k$. To show $\sum_{l=1}^k a_l  W_{m,r_l}(t_l) \xrightarrow{d} \dnorm (0, \bm{a}^\tran \Sigma \bm{a} )$, we will use Neumann's CLT (Theorem \ref{theorem:NeumannCLT}). For $1 \leq i \leq m < \infty$ define,
    $$\begin{aligned} Y_{mi} \equiv  \frac{1}{\sqrt{m}}\sum_{l=1}^k  \Big( H_{mir_l} I \{ \bar{\Phi}(\err_{mi}+\bm{L}_{mi}^\tran \bswl +\mu_A r_l ) \leq t_l \} -\pi_{r_l}^{(m)} \gamma_{mir_l}(t_l) \Big) a_l. \end{aligned}$$
    Clearly $\e(Y_{mi})=0$ for all $m,i$. Note that $\vert Y_{mi} \vert  \leq  k \Vert \bm{a} \Vert_{\infty}/\sqrt{m}$ which implies that $\e(Y_{mi}^2) \leq k^2 \Vert \bm{a} \Vert_{\infty}^2/m$ and, therefore, $\sup_{m \in \mathbb{N}} \sum_{i=1}^m \e(Y_{mi}^2) \leq k^2 \Vert \bm{a} \Vert_{\infty}^2 < \infty$. Therefore, to apply Theorem \ref{theorem:NeumannCLT} to the triangular array $\{Y_{mi} \such 1 \leq i \leq m < \infty \}$, we must check \eqref{eq:NeumannCond1}, \eqref{eq:NeumannCond2}, as well as the existence of a summable $\{\theta_d\}_{d \in \mathbb{N}}$ satisfying the desired properties.
    
    To check \eqref{eq:NeumannCond1} note that because $\sum_{i=1}^m Y_{mi} =  \sum_{l=1}^k a_l W_{m,r_l}(t_l)$,
    \begin{align*}
    \sigma_m^2 \equiv  \var \Big( \sum_{i=1}^m Y_{mi} \Big)=\sum_{l=1}^k  \sum_{l'=1}^k a_l a_{l'}  c_m^{(r_l,r_{l'})}(t_l,t_{l'})= \bm{a}^\tran \Sigma \bm{a} +o(1), 
    \end{align*} where the last equality holds by \eqref{eq:defc} under Condition \ref{cond:cov_kernel_converges} and by our definition of $\Sigma$. Therefore, letting $\sigma^2 \equiv \bm{a}^\tran \Sigma \bm{a} \in [0,\infty)$, \eqref{eq:NeumannCond1} holds. It is easy to see that \eqref{eq:NeumannCond2} holds because for all $i$ and $m$,  $\vert Y_{mi} \vert  \leq  k \Vert \bm{a} \Vert_{\infty}/\sqrt{m}$, where $k\Vert \bm{a} \Vert_{\infty}$ does not depend on $m$ or $i$. 
    
    Now for each $d \in \mathbb{N}$ define $\theta_d \equiv 4 k^2 \Vert \bm{a} \Vert_{\infty}^2 \alpha(d)$, where $\alpha(d)$ is defined in Equation \eqref{eq:def_alpha_mixing_coefficient}. By Condition \ref{cond:niceQandgamma}, $\sum_{d=1}^{\infty} \alpha(d) < \infty$ and therefore $\{\theta_d\}_{d \in \mathbb{N}}$ is a summable sequence of real numbers. Fix any $u,d,m \in \mathbb{N}$ and any $j_1,\dots, j_u, b_1,b_2 \in \mathbb{N}$ satisfying $1 \leq j_1 < j_2 < \dots < j_u < j_u+d =b_1 \leq b_2 \leq m$. Further, fix any measurable function $g: \real^u \to \real$ such that $\sup_{x \in \real^u} \vert g(x) \vert \leq 1$ and we will show that covariance inequalities (i) and (ii) from Theorem \ref{theorem:NeumannCLT} hold. To prove covariance inequality (i), let $U = g(Y_{mj_1}, \dots , Y_{mj_u})Y_{mj_u}$ and $V=Y_{mb_1}$. Clearly, both $\vert U \vert \leq k \Vert \bm{a} \Vert_{\infty}/\sqrt{m}$ and $\vert V \vert \leq k \Vert \bm{a} \Vert_{\infty}/\sqrt{m}$ almost surely. Now recalling the $\sigma$-algebras defined in Section \ref{sec:mixing_conditions} note that $U$ is $\mathcal{A}_1^{j_u}(m)$-measurable and $V$ is $\mathcal{A}_{j_u+d}^{\infty}(m)$-measurable. Thus by Theorem \ref{theorem:HallandHydeA5}, $$\begin{aligned} 
        \big| \cov(U,V) \big| \leq & \frac{4 k^2 \Vert \bm{a} \Vert_{\infty}^2 }{m} \sup_{\substack{A_0 \in \mathcal{A}_1^{j_u}(m) \\ A_1 \in \mathcal{A}_{j_u+d}^{\infty}(m)}} \big| \Pr(A_0 \cap A_1) - \Pr (A_0) \Pr(A_1) \big| \\ \leq & 4 k^2 \Vert \bm{a} \Vert_{\infty}^2 \alpha(d) /m  \\  \leq &   \big( \e(Y_{mj_u}^2) + \e(Y_{mb_1}^2)+1/m \big) \theta_d,   
    \end{aligned}$$ where the second inequality uses Equation \eqref{eq:def_alpha_mixing_coefficient} and the third inequality uses nonnegativity of $\e(Y_{mi}^2)$ and the definition of $\theta_d$. Therefore, covariance inequality (i) of Theorem \ref{theorem:NeumannCLT} holds. To show covariance inequality (ii) holds, we use the same argument except we let $U=g(Y_{mj_1}, \dots , Y_{mj_u})$ and $V=Y_{mb_1}Y_{mb_2}$. Note that both $\vert U \vert \leq 1$ and $\vert V \vert \leq k^2 \Vert \bm{a} \Vert_{\infty}^2/m$ almost surely and that again $U$ is  $\mathcal{A}_1^{j_u}(m)$-measurable while $V$ is $\mathcal{A}_{j_u+d}^{\infty}(m)$-measurable. Therefore, by Theorem \ref{theorem:HallandHydeA5} and the definitions of $\alpha(d)$ and $\theta_d$, $$ \big| \cov(U,V) \big| \leq  \frac{4 k^2 \Vert \bm{a} \Vert_{\infty}^2 \alpha(d) }{m} \leq  \big( \e(Y_{mb_1}^2) + \e(Y_{mb_2}^2)+1/m \big) \theta_d.$$
    
    We have thus shown that with the summable sequence $\{\theta_d \}_{d \in \mathbb{N}}$, covariance inequality (i) and (ii) of Theorem \ref{theorem:NeumannCLT} hold for any fixed $u,d,m \in \mathbb{N}$, any $j_1,\dots, j_u, b_1,b_2 \in \mathbb{N}$ satisfying $1 \leq j_1 < j_2 < \dots < j_u < j_u+d =b_1 \leq b_2 \leq m$, and any $g: \real^u \to \real$ such that $\sup_{x \in \real^u} \vert g(x) \vert \leq 1$. Thus by applying Neumann's CLT (Theorem \ref{theorem:NeumannCLT}), $$\sum_{l=1}^k a_l  W_{m,r_l}(t_l) = \sum_{i=1}^m Y_{mi} \xrightarrow{d} \dnorm(0, \bm{a}^\tran \Sigma \bm{a}). $$

\end{proof}

Now that we have established $(W_{m,0},W_{m,1}) \xrightarrow{\fdd} (\tilde{W}_0,\tilde{W}_1)$ we can derive an FCLT for $(W_{m,0},W_{m,1})$ by applying a result from \cite{AndrewsAndPollard}, which is based on a chaining argument. The FCLT uses a slightly modified definition of weak convergence which can be applied to a sequence of non-Borel measurable $D[a,b]_2$-valued random variables (see Definition \ref{def:Pollard_weak_convergence} or equivalently Definition \ref{def:VDV_weak_convergence}.)

\begin{theorem}\label{theorem:FCLT_Pollard}
    Under the conditions of Theorem \ref{theorem:conditionalFDPCLT}, $(W_{m,0},W_{m,1}) \xrightarrow{\mathcal{D}} (W_0,W_1)$, where $(W_0,W_1)$ is a joint Gaussian process with mean zero and joint covariance kernel $c$, and $(W_0,W_1) \in C[a,b]_2$ almost surely.
\end{theorem}

\begin{proof} 
The proof is lengthy and so we split it into a sequences of steps as follows:
\begin{enumerate}
    \item [\quad\sl 1)] Defining the function class $\cf$ and seminorm $\seminorm$ that we need in order to use Theorem~\ref{theorem:FCLT_AndrewsPollard}, the FCLT from \cite{AndrewsAndPollard},
    \item [\quad\sl 2)]
    upper bounding the bracketing number $N(x,\cf,\seminorm)$,
    \item [\quad\sl 3)] checking the FCLT conditions,
    \item [\quad\sl 4)] applying the FCLT,
    \item [\quad\sl 5)]defining limiting processes $W_0$ and $W_1$, and finally
    \item [\quad\sl 6)] verifying that $(W_0,W_1)\in C[a,b]_2$.
\end{enumerate}
    \paragraphLocal{Defining the relevant function class, $\mathcal{F}$, and seminorm, $\seminorm$.}
    First, for each $m,i$ recall that 
    $\tilde{\err}_{mi}=\err_{mi}/\sqrt{1-\Vert L_{mi} \Vert_2^2}$ has unit variance. Let $\{\tilde{\xi}_{mi} \such 1 \leq i \leq m < \infty \}$ be an expansion of our triangular array where $\tilde{\xi}_{mi}=(\tilde{\err}_{mi},H_{mi},\bm{L}_{mi},\bswl)$. 
    Because $\bm{L}_{mi}$ and $\bswl$ are deterministic and $\Vert \bm{L}_{mi} \Vert_2^2$ is bounded away from 1 by Condition~\ref{cond:boundedL}, the arrays $\{\tilde{\xi}_{mi} \such 1 \leq i \leq m < \infty \}$ and $\{ (\err_{mi}, H_{mi}) \such 1 \leq i \leq m < \infty \}$ have the same $\alpha$-mixing coefficients. %
    These coefficients $\{\alpha(d) \}_{d=1}^{\infty}$ satisfy Condition~\ref{cond:niceQandgamma}. 
    Now define $\mathcal{Y}=\real \times \{0,1 \} \times \real^k \times \real^k$, and define $\tilde{f}: \mathcal{Y} \times \big( [a,b] \times \{0,1 \} \big) \rightarrow [0,1]$ via
    $$ \tilde{f} \big( (\tilde{\err},H, \bm{L}, \bswl), (t,r) \big)= I \{H=r \} I \biggl\{ \bar{\Phi}( \tilde{\err} ) \leq \bar{\Phi} \bigg( \frac{\bar{\Phi}^{-1}(t) -\bm{L}^\tran  \bswl -\mu_A r  }{\sqrt{1- \Vert \bm{L} \Vert_2^2 }} \bigg)  \biggr\}.$$ 
    For $(t,r) \in [a,b] \times \{0,1\}$, note that $\tilde{f} \big( \tilde{\xi}_{mi},(t,r) \big)=H_{mir} I \{ \bar{\Phi}(\tilde{\err}_{mi}) \leq \gamma_{mir}(t) \}$. 

    The function class we will use includes the various restrictions of $\tilde{f}$ to a fixed $t \in [a,b]$ and $r \in \{0,1\}$ while $\tilde\err$, $H$, $\bm{L}$ and $\bswl$ vary: 
    \begin{equation}
        \mathcal{F} \equiv \bigl\{ \tilde{f} \big( \cdot, (t,r) \big) \such (t,r) \in [a,b] \times \{0,1\} \bigr\}.
    \end{equation} Letting $\seminorm$ be the seminorm function for the array $\{\tilde{\xi}_{mi} \such 1 \leq i \leq m < \infty \}$ given by Definition \ref{def:rho_semi_norm}, observe that for $t,s \in [a,b]$ and $r \in \{0,1 \}$ a simple computation gives
    \begin{equation}\label{eq:inequality_on_function_class}
    \seminorm \big(  \tilde{f} ( \cdot, (t,r) ) -   \tilde{f} ( \cdot, (s,r) )  \big)  =  \sup_{m,i} \sqrt{ \pi_r^{(m)} \vert  \gamma_{mir}(t) - \gamma_{mir}(s)\vert} \leq  \sqrt{C \vert t-s \vert},
    \end{equation} where the inequality follows from Lemma \ref{lemma:Lipschitz}. 
    
    \paragraphLocal{Upper bounding the bracketing number $N(x,\mathcal{F},\seminorm)$.} 
    To apply the FCLT (Theorem \ref{theorem:FCLT_AndrewsPollard}), we must first compute an upper bound on the bracketing number $N(x,\mathcal{F},\seminorm)$ for each $x>0$. 
    Fixing $x > 0$, take $N_* \equiv \lceil {C(b-a)}/{x^2} \rceil$. 
    For each $l \in \{1,\dots,N_*,N_*+1 \}$ and $r \in \{0,1 \}$ define 
    $$f_{l,r}= \tilde{f} \big( \cdot, (a+ (b-a)(l-1)/N_* , r) \big),$$ 
    and 
    %further for each $l \in \{1,\dots,N_* \}$ and $r \in \{0,1 \}$ 
    for $1\le l\le N_*$ define $b_{l,r}= f_{l+1,r}-f_{l,r}$. 
    Note for a given $f \in \mathcal{F} \setminus \{f_{1,0},f_{1,1} \}$, since $f=\tilde{f}(\cdot, (t_*,r_*) )$ for some $(t_*,r_*) \in (a,b] \times \{0,1 \}$ we can choose $r=r_*$ and $l \in \{1,\dots,N_* \}$ such that $t_* \in \big(a+{(b-a)(l-1)}/{N_*}, a+ {(b-a)l}/{N_*} \big]$, and it is easy to check that $\vert f -f_{l,r} \vert \leq b_{l,r}$. Adding the trivial cases of $f \in \{f_{1,0},f_{1,1} \}$ it follows that for any $f \in \mathcal{F}$, there exists an $l \in \{1,\dots,N_* \}$ and $r \in \{0,1 \}$ such that $\vert f - f_{l,r} \vert \leq b_{l,r}$. Furthermore, for any $l \in  \{1,\dots,N_* \}$ and $r \in \{0,1 \}$, by applying Inequality \eqref{eq:inequality_on_function_class}, $$\seminorm(b_{l,r})=\seminorm(f_{l+1,r}-f_{l,r} ) \leq \sqrt{C(b-a)/N_*} \leq x.$$ 
    Thus, we have shown that there exists functions $f_{1,0},\dots,f_{N_*,0},f_{1,1},\dots,f_{N_*,1} \in \mathcal{F}$ and functions $b_{1,0},\dots,b_{N_*,0}, b_{1,1},\dots,b_{N_*,1} : \mathcal{Y} \rightarrow \real$, satisfying $\seminorm(b_{l,r}) \leq x$ for all $(l,r) \in \{1,\dots,N_* \} \times \{0,1 \}$, such that for any $f \in \mathcal{F}$ there exists an $(l,r) \in \{1,\dots,N_* \} \times \{0,1 \}$ for which $\vert f - f_{l,r} \vert \leq b_{l,r}$. Hence, by Definition \ref{def:Pollard_bracketing_numbers}, $N(x, \mathcal{F},\seminorm) \leq 2N_*= 2 \lceil {C(b-a)}/{x^2} \rceil$, and this argument holds for any $x>0$.
    
    \paragraphLocal{Checking the conditions for the FCLT of \cite{AndrewsAndPollard} (Theorem \ref{theorem:FCLT_AndrewsPollard}.)} Letting $Q > 2$ be the even integer and $\gamma>0$ be the number guaranteed by Condition~\ref{cond:niceQandgamma}, we find that \begin{equation}\label{eq:cond2PollardAndrews}
        \int_0^1 x^{-\frac{\gamma}{2+\gamma}} N(x,\mathcal{F}, \seminorm)^{\frac{1}{Q}} dx \leq \int_0^1 x^{-\frac{\gamma}{2+\gamma}} \Big( 2 \Bigl\lceil \frac{C(b-a)}{x^2} \Bigr\rceil \Big)^{\frac{1}{Q}} dx < \infty.
    \end{equation} 
    The integral above is finite because ${\gamma}/({2+\gamma})+{2}/{Q} <1$. 
    The $\alpha$-mixing coefficients of $\{\tilde{\xi}_{mi},\ 1\le i\le m<\infty\}$ are $\alpha(d)$ which satisfy Condition~\ref{cond:niceQandgamma}.
Therefore,
\begin{equation}\label{eq:cond1PollardAndrews} \sum_{d=1}^{\infty} d^{Q-2} \alpha(d)^{\frac{\gamma}{Q+\gamma}} < \infty.
    \end{equation}
    
    \paragraphLocal{Applying the FCLT of \cite{AndrewsAndPollard} (Theorem \ref{theorem:FCLT_AndrewsPollard}.)} By \eqref{eq:cond2PollardAndrews} and \eqref{eq:cond1PollardAndrews}, since $\mathcal{F}$ is a uniformly bounded class of real-valued functions, and since $\{\tilde{\xi}_{mi}, 1 \leq i \leq m < \infty \}$ is strongly mixing triangular array, we are in the setting of Theorem \ref{theorem:FCLT_AndrewsPollard}. To obtain an FCLT it therefore remains to check that for any $f_1,\dots,f_k \in \mathcal{F}$, $(\nu_m f_1,\dots,\nu_m f_k)$ converges in distribution to a multivariate Gaussian as $m \rightarrow \infty$, where $\nu_m$ is the operator from Definition \ref{def:indexing_operator}. This is indeed the case because for any $f \in \mathcal{F}$ there exists a $(t,r) \in[a,b] \times \{0,1\}$ such that $\nu_m f = W_{m,r}(t)$ and because by Proposition \ref{prop:fdd_convergence} $(W_{m,0},W_{m,1})$ converges in f.d.d to a Gaussian process on $[a,b] \times \{0,1 \}$. Thus applying Theorem \ref{theorem:FCLT_AndrewsPollard}, it follows that the stochastic process $\{ \nu_m f \such f \in \mathcal{F} \}$ converges in distribution to a Gaussian process indexed by $\mathcal{F}$ which has $\seminorm$-continuous sample paths in the sense of Definition \ref{def:AP_rho_continuous}, where convergence in distribution is in the sense of Definition \ref{def:Pollard_weak_convergence}.
    
    \paragraphLocal{Defining $(W_0,W_1)$ in terms of $W_{\mathcal{F}}$.} Let $W_{\mathcal{F}}$ be the Gaussian process on $\mathcal{F}$ to which $\{ \nu_m f \such f \in \mathcal{F} \}$ converges. For $(t,r) \in[a,b] \times \{0,1\}$ we define $W_r(t) \equiv W_{\mathcal{F}} \big( \tilde{f}(\cdot, (t,r)) \big)$. By considering the correspondence between $\mathcal{F}$ and $[a,b] \times \{0,1\}$ where for any $f \in \mathcal{F}$, $\nu_m f = W_{m,r}(t)$ for some unique $(t,r) \in[a,b] \times \{0,1\}$ because $\{ \nu_mf \such f \in \mathcal{F} \} \xrightarrow{\mathcal{D}} W_{\mathcal{F}}$ it is clear that $( W_{m,0}, W_{m,1}) \xrightarrow{\mathcal{D}} (W_0,W_1)$. This of course implies that $( W_{m,0}, W_{m,1}) \xrightarrow{\fdd} (W_0,W_1)$. 
    But since by Proposition \ref{prop:fdd_convergence}, $( W_{m,0}, W_{m,1}) \xrightarrow{\fdd} (\tilde{W}_0,\tilde{W}_1)$ it follows that $(W_0,W_1)$ has the same finite dimensional distribution as $(\tilde{W}_0,\tilde{W}_1)$, so $(W_0,W_1)$ is indeed a joint Gaussian process with mean zero and joint covariance kernel $c$.
    
    \paragraphLocal{Checking that $(W_0,W_1) \in C[a,b]_2$ always, using $\seminorm$-continuity of $W_{\mathcal{F}}$.} 
    It remains to show that $(W_0,W_1) \in C[a,b]_2$ always. 
    To check this let $\Omega$ be the probability space on which $\{ \tilde{\xi}_{mi} \}$ are defined. 
    Fix any $\omega \in \Omega$, and we will show that $\big(W_0(\cdot)(\omega), W_1(\cdot)(\omega) \big) \in C[a,b]_2$. To do this, further fix any $t \in [a,b]$, $r \in \{0,1 \}$ and $\epsilon>0$. 
    Let $g= \tilde{f}(\cdot,(t,r))$. 
    Since $W_{\mathcal{F}}(\cdot)(\omega)$ is $\seminorm$-continuous at $g$, by Definition \ref{def:AP_rho_continuous}, there exists an $\eta>0$ such that if $f \in \mathcal{F}$ with $\seminorm(f-g)< \eta$, then $\vert W_{\mathcal{F}}(f)(\omega) - W_{\mathcal{F}}(g)(\omega)\vert < \epsilon$. Taking such an $\eta$, define $\eta_*=\eta^2/C$. 
    We must show that $\vert W_r(t)(\omega)- W_r(s)(\omega) \vert < \epsilon$ for any $s \in [a,b]$ with $\vert t-s \vert< \eta_*$. 
    Let $s \in [a,b]$ with $\vert t-s \vert < \eta_*$. 
    Define $h=\tilde{f}(\cdot, (s,r) ) \in \mathcal{F}$. By \eqref{eq:inequality_on_function_class},  $$\seminorm(h-g) \leq \sqrt{C \vert t-s \vert} < \sqrt{C \eta_*} = \eta$$ which, by a previous claim, further implies $\vert W_{\mathcal{F}}(h)(\omega) - W_{\mathcal{F}}(g)(\omega)\vert < \epsilon$. Now by the definition of $W_r$, we get $\vert W_r (s)(\omega) - W_r (t)(\omega) \vert < \epsilon$. 
    Thus, we have shown that for any $\epsilon>0$ there exists $\eta_*>0$ such that if $s \in [a,b]$ and  $\vert t-s \vert < \eta_*$ then $\vert W_r ( t)(\omega) - W_r ( s)(\omega) \vert < \epsilon$. 
    Hence, $W_r(\cdot)(\omega)$ is continuous at $t$. 
    Since this argument holds for any $(t,r) \in [a,b] \times \{0,1\}$,  $\big( W_0(\cdot)(\omega) ,W_1( \cdot)(\omega) \big) \in C[a,b]_2$. 
    Since this holds for any $\omega \in \Omega$,  $( W_0 ,W_1) \in C[a,b]_2$ always.
   
\end{proof}

With some manipulation of the previous result, we can obtain the following result, which sets us up for using the functional delta method (Theorem \ref{theorem:functional_delta_method}).

\begin{corollary}\label{cor:VDV_FCLT}
    Under the conditions of Theorem~\ref{theorem:conditionalFDPCLT}, $(\hat{W}_{m,0},\hat{W}_{m,1}) \xrightarrow{\mathcal{D}} (W_0,W_1)$, where $(W_0,W_1)$ is a joint Gaussian process with mean zero and joint covariance kernel $c$, and $(W_0,W_1) \in C[a,b]_2$ almost surely.
\end{corollary}

\begin{proof}

   Observe that $$(\hat{W}_{m,0},\hat{W}_{m,1})=(W_{m,0},W_{m,1})+ \big( \sqrt{m} (F_{m,0}-F_0),\sqrt{m} (F_{m,1}-F_1) \big)$$ 
   where by Condition~\ref{cond:F0F1_fast_unif_conv}, the right hand term deterministically converges in $C[a,b]_2$ with respect to the $\sup$-norm to $(0,0) \in C[a,b]_2$, and hence, it also converges in probability to $(0,0)$. 
   Therefore, by Slutsky's lemma (Lemma \ref{lemma:Slutsky_new}) and Theorem \ref{theorem:FCLT_Pollard}, $(\hat{W}_{m,0},\hat{W}_{m,1}) \xrightarrow{\mathcal{D}} (W_0,W_1)$. 
   Also by Theorem \ref{theorem:FCLT_Pollard}, $(W_0,W_1) \in C[a,b]_2$ almost surely and $(W_0,W_1)$ is a joint Gaussian process with mean zero and joint covariance kernel $c$.
    \end{proof}

To derive CLTs for $\FDP_m$ and $V_m/m$ from an FCLT for $(\hat{W}_{m,0},\hat{W}_{m,1})$ we must apply the functional delta method. To make this precise, we define some functions.
First, we define $\mathcal{T} : D[a,b] \rightarrow [a,b]$ via 
\begin{equation}\label{eq:mathcalTDef} \mathcal{T}(H) = \sup \{ u \in [a,b] \such H(u) \geq {u}/{q} \} \end{equation} for all $H \in D[a,b]$. 
The Simes point is
%By our definition of $\tau_*$, it is easy to see that 
$\tau_* = \mathcal{T} (G)$. Next, define $\mathcal{R}: D[a,b]_2 \rightarrow \real$ as 
\begin{equation}\label{eq:mathcalRDef} \mathcal{R}(H_0,H_1)= (H_0+H_1) \big( \mathcal{T} (H_0+H_1) \big)\end{equation} 
for all $(H_0,H_1) \in D[a,b]_2$. 
Now for $(H_0,H_1)\in D[a,b]_2$, define real-valued FPR and FDP functions
\begin{align}
    \label{eq:PsiFPRDef} 
    \Psi^{(\FPR)} (H_0,H_1)&=H_0\big( \mathcal{T}(H_0+H_1)\big),\quad\text{and}\\ \label{eq:PsiFDPDef}
\Psi^{(\FDP)}(H_0,H_1)&=\frac{H_0 \big( \mathcal{T}(H_0+H_1) \big)}{(H_0+H_1) \big( \mathcal{T}(H_0+H_1) \big)}= \frac{\Psi^{(\FPR)} (H_0,H_1)}{\mathcal{R}(H_0,H_1)}. \end{align}

Our goal is to compute the Hadamard derivative of both $\Psi^{(\FPR)}$ and $\Psi^{(\FDP)}$ at the point $(F_0,F_1) \in C[a,b]_2$ tangentially to $C[a,b]_2$. 
While the supplement of \cite{DR16} calculates the Hadamard derivative of $\Psi^{(\FDP)}$ at $(F_0,F_1)$ in the case where $F_0(t)=\pi_0 t$, our Hadamard derivative calculation comes out differently because, in our setting, $F_0$ is not necessarily linear with slope $\pi_0$. 
To compute the Hadamard derivative for $\Psi^{(\FPR)}$ and $\Psi^{(\FDP)}$ at the point $(F_0,F_1)$ tangentially to $C[a,b]_2$, it is helpful to first calculate the Hadamard derivative of $\mathcal{T}$ at $(F_0,F_1)$ tangentially to $C[a,b]_2$. To compute the latter, we mimic the calculations seen in \cite{Neuvial2008}, reproving some of their results in order to generalize from functions defined on $D[0,1]$ to functions defined on $D[a,b]$ and to situations where $F_0$ is not linear with slope $\pi_0$. We chose to reproduce the proof for Hadamard differentiability of $\mathcal{T}$ at $(F_0,F_1)$ tangentially to $C[a,b]_2$ rather than reference to the result from \cite{Neuvial2008} in order to avoid generalization and notation misinterpretation errors and in order to familiarize the reader with Hadamard derivative calculations which, regardless, are needed for $\Psi^{(\FPR)}$ in our setting. The proofs of Hadamard differentiability and the calculations of the Hadamard derivative of both $\Psi^{(\FPR)}$ and $\Psi^{(\FDP)}$ at the point $(F_0,F_1) \in C[a,b]_2$ tangentially to $C[a,b]_2$ are carried out in Section \ref{sec:HadamardDerivComp}. Some readers may prefer to skip Section \ref{sec:HadamardDerivComp} because it only contains Hadamard differentiability proofs and because the proof methodology is similar to that in \cite{Neuvial2008}.

With the Hadamard derivatives of $\Psi^{(\FDP)}$ and $\Psi^{(\FPR)}$ at $(F_0,F_1)$, we can apply the functional delta method (Theorem \ref{theorem:functional_delta_method}) to the result of Corollary \ref{cor:VDV_FCLT} to get a CLT for $\Psi^{(\FDP)}(\hat{F}_{m,0},\hat{F}_{m,1})$ and $\Psi^{(\FPR)}(\hat{F}_{m,0},\hat{F}_{m,1})$. However, we are primarily interested in CLTs for $\FDP_m$ and $V_m/m$, so we prove the following lemma to help us derive CLTs for $\FDP_m$ and $V_m/m$ using CLTs for $\Psi^{(\FDP)}(\hat{F}_{m,0},\hat{F}_{m,1})$ and $\Psi^{(\FPR)}(\hat{F}_{m,0},\hat{F}_{m,1})$.

\begin{lemma}\label{lemma:PsiConverenceInProb}
    Under the conditions of Theorem \ref{theorem:conditionalFDPCLT}, 
    \begin{align*}
    \sqrt{m}\big( \emph{\FDP}_m -\Psi^{(\emph{\FDP})}(\hat{F}_{m,0},\hat{F}_{m,1}) \big) &\xrightarrow{\mathcal{P}} 0\quad\text{and}\\
    \sqrt{m}\big( V_m/m -\Psi^{(\emph{\FPR})}(\hat{F}_{m,0},\hat{F}_{m,1}) \big) &\xrightarrow{\mathcal{P}} 0
    \end{align*}
    %and $\sqrt{m}\big( V_m/m -\Psi^{(\emph{\FPR})}(\hat{F}_{m,0},\hat{F}_{m,1}) \big) \xrightarrow{\mathcal{P}} 0$ 
    as $m \rightarrow \infty$, where  $\xrightarrow{\mathcal{P}}$ denotes convergence in probability in the sense of Definition \ref{def:VDV_convergence_in_prob}.
\end{lemma}

\begin{proof}
To prove this, define $\eta \equiv (G(a)-a/q )/2>0$. 
We will show that whenever $\vert \hat{G}_m(a) -G(a) \vert < \eta$, both $\FDP_m=\Psi^{(\FDP)}(\hat{F}_{m,0},\hat{F}_{m,1})$ and $V_m/m =\Psi^{(\FPR)}(\hat{F}_{m,0},\hat{F}_{m,1})$. 
To do this, note that $\vert \hat{G}_{m}(a) - G(a) \vert < \eta \Rightarrow \hat{G}_m(a)> a/q$ because if $\vert \hat{G}_{m}(a) - G(a) \vert < \eta$, then $$\hat{G}_{m}(a) - \frac{a}{q}=\hat{G}_{m}(a)-G(a)+G(a)-\frac{a}{q} =\hat{G}_{m}(a)-G(a)+2 \eta > \eta >0.$$ 
Thus, letting $\tauBHm$ be the Benjamini-Hochberg threshold when applying the procedure at level $q$ to the first $m$ $p$-values (with ECDF $\hat{G}_m$), 
$$\vert \hat{G}_{m}(a) - G(a) \vert < \eta \Rightarrow \tauBHm = \sup \Bigr\{ t \in [0,1] \such  \hat{G}_m(t)  \geq \frac{t}{q} \Bigl\} \geq a.$$ 

Noting that $\tauBHm \leq b$ always (because $b>q$ and $\hat{G}_m \leq 1$), it follows that $\tauBHm=\sup \bigr\{ t \in [a,b] \such  \hat{G}_m(t)  \geq t/q \bigl\}=\mathcal{T}(\hat{G}_m)$ whenever $\vert \hat{G}_{m}(a) - G(a) \vert < \eta$. Now, since $$\FDP_m=\frac{\hat{F}_{m,0}(\tauBHm)}{\hat{G}_{m}(\tauBHm)} \quad \text{ and } \quad \ \Psi^{(\FDP)}(\hat{F}_{m,0},\hat{F}_{m,1})= \frac{\hat{F}_{m,0}(\mathcal{T}(\hat{G}_m))}{\hat{G}_{m}(\mathcal{T}(\hat{G}_m))},$$ while $$\frac{V_m}{m}=\hat{F}_{m,0}(\tauBHm) \quad \text{ and } \quad \ \Psi^{(\FPR)}(\hat{F}_{m,0},\hat{F}_{m,1})= \hat{F}_{m,0}(\mathcal{T}(\hat{G}_m)),$$ it is easy to see that whenever $\vert \hat{G}_{m}(a) - G(a) \vert < \eta$, both $\FDP_m=\Psi^{(\FDP)}(\hat{F}_{m,0},\hat{F}_{m,1})$ and $V_m/m =\Psi^{(\FPR)}(\hat{F}_{m,0},\hat{F}_{m,1})$.

Now note that by Proposition \ref{prop:fdd_convergence}, $\big(W_{m,0}(a),W_{m,1}(a) \big) \xrightarrow{d} \big(\tilde{W}_0(a),\tilde{W}_1(a) \big)$. Since Condition~\ref{cond:F0F1_fast_unif_conv} implies $\sqrt{m} \big( F_{m,0}(a)-F_0(a),F_{m,1}(a)-F_1(a) \big) \xrightarrow{p} 0$, by Slutsky's lemma we get  
$$ \begin{bmatrix} \hat{W}_{m,0}(a) \\ \hat{W}_{m,1}(a) \end{bmatrix}=  \begin{bmatrix} W_{m,0}(a) \\ W_{m,1}(a) \end{bmatrix} +\sqrt{m} \begin{bmatrix} F_{m,0}(a)-F_0(a) \\ F_{m,1}(a)-F_1(a) \end{bmatrix}  \xrightarrow{d}  \begin{bmatrix} \tilde{W}_0(a) \\ \tilde{W}_1(a)  \end{bmatrix}.$$ 
Then the continuous mapping theorem implies that $\hat{W}_{m,0}(a)+\hat{W}_{m,1}(a) \xrightarrow{d} \tilde{W}_0(a)+\tilde{W}_1(a)$ or, equivalently, that $\sqrt{m} \big(\hat{G}_m(a)-G(a)\big) \xrightarrow{d} \tilde{W}_0(a)+\tilde{W}_1(a)$ where $\tilde{W}_0(a)+\tilde{W}_1(a)$ is a Guassian random variable with finite variance. It follows that $\vert \hat{G}_m(a)-G(a) \vert \xrightarrow{p} 0$ which, by the definition of convergence in probability, implies that $\lim_{m \rightarrow \infty} \mathbb{P}(\vert \hat{G}_m(a)-G(a) \vert \geq \eta)=0$. Now fixing $\epsilon>0$ and letting $\mathbb{P}^*$ denote outer probability measure, it follows that 
$$\begin{aligned} 0= & \limsup_{m \rightarrow \infty} \mathbb{P}(\vert \hat{G}_m(a)-G(a) \vert \geq \eta) \\ \geq & \limsup_{m \rightarrow \infty} \mathbb{P}^* \Big( \FDP_m \neq \Psi^{(\FDP)}(\hat{F}_{m,0},\hat{F}_{m,1}) \Big) \\ \geq & \limsup_{m \rightarrow \infty} \mathbb{P}^* \Big( \big| \sqrt{m} \big(\FDP_m -\Psi^{(\FDP)}(\hat{F}_{m,0},\hat{F}_{m,1}) \big) \big| > \epsilon  \Big) \end{aligned}$$
%$$\begin{aligned} \lim_{m \rightarrow \infty} P^* \Big( \sqrt{m} \big(\FDP_m -\Psi^{(\FDP)}(\hat{F}_{m,0},\hat{F}_{m,1}) \big) > \epsilon  \Big) \leq \lim_{m \rightarrow \infty} P(\vert \hat{G}_m(a)-G(a) \vert \geq \eta) =0 \end{aligned}$$
by the fact that $\FDP_m \neq \Psi^{(\FDP)}(\hat{F}_{m,0},\hat{F}_{m,1}) \Rightarrow \vert \hat{G}_m(a)-G(a) \vert \geq \eta$ and by monotonicity of outer probability measure. A similar argument shows that, $$\begin{aligned} 0 \geq \limsup_{m \rightarrow \infty} \mathbb{P}^*\big( \big| \sqrt{m} \big({V_m}/{m} -\Psi^{(\FPR)}(\hat{F}_{m,0},\hat{F}_{m,1}) \big) \big| > \epsilon  \big). \end{aligned}$$ 
Since this holds for any $\epsilon>0$ and since outer measure is nonnegative, by the definition of convergence in probability in the sense of Definition \ref{def:VDV_convergence_in_prob} we get the desired result. 
\end{proof}

\subsection{Proofs of Theorems \ref{theorem:conditionalFDPCLT} and \ref{theorem:conditionalFPR_CLT}}

Assume that all of the conditions of Theorem \ref{theorem:conditionalFDPCLT} hold. 
By our definition for $\hat{W}_{m,0}$ and $\hat{W}_{m,1}$ in \eqref{eq:defws} and by Corollary \ref{cor:VDV_FCLT},
\begin{equation}\label{eq:applyFunctionDmethodBefore}
    \sqrt{m} \big( (\hat{F}_{m,0}  , \hat{F}_{m,1})  -   (F_0, F_1)   \big) \xrightarrow{\mathcal{D}}  (W_0,W_1) 
\end{equation} where $(W_0,W_1) \in C[a,b]_2$ and is a joint Gaussian process with mean zero and joint covariance kernel $c$.

Since according to Propositions \ref{prop:FDPHadDeriv} and \ref{prop:FPRHadDeriv}, both $\Psi^{(\FDP)}$ and $\Psi^{(\FPR)}$ are Hadamard differentiable at $(F_0,F_1)$ tangentially to $C[a,b]_2$, by applying the functional delta method (Theorem \ref{theorem:functional_delta_method}) to \eqref{eq:applyFunctionDmethodBefore} we get \begin{equation}\label{eq:FDPapplyFDelta}
    \sqrt{m} \Big( \Psi^{(\FDP)}(\hat{F}_{m,0}  , \hat{F}_{m,1})-\Psi^{(\FDP)}(F_0,F_1) \Big) \xrightarrow{\mathcal{D}}  \dot{\Psi}_{(F_0,F_1)}^{(\FDP)}(W_0,W_1)
\end{equation} 
and 
\begin{equation}\label{eq:FPRapplyFDelta}
    \sqrt{m} \Big( \Psi^{(\FPR)}(\hat{F}_{m,0}  , \hat{F}_{m,1})-\Psi^{(\FPR)}(F_0,F_1) \Big) \xrightarrow{\mathcal{D}}  \dot{\Psi}_{(F_0,F_1)}^{(\FPR)}(W_0,W_1).
\end{equation}

Now $\sqrt{m}\big( \FDP_m -\Psi^{(\FDP)}(\hat{F}_{m,0},\hat{F}_{m,1}) \big) \xrightarrow{\mathcal{P}} 0$ and $\sqrt{m}\big( V_m/m -\Psi^{(\FPR)}(\hat{F}_{m,0},\hat{F}_{m,1}) \big) \xrightarrow{\mathcal{P}} 0$, by Lemma \ref{lemma:PsiConverenceInProb}, so adding these terms that converge to $0$ in probability to the left hand sides of \eqref{eq:FDPapplyFDelta} and \eqref{eq:FPRapplyFDelta}, respectively, and applying Slutsky's lemma (Lemma \ref{lemma:Slutsky_new}), we find that  \begin{equation}\label{eq:FDPCLTunsimplified}
    \sqrt{m} \Big( \FDP_m-\Psi^{(\FDP)}(F_0,F_1) \Big) \xrightarrow{\mathcal{D}}  \dot{\Psi}_{(F_0,F_1)}^{(\FDP)}(W_0,W_1)
\end{equation} and \begin{equation}\label{eq:FPRCLTunsimplified}
    \sqrt{m} \Big( \frac{V_m}{m}-\Psi^{(\FPR)}(F_0,F_1) \Big) \xrightarrow{\mathcal{D}}  \dot{\Psi}_{(F_0,F_1)}^{(\FPR)}(W_0,W_1).
\end{equation}
Now noting that $\Psi^{(\FDP)}(F_0,F_1)=F_0(\tau_*)/G(\tau_*)=q F_0(\tau_*)/\tau_*$ and that $\Psi^{(\FPR)}(F_0,F_1)=F_0(\tau_*)$ and recalling the formulas for $\dot{\Psi}_{(F_0,F_1)}^{(\FDP)}$ and $\dot{\Psi}_{(F_0,F_1)}^{(\FPR)}$ from Propositions \ref{prop:FPRHadDeriv} and \ref{prop:FDPHadDeriv}, \eqref{eq:FDPCLTunsimplified} and \eqref{eq:FPRCLTunsimplified} simplify to  \begin{equation}\label{eq:FDPCLTsimplified}
    \sqrt{m} \Big( \FDP_m-\frac{qF_0(\tau_*) }{\tau_*} \Big) \xrightarrow{\mathcal{D}}  \frac{q}{\tau_*} \Big( (1+\alpha) W_0(\tau_*) + \alpha W_1(\tau_*) \Big)
\end{equation} and \begin{equation}\label{eq:FPRCLTsimplified}
    \sqrt{m} \Big( \frac{V_m}{m}-F_0(\tau_*) \Big) \xrightarrow{\mathcal{D}}  W_0(\tau_*)+\frac{F_0'(\tau_*)}{1/q-G'(\tau_*)} \big(W_0(\tau_*)+W_1(\tau_*) \big).
\end{equation}

We now remark that for a finite $m$, $\FDP_m$ and $V_m$ are Borel measurable random variables, because they take on finitely many possible values and each value they take on can be written as finite unions and intersections of measurable sets. Because $\FDP_m$ and $V_m$ are Borel measurable random variables, the notion of convergence in distribution in Definition \ref{def:VDV_weak_convergence} that is referenced in \eqref{eq:FDPCLTsimplified} and \eqref{eq:FPRCLTsimplified} is interchangeable with the standard notion of convergence in distribution, which we denote with the symbol $\xrightarrow{d}$. Combining this with the fact that $(W_0,W_1)$ is a joint Gaussian process with mean zero and joint covariance kernel $c$, we get that \eqref{eq:FDPCLTsimplified} and \eqref{eq:FPRCLTsimplified} 
simplify to  \begin{equation}\label{eq:FDPCLTsimplifiedMore}
    \sqrt{m} \Big( \FDP_m-\frac{qF_0(\tau_*) }{\tau_*} \Big) \xrightarrow{d} \dnorm(0,\sigma_L^2) 
\end{equation} and \begin{equation}\label{eq:FPRCLTsimplifiedMore}
    \sqrt{m} \Big( \frac{V_m}{m}-F_0(\tau_*) \Big) \xrightarrow{d}  \dnorm(0,\sigma_R^2)
\end{equation}
where 
$$\sigma_L^2 \equiv \frac{q^2}{\tau_*^2} \Big( (1+\alpha)^2 c^{(0,0)}(\tau_*,\tau_*)+2 \alpha (1+\alpha)c^{(1,0)}(\tau_*,\tau_*) +\alpha^2c^{(1,1)}(\tau_*,\tau_*)  \Big) $$
and where, defining $\beta \equiv F_0'(\tau_*)/(1/q-G'(\tau_*))$, 
$$\sigma_R^2 \equiv  (1+\beta)^2 c^{(0,0)}(\tau_*,\tau_*)+2 \beta (1+\beta)c^{(1,0)}(\tau_*,\tau_*) +\beta^2c^{(1,1)}(\tau_*,\tau_*).$$ 
Since equations \eqref{eq:FDPCLTsimplifiedMore} and \eqref{eq:FPRCLTsimplifiedMore} hold under the conditions of Theorem~\ref{theorem:conditionalFDPCLT}, we have thus proved Theorems \ref{theorem:conditionalFDPCLT} and \ref{theorem:conditionalFPR_CLT}. 
\qed

\subsection{Proof of Theorem \ref{theorem:tau_star_0_res}}

Set $[a,b]=[0,1]$ and fix the $\bswu=\bswl \in \real^k$ on which we condition and suppose that the conditions of Theorem \ref{theorem:tau_star_0_res} hold.  Fix $\epsilon \in (0,q)$, and we will show that $\lim_{m \rightarrow \infty} \Pr( \tauBHm \leq \epsilon  )=1$. To do this, define $\psi_G : [0,1] \rightarrow \real$, to be the function given by $\psi_G(t)=G(t)-t/q$. Since $\tau_*=0$, $\psi_G(t) < 0$ for all $t \in (0,1)$. By Proposition \ref{prop:G_is_continuous}, $\psi_G$ is continuous on $[\epsilon,q]$, so by the extreme value theorem, $\psi_G$ attains a maximum at some $t_{\text{max}} \in [\epsilon,q]$. Because $\psi_G(t_{\text{max}}) < 0$, it follows that letting $\delta = -\psi_G(t_{\text{max}}) > 0$, $\psi_G(t) \leq -\delta$ for all $t \in [\epsilon,q]$.
    
Now for $l \in \mathbb{N}$, define $t_l = lq\delta/2$, and let $L_{\delta} = \{ l \such t_l \in [\epsilon,q] \}$. Let $T_{\delta}=\{q \} \cup \{t_l \such l \in L_{\delta} \}$. $T_{\delta}$ is a finite set. Let $$A_m =  \bigcap\limits_{t \in T_{\delta}} \{ \hat{G}_m(t) -t/q < -\delta/2 \}.$$ We will first show that $A_m \subseteq \{ \tauBHm \leq \epsilon \}$, and then we will show that $\lim_{m \rightarrow \infty} \Pr( \tauBHm \leq \epsilon  )=1$. To prove the former, suppose the event $A_m$ occurs and take any $t \in [\epsilon,1]$. If $t > q$, then $\hat{G}_m(t)-t/q <0$, since $\hat{G}_m(t) \leq 1$. If instead $t \in [\epsilon, q]$, there exists a $t_{\uparrow} \in T_{\delta}$ such that $t_{\uparrow}-t \in [0,q \delta/2)$. Therefore, since $\hat{G}_m$ is non-decreasing, 
$$\begin{aligned} \hat{G}_m(t) - \frac{t}{q} \leq \hat{G}_m(t_{\uparrow})-\frac{t_{\uparrow}}{q}+\frac{t_{\uparrow}-t}{q} < -\frac{\delta}{2}+\frac{\delta}{2} <0 \end{aligned},$$ where the second to last inequality holds under the event $A_m$. 
Thus under the event $A_m$, $\hat{G}_m(t)-t/q <0$ for all $t \in [\epsilon,1]$, implying that 
%under the event 
$A_m \subseteq \{ \tauBHm \leq \epsilon \}$.
    
To show that $\lim_{m \rightarrow \infty} \Pr( \tauBHm \leq \epsilon  )=1$, note that the proof of Proposition \ref{prop:fdd_convergence} holds when Conditions \ref{cond:boundedL}, \ref{cond:niceQandgamma}, \ref{cond:F0F1_defined}, and \ref{cond:cov_kernel_converges} hold, implying that $(W_{m,0},W_{m,1}) \xrightarrow{\fdd} (\tilde{W}_0,\tilde{W}_1)$ where $(\tilde{W}_0,\tilde{W}_1)$ is a joint Guassian process on $[0,1]$ with mean zero and joint covariance kernel $c$. Since Condition \ref{cond:F0F1_fast_unif_conv} holds for $[a,b]=[0,1]$, by Slutsky's lemma, $(\hat{W}_{m,0},\hat{W}_{m,1}) \xrightarrow{\fdd} (\tilde{W}_0,\tilde{W}_1)$. By the continuous mapping theorem, $\hat{W}_{m,0}+\hat{W}_{m,1} \xrightarrow{\fdd} \tilde{W}_0+\tilde{W}_1$ which implies that $\big( \hat{W}_{m,0}(t) +\hat{W}_{m,1}(t) \big)_{t \in T_{\delta}}$ converges in distribution to a multivariate Gaussian. Since $\hat{W}_{m,0}(t) +\hat{W}_{m,1}(t)= \sqrt{m} \big( \hat{G}_m(t)-G(t) \big)$ and since $T_{\delta}$ is a finite set, if we let $B_m$ be the event that $\vert \hat{G}_m(t)-G(t) \vert < \delta/2$ for all $t \in T_{\delta}$, it is clear that $\lim_{m \rightarrow \infty} \Pr(B_m)=1$. Now note that if $B_m$ occurs $A_m$ occurs, since $G(t)-t/q \leq -\delta$ for all $t \in [\epsilon,q]$. Because $B_m \subseteq A_m \subseteq \{ \tauBHm \leq \epsilon \}$ and $\lim_{m \rightarrow \infty} \Pr(B_m)=1$, $\lim_{m \rightarrow \infty} \Pr( \tauBHm \leq \epsilon  )=1$. The above argument holds for any $\epsilon >0$, so by nonnegativity of $\tauBHm$, $\tauBHm \xrightarrow{p} 0$.
    
To show that $V_m/m \xrightarrow{p} 0$, fix $\epsilon>0$. Note that if $\tauBHm \leq q \epsilon/2$, then $$\frac{V_m}{m} = \hat{F}_{m,0}(\tauBHm) \leq \hat{G}_m(\tauBHm) \leq \hat{G}_m(q \epsilon) < \frac{q\epsilon}{q}=\epsilon$$ 
where the last inequality holds because, if it did not hold, there would be a contradiction to $\tauBHm \leq q \epsilon/2$. 
Therefore, $\tauBHm \leq q \epsilon/2 \Rightarrow V_m/m \leq \epsilon$, and so by the result of the previous paragraph, $\lim_{m \rightarrow \infty} \Pr( V_m/m \leq \epsilon  )=1$. 
By nonnegativity of $V_m/m$ it follows that $V_m/m \xrightarrow{p} 0$. 
\qed

\section{Hadamard derivative calculations for proofs of Theorems \ref{theorem:conditionalFDPCLT} and \ref{theorem:conditionalFPR_CLT}}\label{sec:HadamardDerivComp}

Recall that our goal is to compute the Hadamard derivative of both $\Psi^{(\FPR)}$ and $\Psi^{(\FDP)}$ at the point $(F_0,F_1) \in C[a,b]_2$ tangentially to $C[a,b]_2$. We start by computing the Hadamard derivative of $\mathcal{T}$ (defined in Equation \eqref{eq:mathcalTDef}) at $G$ tangentially to $C[a,b]$, but to do this, it first helps to prove the following modified version of Lemma 7.7 and Proposition 7.8 in \cite{Neuvial2008}.

\begin{lemma}\label{lemma:Neuivial_Helper}
    For any $F \in D[a,b]$ such that $\mathcal{T}(F) \in (a,b)$, one of the following must hold:
    
    \begin{enumerate}
        \item[(i)]  $F( \mathcal{T}(F))= \mathcal{T}(F)/q$, or
        \item[(ii)] $F( \mathcal{T}(F)) \leq  \mathcal{T}(F)/q \leq \lim_{t \uparrow \mathcal{T}(F) } F(t)$, 
\end{enumerate}  where case (i) will hold whenever $F$ is monotone nondecreasing.
\end{lemma}

\begin{proof}
Fix $F \in D[a,b]$ such that $\mathcal{T}(F) \in (a,b)$. 
Assume by way of contradiction that $F (  \mathcal{T}(F) ) > \mathcal{T}(F)/q$. By right continuity of $t \mapsto F(t)-t/q$, there exists a sequence of $t_n \downarrow \mathcal{T}(F)$ (where $t_n \in (a,b)$ for all $n$) such that $\lim_{n \rightarrow \infty} F(t_n)-t_n/q = F (  \mathcal{T}(F) )-\mathcal{T}(F)/q> 0$, implying that for some $n \in \mathbb{N}$, $F(t_n)-t_n/q>0$. 
But that implies that for some $t_n \in (a,b)$ and $t_n >  \mathcal{T}(F)$, $F(t_n) \geq t_n/q$, contradicting the definition of $\mathcal{T}(F)$. 
Hence it follows that $F (  \mathcal{T}(F) ) \leq \mathcal{T}(F)/q$ always. 
If  $F (  \mathcal{T}(F) ) = \mathcal{T}(F)/q$, then we are in case (i), and otherwise, $F (  \mathcal{T}(F) ) < \mathcal{T}(F)/q$. 
Now if $F (  \mathcal{T}(F) ) < \mathcal{T}(F)/q$, by definition of $\mathcal{T}(F)$, there exists a sequence of $t_n \uparrow \mathcal{T}(F)$ such that $F(t_n)- t_n/q \geq 0$. 
So taking the limit as $n \rightarrow \infty$ of each side, since $F$ always has existing left limits, $\lim_{t \uparrow \mathcal{T}(F)} F(t) - \mathcal{T}(F)/q \geq 0$. Hence whenever $F (  \mathcal{T}(F) ) < \mathcal{T}(F)/q$, $ \mathcal{T}(F)/q \leq \lim_{t \uparrow \mathcal{T}(F) } F(t)$, so case (ii) will hold. When $F$ is monotone nondecreasing case (i) holds because the RHS in case (ii) is bounded above by $F( \mathcal{T}(F))$ for monotone nondecreasing $F$.
\end{proof}

Using this Lemma and an approach similar to \cite{Neuvial2008}, we can calculate the Hadamard derivative of $\mathcal{T}$ at $G$ tangentially to $C[a,b]$. The proposition below involves the following constant which will be convenient to define for the remainder of the supplement:

\begin{equation}\label{eq:cG_def}
    c_G \equiv 1/q-G'(\tau_*)>0
\end{equation}

By Condition \ref{cond:F0F1_differfentiable}
$c_G$ is well defined. To see why this constant is positive, note that by Condition \ref{cond:noaccum}, Proposition \ref{prop:existance_ab} and our definitions of $\tau_*$ and the interval $[a,b]$, $\tau_*$ is the unique point in $[a,b]$ where the function $t \mapsto G(t) -t/q$ crosses from a positive to negative value (and does not merely touch zero).

\begin{proposition}\label{prop:HadDeriv_T}
   $\mathcal{T}$ is Hadamard differentiable at $G$, tangentially to $C[a,b]$ with Hadamard derivative at G given by $\dot{\mathcal{T}}_G(H)= H(\tau_*)/c_G$ for $H \in C[a,b]$.
\end{proposition}

\begin{proof}
    
    To prove that the Hadamard derivative of $\mathcal{T}$ at $G$ tangentially to $C[a,b]$ is given by $\dot{\mathcal{T}}_G(H)= H(\tau_*)/c_G$, fix any $H \in C[a,b]$ and fix any collection $(H_t)_{t>0}$ where $H_t \in D[a,b]$ for all $t$ and $ \lim_{t \downarrow 0} \Vert  H_t -H \Vert_{\infty} = 0$. We must show that \begin{equation}\label{eq:T_is_HadDiff}
        \lim_{t \downarrow 0} \Big| \frac{\mathcal{T}(G+tH_t)-\mathcal{T}(G)}{t} -  \frac{H(\tau_*)}{c_G} \Big|=0.
    \end{equation} 
    To do this, define $G_t \equiv G+tH_t$ and $\tau_t \equiv \mathcal{T}(G_t)$ for $t>0$. 
    Also define for any $F \in D[a,b]$,  $\psi_F : [a,b] \rightarrow \real$ to be given by $\psi_F(u)={u}/{q} -F(u)$. 
    Note $\mathcal{T}(F)=\sup \{ u \in [a,b], \psi_F(u) \leq 0 \}$. 
    We must first show that there exists a $\delta>0$ such that for all  $t \in (0, \delta)$, $\tau_t \in (a,b)$. 
    To see this, note that by our choice of interval $(a,b)$ (see Section \ref{sec:Define_ab_subsection}, and more specifically, see Proposition \ref{prop:existance_ab}), $G(a) > {a}/{q} \Rightarrow \psi_G(a) < 0$, $G(b) < {b}/{q} \Rightarrow \psi_G(b) >0 $, and $\psi_G(\cdot)$ crosses $0$ once on the interval $[a,b]$. 
    $\psi_G$ is continuous by continuity of $G$ (see Proposition \ref{prop:G_is_continuous}), so there exists a $b_- < b$ such that $\psi_G(u)> \frac{1}{2} \psi_G(b)$ for all $u \in [b_-,b]$. Also note that 
    $$\Vert G- G_t \Vert_{\infty}=\Vert t H_t \Vert_{\infty} \leq t \Vert  H \Vert_{\infty} + t \Vert H-H_t \Vert_{\infty}.$$ 
    Since $H \in C[a,b]$, $\Vert H \Vert_{\infty} < \infty$, so $\lim_{t \downarrow 0} t \Vert  H \Vert_{\infty}=0$. Hence, taking the limit as $t \downarrow 0$ of the above result we get that $\lim_{t \downarrow 0} \Vert G- G_t \Vert_{\infty} =0$. 
    Thus, we can pick $\delta>0$ such that for all $t \in (0, \delta)$, $ \Vert G- G_t \Vert_{\infty} < \min \{ \vert \psi_G(a) \vert ,  \vert \psi_G(b) \vert/2 \}$. 
    Since for all $u \in [a,b]$, $\vert \psi_{G_t}(u) - \psi_G(u) \vert \leq  \Vert G- G_t \Vert_{\infty}$, we get that for all $t \in (0,\delta)$ and $u \in [b_-,b]$, $\psi_{G_t}(a) < \psi_{G}(a)+ \vert \psi_G(a) \vert =0$ and $\psi_{G_t}(u) >  \psi_{G}(u)- \psi_G(b)/2 >0$. 
    Thus, for all $t \in (0,\delta)$ and $u \in [b_-,b]$, both $\psi_{G_t}(a) < 0$ and $\psi_{G_t}(u)>0$, where $\psi_{G_t}$ is right continuous with left limits. 
    This implies that $\tau_t =\mathcal{T}(G_t) \in (a,b)$ for all $t \in (0,\delta)$.
    
    For any $t \in (0,\delta)$, we will next apply Lemma \ref{lemma:Neuivial_Helper} to $G_t$ to show that $ \vert \psi_G(\tau_t) \vert \leq \Vert G- G_t \Vert_{\infty}$. 
    If case (ii) of Lemma \ref{lemma:Neuivial_Helper} holds for $G_t$, then $$-\lim_{u \uparrow \mathcal{T}(G_t) } G_t(u) \leq -  \frac{\mathcal{T}(G_t)}{q} \leq  -G_t( \mathcal{T}(G_t)),$$ so adding $G( \tau_t)=G(\mathcal{T}(G_t))$ to each side of the equality and by continuity of $G$, we get $$\lim_{u \uparrow \mathcal{T}(G_t) } \big(  (G-G_t)(u) \big) \leq G(\tau_t)-  \frac{\tau_t}{q} \leq (G-G_t) (\tau_t),$$
    which implies that
    $$\vert \psi_G(\tau_t) \vert= \Big\vert G(\tau_t)-  \frac{\tau_t}{q} \Big\vert \leq \Vert G- G_t \Vert_{\infty}.$$ 
    Alternatively, if case (i) from Lemma \ref{lemma:Neuivial_Helper} holds for $G_t$, then $-{\tau_t}/{q}=-G_t( \tau_t)$ implying $G(\tau_t)-{\tau_t}/{q}= (G-G_t)(\tau_t)$, further implying $\vert \psi_G(\tau_t) \vert \leq  \Vert G- G_t \Vert_{\infty}$. Thus, in either case, $ \vert \psi_G(\tau_t) \vert \leq \Vert G- G_t \Vert_{\infty}$ for all $t \in (0,\delta)$. Because $\lim_{t \downarrow 0} \Vert G- G_t \Vert_{\infty} =0$, $\lim_{t \downarrow 0} \psi_G(\tau_t) =0$. By considering existing subsequences of $(\tau_t)_{t>0}$ converging to $\limsup_{t \downarrow 0} \tau_t$ and $\liminf_{t \downarrow 0} \tau_t$ and noting that $\psi_G$ is continuous (because $G$ is continuous), it follows that $ \psi_G(\limsup_{t \downarrow 0} \tau_t)=0$ and $ \psi_G(\liminf_{t \downarrow 0} \tau_t)=0$. Because  $\tau_*$ is the unique and existing $u \in [a,b]$ such that $\psi_G(u)=0$ (see the definition of $[a,b]$ and Proposition \ref{prop:existance_ab}), it follows that $\liminf_{t \downarrow 0} \tau_t=\tau_*=\limsup_{t \downarrow 0} \tau_t $. Hence, $\lim_{t \downarrow 0} \tau_t = \tau_*$.
    
    We would also like to prove that $\lim_{t \downarrow 0} \psi_G(\tau_t)/t=H(\tau_*)$. 
    To show this, we start by showing that for $t \in (0,\delta)$,  \begin{equation}\label{eq:upper_bound_in_HD_proof} \Big| \frac{\psi_G( \tau_t )}{t} - H(\tau_t) \Big| \leq \max \Bigl\{ \Big| H_t(\tau_t) - H(\tau_t) \Big| , \Big| \lim_{u \uparrow \tau_t} ( H_t( u) -H ( u ) ) \Big| \Bigr\}. \end{equation} 
    Let $T_{(i)}$ be the set of $t \in (0,\delta)$ such that $G_t$ is in case (i) of Lemma \ref{lemma:Neuivial_Helper} and let $T_{(ii)}$ be the set of $t \in (0,\delta)$ such that $G_t$ is in case (ii) of Lemma \ref{lemma:Neuivial_Helper}. 
    Note that if $t \in T_{(i)}$, then $\psi_G(\tau_{t})={\tau_{t}}/{q}-G(\tau_{t})=(G_{t}-G)(\tau_{t})=t H_{t}(\tau_{t})$ which implies that ${\psi_G(\tau_{t})}/{t}-H(\tau_t)= H_{t}(\tau_{t})-H(\tau_t)$ and hence \eqref{eq:upper_bound_in_HD_proof} holds for $t \in T_{(i)}$. 
    If instead $t \in T_{(ii)}$, then
    $$G_{t}( \tau_{t}) \leq \frac{\tau_{t}}{q} \leq \lim_{u \uparrow \tau_{t}} G_{t}(u),$$ 
    so subtracting $G(\tau_t)$ from each side and using the fact that $G$ is continuous, we get that 
    $$(G_t-G)(\tau_t) \leq \psi_G(\tau_t) \leq  \lim_{u \uparrow \tau_{t}}  (G_t-G)(u).$$ Recalling that $G_t-G=tH_t$, we can divide each side of the above inequality by $t$ and subtract $H(\tau_t)$ to get $$H_t (\tau_t) - H(\tau_t) \leq \frac{ \psi_G(\tau_t)}{t} -H(\tau_t) \leq  \lim_{u \uparrow \tau_t} H_t( u) -H ( \tau_t ),$$ where the RHS can be rewritten as $ \lim_{u \uparrow \tau_t} ( H_t( u) -H ( u ) ) $ by continuity of $H$. 
    Thus, \eqref{eq:upper_bound_in_HD_proof} holds for any $t \in T_{(ii)}$. 
    Since $T_{(i)} \cup T_{(ii)}= (0,\delta)$ it follows that \eqref{eq:upper_bound_in_HD_proof} holds for all $t \in (0,\delta)$.
    
    Now observe that both $\vert H_t(\tau_t) - H(\tau_t) \vert \leq \Vert H_t - H \Vert_{\infty}$ and $\vert \lim_{u \uparrow \tau_t} ( H_t( u) -H ( u ) ) \vert \leq \Vert H_t - H \Vert_{\infty}$ hold. Therefore, \eqref{eq:upper_bound_in_HD_proof} implies that for any $t \in (0,\delta)$, $\vert \psi_G( \tau_t )/t - H(\tau_t) \vert \leq \Vert H_t - H \Vert_{\infty}$, and so
%    Thus by the triangle inequality,
$$\begin{aligned} \Big\vert \frac{\psi_G( \tau_t )}{t} - H(\tau_*) \Big\vert \leq &   \Big\vert \frac{\psi_G( \tau_t )}{t} - H(\tau_t) \Big\vert+ \vert H(\tau_t) - H(\tau_*) \vert \\ \leq & \Vert H_t - H \Vert_{\infty}+ \vert H(\tau_t) - H(\tau_*) \vert. \end{aligned}$$ 
    Taking the limit as $t \downarrow 0$ of each side above completes the proof that $\lim_{t \downarrow 0} \psi_G(\tau_t)/t=H(\tau_*)$,  where the fact that $\lim_{t \downarrow 0}  \vert H(\tau_t) - H(\tau_*) \vert=0$ holds by the composite limit theorem and our earlier result that $\lim_{t \downarrow 0} \tau_t = \tau_*$.
    
    To complete the proof of \eqref{eq:T_is_HadDiff}, note that $c_G= 1/q-G'(\tau_*)= \psi_G'(\tau_*)$, and hence by the alternative definition of a derivative $$c_G=  \psi_G'(\tau_*) = \lim_{u \rightarrow \tau_*} \frac{\psi_G(u)-\psi_G(\tau_*)}{u-\tau_*}=\lim_{u \rightarrow \tau_*} \frac{\psi_G(u)}{u-\tau_*}=\lim_{t \downarrow 0} \frac{\psi_G(\tau_t)}{\tau_t-\tau_*},$$ where the last equality holds because $\lim_{t \downarrow 0} \tau_t = \tau_*$ and because the limit on the LHS exists by Condition \ref{cond:F0F1_differfentiable}. Therefore, for $t >0$, $$\frac{\mathcal{T}(G+tH_t)-\mathcal{T}(G)}{t} = \frac{\tau_t-\tau_*}{t} = \frac{\psi_G(\tau_t)}{t} \Big( \frac{\psi_G(\tau_t)}{\tau_t-\tau_*}  \Big)^{-1} $$ and hence, $$\lim_{t \downarrow 0} \frac{\mathcal{T}(G+tH_t)-\mathcal{T}(G)}{t}=  \Big( \lim_{t \downarrow 0} \frac{\psi_G(\tau_t)}{t} \Big) \Big( \lim_{t \downarrow 0}\frac{\psi_G(\tau_t)}{\tau_t-\tau_*}  \Big)^{-1} = \frac{H(\tau_*)}{c_G}$$ proving \eqref{eq:T_is_HadDiff}. The above proof of \eqref{eq:T_is_HadDiff} holds for any fixed $H \in C[a,b]$ and fixed collection $(H_t)_{t>0}$ of $H_t \in D[a,b]$ such that $\lim_{t \downarrow 0} \Vert  H_t -H \Vert_{\infty} = 0$. Thus, the Hadamard derivative of $\mathcal{T}$ at $G$ tangentially to $C[a,b]$ is given by $\dot{\mathcal{T}}_G(H)= H(\tau_*)/c_G$.
\end{proof}

Using Proposition \ref{prop:HadDeriv_T} and the chain rule for Hadamard differentiability we can compute the Hadamard derivatives of $\Psi^{(\FPR)}$, $\mathcal{R}$, and $\Psi^{(\FDP)}$ at $(F_0,F_1)$ tangentially to $C[a,b]_2$. We start by computing the Hadamard derivative of $\Psi^{(\FPR)}$.

\begin{proposition}\label{prop:FPRHadDeriv}
    $\Psi^{(\emph{\FPR})}$ (defined in Equation \eqref{eq:PsiFPRDef}) is Hadamard differentiable at $(F_0,F_1)$ tangentially to $C[a,b]_2$ with $$\dot{\Psi}_{(F_0,F_1)}^{(\emph{\FPR})}(H_0,H_1)= H_0(\tau_*)+F_0'(\tau_*) \big(H_0(\tau_*)+H_1(\tau_*) \big)/c_G$$ 
    for $(H_0,H_1) \in C[a,b]_2$ being the Hadamard derivative at $(F_0,F_1)$.
\end{proposition}

\begin{proof}

To compute the Hadamard derivative of $\Psi^{(\text{FPR})}$ at $(F_0,F_1)$ tangentially to $C[a,b]_2$ fix $(H_0,H_1) \in C[a,b]_2$ and a collection $\big( (H_{0t},H_{1t}) \big)_{t>0}$ of elements of $D[a,b]_2$ such that $$\lim_{t \downarrow 0} \Vert (H_{0t},H_{1t}) - (H_0,H_1) \Vert_{\infty}=0.$$ Let $F_{0t} \equiv F_0+tH_{0t}$, $F_{1t} \equiv F_1+tH_{1t}$, and $G_t \equiv F_{0t} +F_{1t}  = (F_0+F_1)+t(H_{0t}+H_{1t})$ where $H_{0t}+H_{1t} \rightarrow H_0+H_1$ as $t \downarrow 0$. Letting $\tau_t \equiv \mathcal{T}(G_t)$ and recalling $\tau_*= \mathcal{T}(G)$, note that \begin{equation}\label{eq:FPRProofEq1} \frac{\Psi^{(\text{FPR})}(F_{0t}, F_{1t}) - \Psi^{(\text{FPR})}(F_0, F_1) }{t}=  H_{0t}(\tau_t)+ \frac{F_0(\tau_t)-F_0(\tau_*) }{t} .\end{equation} 
Since $\mathcal{T}$ is Hadamard differentiable at $G$ tangentially to $C[a,b]$ (see Proposition \ref{prop:HadDeriv_T}) and $F_0$ is differentiable at $\mathcal{T}(G)$, by the chain rule for Hadamard differentiability (see, for example, Theorem 20.9 in \cite{VDV}), it follows that $F_0 \circ \mathcal{T} : D[a,b] \rightarrow [0,1]$ is Hadamard differentiable at $G$ tangentially to $C[a,b]$ with Hadamard derivative given by $F_0'(\mathcal{T}(G)) \dot{\mathcal{T}}_G (H)$ for $H \in C[a,b]$. 
By the definition of Hadamard differentiability and since $G_t=(F_0+F_1)+t(H_{0t}+H_{1t})$ where $H_{0t}+H_{1t} \rightarrow H_0+H_1$ as $t \downarrow 0$, it follows that 
$$\lim_{t \downarrow 0} \frac{F_0(\tau_t)-F_0(\tau_*) }{t}=\lim_{t \downarrow 0} \frac{F_0(\mathcal{T}(G_t))-F_0(\mathcal{T}(G)) }{t} = F_0'(\mathcal{T}(G)) \dot{\mathcal{T}}_G (H_0+H_1).$$ 
By combining the above result with the formula from Proposition \ref{prop:HadDeriv_T},
\begin{equation}\label{eq:FPRLimit2} \lim_{t \downarrow 0} \frac{F_0(\tau_t)-F_0(\tau_*) }{t} =\frac{F_0'(\tau_*) \big(H_0(\tau_*)+H_1(\tau_*) \big) }{c_G}. \end{equation}
%Now note that by the triangle inequality
% (I think the triangle inequality is too immediate to cite)
Also note that
$$\begin{aligned} \vert H_{0t} ( \tau_t ) - H_0( \tau_*) \vert = & \vert  H_{0t} ( \tau_t ) - H_0 ( \tau_t )\vert + \vert H_0 ( \tau_t ) -H_0( \tau_*) \vert \\ \leq & \Vert H_{0t} -H_0 \Vert_{\infty} + \vert H_0 ( \tau_t ) -H_0(\tau_*) \vert. \end{aligned}$$

 We know that $\lim_{t \downarrow 0} \Vert H_{0t} -H_0 \Vert_{\infty}=0$. Further since $H_0$ is continuous and we showed in the proof of Proposition \ref{prop:HadDeriv_T} that $\lim_{t \downarrow 0}  \tau_t =  \tau_*$ (in a setting where $\tau_t=\mathcal{T}(G_t)$ and $G_t=G+tH_t \in D[a,b]$ for some $H_t \rightarrow H \in C[a,b]$ as $t \downarrow 0$) we get by the composite limit theorem that $\lim_{t \downarrow 0} \vert H_0 ( \tau_t ) -H_0( \tau_*) \vert =0$. Hence by an above inequality it follows that $\lim_{t \downarrow 0} H_{0t} ( \tau_t )=H_0(\tau_*)$.
 
 Combining this with \eqref{eq:FPRProofEq1} and \eqref{eq:FPRLimit2}, we get that $$\begin{aligned} \lim_{t \downarrow 0}  \frac{\Psi^{(\text{FPR})}(F_{0t}, F_{1t}) - \Psi^{(\text{FPR})}(F_0, F_1) }{t}= H_0(\tau_*)+\frac{F_0'(\tau_*) \big(H_0(\tau_*)+H_1(\tau_*) \big) }{c_G}. \end{aligned}$$ Since this argument holds for any fixed $(H_0,H_1) \in C[a,b]_2$ and collection $\big( (H_{0t},H_{1t}) \big)_{t>0}$ of elements of $D[a,b]_2$  such that $\lim_{t \downarrow 0} \Vert (H_{0t},H_{1t}) - (H_0,H_1) \Vert_{\infty}=0$ where $F_{0t} \equiv F_0+tH_{0t}$ and $F_{1t} \equiv F_1+tH_{1t}$, it follows that $\Psi^{(\text{FPR})}$ is Hadamard differentiable at $(F_0,F_1)$ tangentially to $C[a,b]_2$ with Hadamard derivative given by \begin{align*} \dot{\Psi}_{(F_0,F_1)}^{(\text{FPR})}(H_0,H_1)= & H_0(\tau_*)+F_0'(\tau_*) \big(H_0(\tau_*)+H_1(\tau_*) \big)/c_G. & \qedhere \end{align*}
\end{proof}

%\qedhere can be re-added here but, I commented it out to remove warnings

\begin{proposition}\label{prop:RHadDeriv}
    $\mathcal{R}$ (defined in Equation \eqref{eq:mathcalRDef}) is Hadamard differentiable at $(F_0,F_1)$ tangentially to $C[a,b]_2$ with $$\dot{\mathcal{R}}_{(F_0,F_1)}(H_0,H_1)= \big( H_0(\tau_*)+H_1(\tau_*) \big)/(qc_G)$$ for $(H_0,H_1) \in C[a,b]_2$ being the Hadamard derivative at $(F_0,F_1)$.
\end{proposition}

\begin{proof}
 
 Define $\Psi^{(\text{A})}: D[a,b]_2 \rightarrow \real$ via 
 $\Psi^{(\text{A})}(H_0,H_1)=H_1 \big( \mathcal{T}(H_0+H_1) \big)$ for any $(H_0,H_1) \in D[a,b]_2$. Swapping the indices of the two $D[a,b]$ functions in the input of $\Psi^{(\text{A})}$ gives exactly the same function as $\Psi^{(\FPR)}$. Therefore, by the argument in  Proposition \ref{prop:FPRHadDeriv}, $\Psi^{(\text{A})}$ is Hadamard differentiable at $(F_0,F_1)$ tangentially to $C[a,b]_2$ with Hadamard derivative given by $\dot{\Psi}_{(F_0,F_1)}^{(\text{A})}(H_0,H_1)= H_1(\tau_*)+F_1'(\tau_*) \big(H_0(\tau_*)+H_1(\tau_*) \big)/c_G$. 
 %Noting that
 Because $\mathcal{R}=\Psi^{(\FPR)}+\Psi^{(\text{A})}$ 
 %and that sum Hadamard differentiability is preserved under finite summation with the Hadamard derivative a linear operator, 
 we find that $\mathcal{R}$ is Hadamard differentiable at $(F_0,F_1)$ tangentially to $C[a,b]_2$ with Hadamard derivative %given by
 $\dot{\mathcal{R}}_{(F_0,F_1)}=\dot{\Psi}_{(F_0,F_1)}^{(\FPR)}+\dot{\Psi}_{(F_0,F_1)}^{(\text{A})}$. Combining our formula for $\dot{\Psi}_{(F_0,F_1)}^{(\text{A})}$ with Proposition \ref{prop:FPRHadDeriv} we get that for any $(H_0,H_1) \in C[a,b]_2$, \begin{align*} \dot{\mathcal{R}}_{(F_0,F_1)}(H_0,H_1) = & \big( H_0(\tau_*)+H_1(\tau_*) \big) \Big(1+\frac{F_0'(\tau_*)+F_1'(\tau_*)}{c_G} \Big) \\ = &  \frac{ H_0(\tau_*)+H_1(\tau_*) }{qc_G}.  & \qedhere \end{align*} 
 \end{proof}

\begin{proposition}\label{prop:FDPHadDeriv}
    $\Psi^{(\emph{\FDP})}$ (defined in Equation \eqref{eq:PsiFDPDef}) is Hadamard differentiable at $(F_0,F_1)$ tangentially to $C[a,b]_2$ with $$\dot{\Psi}_{(F_0,F_1)}^{(\emph{\FDP})}(H_0,H_1)= q \big( (1+\alpha) H_0(\tau_*) + \alpha H_1(\tau_*) \big)/\tau_*$$ 
    for $(H_0,H_1) \in C[a,b]_2$ being the formula for the Hadamard derivative at $(F_0,F_1)$, and where $\alpha$ is defined in $\eqref{eq:Def_of_alpha}$.
\end{proposition}

\begin{proof}
Define $\Theta : D[a,b]_2 \rightarrow \real^2$ 
as
$$\Theta (H_0,H_1) = 
\big( \Psi^{(\FPR)}(H_0,H_1) , \mathcal{R}(H_0,H_1)  \big)$$ 
for $(H_0,H_1) \in D[a,b]_2$. It is easy to see that the Hadamard derivative of $\Theta$ at $(F_0,F_1)$ tangentially to $C[a,b]_2$ is given by $\dot{\Theta}_{(F_0,F_1)} (H_0,H_1)= \big( \dot{\Psi}_{(F_0,F_1)}^{(\FPR)}(H_0,H_1), \dot{\mathcal{R}}_{(F_0,F_1)}(H_0,H_1)  \big)$ and by Propositions \ref{prop:FPRHadDeriv} and \ref{prop:RHadDeriv}, $\Theta$ is indeed Hadamard differentiable at $(F_0,F_1)$ tangentially to $C[a,b]_2$. Now let $f : \real^2 \rightarrow \real$ be given by $f(x,y)=x/y$. Note that $\nabla f_{(x,y)} = \big( 1/y , -x/y^2 \big)$ for $y \neq 0$. Note that for any $(H_0,H_1) \in D[a,b]_2$, $\Psi^{(\FDP)}(H_0,H_1) = f( \Theta(H_0,H_1))$. Hence by the chain rule for Hadamard differentiability (Theorem 20.9 in \cite{VDV}) we get that if $\mathcal{R}(F_0,F_1) \neq 0$, then  $\Psi^{(\FDP)}$ is Hadamard differentiable at $(F_0,F_1)$ tangentially to $C[a,b]_2$ with Hadamard derivative given by $ (\nabla f_{\Theta(F_0,F_1)})^\tran \dot{\Theta}_{(F_0,F_1)} (H_0,H_1)$ for $(H_0,H_1) \in C[a,b]_2$. But since $\mathcal{R}(F_0,F_1)=(F_0+F_1)( \mathcal{T}(F_0+F_1) )=G(\tau_*)=\tau_*/q \neq 0$, $\Psi^{(\FDP)}$ is indeed Hadamard differentiable at $(F_0,F_1)$ tangentially to $C[a,b]_2$ with Hadamard derivative given by $ (\nabla f_{\Theta(F_0,F_1)})^\tran \dot{\Theta}_{(F_0,F_1)} (H_0,H_1)$. We compute this value more specifically for $(H_0,H_1) \in C[a,b]_2$ using Propositions \ref{prop:FPRHadDeriv} and \ref{prop:RHadDeriv}: 
$$\begin{aligned} \dot{\Psi}_{(F_0,F_1)}^{(\FDP)} (H_0,H_1 )= &  \frac{\dot{\Psi}^{(\FPR)}_{(F_0,F_1)}(H_0,H_1)}{\mathcal{R}(F_0,F_1)} - \frac{\Psi^{(\FPR)}(F_0,F_1) \dot{\mathcal{R}}_{(F_0,F_1)}(H_0,H_1) }{ ( \mathcal{R}(F_0,F_1))^2 }  \\ 
= &  \frac{q}{\tau_*} \Big( \dot{\Psi}^{(\FPR)}_{(F_0,F_1)}(H_0,H_1) - \frac{q}{\tau_*}  F_0(\tau_*) \dot{\mathcal{R}}_{(F_0,F_1)}(H_0,H_1) \Big)  \\ 
= & \frac{q}{\tau_*} \Big( H_0(\tau_*)+ \frac{\big(F_0'(\tau_*) -F_0(\tau_*)/\tau_* \big) \big(H_0(\tau_*)+H_1(\tau_*) \big)}{c_G} \Big)  \\
= & \frac{q}{\tau_*} \big( (1+\alpha) H_0(\tau_*)+\alpha H_1(\tau_*) \big). \end{aligned}$$
The last step above uses our previous definition of $c_G$ in \eqref{eq:cG_def} and the definition of $\alpha$ in \eqref{eq:Def_of_alpha}.
%\frac{q \Big( H_0(\tau_*)+F_0'(\tau_*)  \big(H_0(\tau_*)+H_1(\tau_*) \big)/c_G \Big) }{\tau_*} - \frac{q^2 F_0(\tau_*) \big(H_0(\tau_*)+H_1(\tau_*) \big)   }{ \tau_*^2 q c_G }  \\ = & \frac{q}{\tau_*} \Bigg( 1+ \frac{1}{c_G} \Big( F_0'(\tau_*) - \frac{F_0(\tau_*)}{\tau_*} \Big) \Bigg) H_0( \tau_*) + \frac{q}{c_G \tau_*} \Big( F_0'(\tau_*) - \frac{F_0(\tau_*)}{\tau_*} \Big) H_1 (\tau_*) \\ = & \frac{q}{\tau_*} \Big( (1+\alpha) H_0(\tau_*)+\alpha H_1(\tau_*) \Big).\end{aligned}$$
\end{proof}

%\bibliographystyle{imsart-nameyear_old.bst}
%\bibliography{FDP_CLT_supp}

\end{document}